\documentclass[12pt,a4paper]{book}
\usepackage{sectsty}
\chapternumberfont{\vspace*{-4cm}\Large\centering}
\chaptertitlefont{\Large\centering\vs{-.4}}
\usepackage{fancyhdr}
\usepackage{titlesec}
\usepackage{titletoc}
\usepackage{graphicx}
\usepackage{amsmath}
\usepackage{amstext}
\usepackage{amsbsy}
\usepackage{amssymb}
\usepackage{theorem}
\usepackage{times}
\usepackage[T1]{fontenc}
\usepackage{eucal}
\usepackage{mathrsfs}
\usepackage{url}

\DeclareGraphicsExtensions{.eps,.ps}

\newcommand{\citealt}[1]{\citet{#1}}
\newcommand{\citet}[1]{\cite{#1}}
\newcommand{\citep}[1]{\cite{#1}}

\newcommand{\clint}[2]{[#1,#2]}
\newcommand{\beq}{\begin{equation}}
\newcommand{\eeq}{\end{equation}}

\newlength{\minipagewidth}
\newcommand{\bookbox}[1]{\small
\par\medskip\noindent
\framebox[\columnwidth]{
\begin{minipage}{\minipagewidth} {#1} \end{minipage} } \par\medskip }

\newcommand{\dsone}{\textrm{1\kern-0.25emI}} 

\newcommand{\ov}[1]{\overline{#1}}
\newcommand{\wt}[1]{\widetilde{#1}}

\newcommand{\B}[1]{\mathbb{#1}}

\DeclareMathOperator{\Tr}{Tr}

\DeclareMathOperator*{\ess}{ess}

\DeclareMathOperator{\Var}{\mathbb{V}ar}

\DeclareMathOperator{\sign}{sign}
\DeclareMathOperator{\Span}{span}

\numberwithin{equation}{section}
\newtheorem{thm}{Theorem} 
\newtheorem{prop}[thm]{Proposition}

\newtheorem{lemma}[thm]{Lemma}

{\theorembodyfont{\rm} 
\newtheorem{dfn}{Definition}[section]
\newtheorem{remark}{Remark}}
\theoremheaderfont{\sc}
\newenvironment{proof}{{\sc Proof.}}{\ $\square$}
\titleformat{\section}[block]{\sc\center}{\thetitle.}{5pt}{}[]
\titlespacing{\section}{0pt}{*4.5}{*3}
\titleformat{\subsection}[runin]{\sc}{\thetitle.}{5pt}{}[.]
\titlespacing{\subsection}{0pt}{*3}{*2}
\titleformat{\subsubsection}[runin]{\it}{\thetitle.}{5pt}{}[.]
\titlespacing{\subsubsection}{0pt}{*2}{*2}

\newcommand{\myeq}[1]{{\rm (\ref{#1}, page \pageref{#1})}}
\allowdisplaybreaks

\begin{document}
\renewcommand{\sectionmark}[1]{\markboth{\thesection\ #1}{}}
\renewcommand{\subsectionmark}[1]{\markright{\thesubsection\ #1}}
\fancyhf{}
\fancyhead[RE]{\small\sc\nouppercase{\leftmark}}
\fancyhead[LO]{\small\sc\nouppercase{\rightmark}}
\fancyhead[LE,RO]{\thepage}
\fancyfoot[RO,LE]{\small\sc Olivier Catoni $\Rightarrow$ Jean-Yves Audibert}
\fancyfoot[LO,RE]{\small\sc\today}
\renewcommand{\footruleskip}{1pt}
\renewcommand{\footrulewidth}{0.4pt}
\newcommand{\mypoint}{\makebox[1ex][r]{.\:\hspace*{1ex}}}
\addtolength{\footskip}{11pt}
\pagestyle{plain}


\newcommand{\oo}{{\mathbb 1}}

\def\bmul{\begin{multline*}}
\def\begar{$$\begin{array}{lll}}
\def\endar{\end{array}$$}
\def\begarc{$}
\def\endarc{$ }
\def\bigbegar{\begin{eqnarray*}} 
\def\bigendar{\end{eqnarray*}}
\def\lbegar{$$\left\{ \begin{array}{lll}}
\def\rendar{\end{array} \right.$$}
\def\rendarp{\end{array} \right..$$}
\def\beglab{\begin{equation} \label}
\def\endlab{\end{equation}}
\newcommand\beglabc[1]{\begin{equation*} }
\def\endlabc{\end{equation*}}
\def\lbegarlab{\begin{equation} \left\{ \begin{array}{lll} \label}
\def\rendarlab{\end{array} \right. \end{equation}}
\def\rendarplab{\end{array} \right.. \end{equation}}

\newcommand\und[2]{\underset{#2}{#1}\;}
\newcommand\undc[2]{{#1}_{#2}\;}

\newcommand\wrt{\text{w.r.t.}} 

\newcommand\cst{\text{Cst}\,}
\newcommand\dsV{{\mathbb V}} 
\newcommand\eqdef{=}

\newcommand\vs[1]{\vspace*{#1cm}}
\newcommand\hs[1]{\hspace*{#1cm}}

\newcommand{\Np}{\N^+}

\newcommand{\AAalg}{\text{AA}}

\newcommand\integ[1]{\left\lfloor{#1}\right\rfloor}
\newcommand\integup[1]{\left\lceil{#1}\right\rceil}

\newcommand{\oX}{\overline{X}}
\newcommand{\oR}{\overline{R}}

\newcommand\A{\mathcal{A}}
\newcommand\cB{\mathcal{B}}
\newcommand\cC{\Theta}
\newcommand\calC{\mathcal{C}}
\newcommand\D{\mathcal{D}}
\newcommand\cD{\mathcal{D}}
\newcommand\E{\mathbb{E}}
\newcommand\cE{\mathcal{E}}
\newcommand\calE{\mathcal{E}}
\newcommand\cF{\mathcal{F}}
\newcommand\G{\mathcal{G}}
\newcommand\cH{\mathcal{H}}
\newcommand\I{\mathcal{I}}
\newcommand\J{\mathcal{J}}
\newcommand\K{\mathcal{K}}
\newcommand\cL{\mathcal{L}}
\newcommand\M{\mathcal{M}}
\newcommand\N{\mathbb{N}}
\newcommand\cN{\mathcal{N}}
\newcommand\cP{{\cal P}}
\newcommand\calP{{\cal P}}
\renewcommand\P{\mathbb{P}}
\newcommand\R{\mathbb{R}}
\newcommand\calR{\mathcal{R}}
\renewcommand\S{\mathcal{S}}
\newcommand\cS{\mathcal{S}}
\newcommand\cR{\mathcal{R}}
\newcommand\W{\mathcal{W}}
\newcommand\X{\mathcal{X}}
\newcommand\Y{\mathcal{Y}}
\newcommand\Z{\mathcal{Z}}

\newcommand{\tg}{\tilde{g}}
\newcommand{\tG}{\tilde{G}}

\newcommand\jyem{\em}

\newcommand{\Gmax}{G_{\textnormal{\text{max}}}}
\newcommand{\its}{(i_1,\dots,i_n)}
\newcommand{\size}{\mathcal S}
\newcommand{\tS}{s}

\newcommand\hdelta{\hat{\delta}}
\newcommand\hg{\hat{g}}
\newcommand\hmu{\hat{\mu}}
\newcommand\hth{\hat{\theta}}
\newcommand\tth{\tilde{\theta}}
\newcommand\hlam{\hat{\lam}}
\newcommand\lamerm{\hlam^{\textnormal{(erm)}}}

\newcommand\tf{\tilde{f}}
\newcommand\tlam{\tilde{\lam}}
\newcommand\tmu{\tilde{\mu}}

\newcommand\be{\beta}
\newcommand\ga{\gamma}
\newcommand\kap{\kappa}
\newcommand\lam{\lambda}

\newcommand\sint{{\textstyle \int}}
\newcommand\inth{\sint \pi(d\theta)}
\newcommand\inthj{\sint \pij(d\theta)}
\newcommand\inthp{\sint \pi(d\theta')}

\newcommand\gas{\ga^*}

\newcommand\hpi{\hat{\pi}}
\newcommand\pis{\pi^*}
\newcommand\tpi{\tilde{\pi}}

\newcommand\ovth{\ov{\theta}}
\newcommand\wth{\wt{\theta}}
\newcommand\wthj{\wt{\theta}_j}

\newcommand\eps{\varepsilon}
\newcommand\ve{\varepsilon}

\newcommand\logeps{\log(\eps^{-1})}
\newcommand\leps{\log^2(\eps^{-1})}

\newcommand{\bi}{\begin{itemize}}
\newcommand{\ei}{\end{itemize}}

\newcommand\ovR{\ov{R}}
\newcommand\lbeta{\lfloor \beta \rfloor} 

\newcommand\hhpi{\hat{\hat{\pi}}}
\newcommand\wwth{\wt{\wt{\theta}}}
\newcommand\pia{\pi^{(1)}}
\newcommand\pib{\pi^{(2)}}
\newcommand\pij{\pi^{(j)}}
\newcommand\tpia{\tilde{\pi}^{(1)}}
\newcommand\tpib{\tilde{\pi}^{(2)}}
\newcommand\tpij{\tilde{\pi}^{(j)}}
\newcommand\wtha{\wt{\theta}_1}
\newcommand\wthb{\wt{\theta}_2}
\newcommand\wths{\wt{\theta}_s}
\newcommand\wtht{\wt{\theta}_t}

\newcommand\cmin{c_{\min}}
\newcommand\cmax{c_{\max}}

\newcommand\Pn{P^{\otimes n}}

\newcommand\ra{\rightarrow}
\newcommand\Lra{\Leftrightarrow}
\newcommand\lra{\longrightarrow}

\newcommand\hatt{\hat{t}}
\newcommand\argmax{\textnormal{argmax}}
\newcommand\argmin{\textnormal{argmin}}

\newcommand\diag{\textnormal{Diag}}

\newcommand\cdd{\mathcal{D}'}

\newcommand\Fg{\cF^\#}

\newcommand\vp{\vec{g}}
\renewcommand\th{\theta}

\newcommand\vsp{\vspace{1cm}}
\newcommand\lhs{\text{l.h.s.}}
\newcommand\rhs{\text{r.h.s.}}

\newcommand\expe[2]{\undc{\E}{#1\sim#2}}
\newcommand\expec[2]{\undc{\E}{#1\sim#2}}
\newcommand\expecc[2]{\E_{#1}}
\newcommand\expecd[2]{\E_{#2}}

\newcommand{\bl}{\bar{\ell}_t}

\newcommand\hC{\hat{C}}
\newcommand\heta{\hat{\eta}}
\newcommand\hf{\hat{f}}
\newcommand\hh{\hat{h}}
\newcommand\hrho{\hat{\rho}}
\newcommand\hrc{\hat{\rho}_{\text{\bf C}}}

\newcommand\px{P_X}
\newcommand\etas{\eta^*}

\newcommand\br{\bar{r}}
\newcommand\chr{\check{r}}
\newcommand\bR{\bar{R}}
\newcommand\bV{\bar{V}}

\newcommand\lan{\langle}
\newcommand\ran{\rangle}

\newcommand\logepsg{\log(|\Cg|\eps^{-1})}

\newcommand\Pemp{\hat{\P}}

\newcommand\fracl[2]{{(#1)}/{#2}}
\newcommand\fracc[2]{{#1}/{#2}}
\newcommand\fracr[2]{{#1}/{(#2)}}
\newcommand\fracb[2]{{(#1)}/{(#2)}}

\newcommand\hfproj{\hf^{\textnormal{(proj)}}}
\newcommand\thproj{\hat{\th}^{\textnormal{(proj)}}}

\newcommand\hfols{\hg^{\textnormal{(ols)}}}
\newcommand\thfols{\tilde{f}^{\textnormal{(ols)}}}
\newcommand\thols{\hat{\th}^{\textnormal{(ols)}}}
\newcommand\therm{\hat{\th}^{\textnormal{(erm)}}}
\newcommand\hferm{\hf^{\textnormal{(erm)}}}
\newcommand\hgerm{\hg^{\textnormal{(erm)}}}
\newcommand\zols{\zeta^{\textnormal{(ols)}}}
\newcommand\hgstar{\hg^{\textnormal{(star)}}}

\newcommand\thrid{\tilde{\th}} 
\newcommand\frid{\tilde{g}} 
\newcommand\greg{g^{\textnormal{(reg)}}}
\newcommand\thrlam{\hth^{\textnormal{(ridge)}}} 
\newcommand\hfrlam{\hg^{\textnormal{(ridge)}}} 
\newcommand\thllam{\hth^{\textnormal{(lasso)}}} 
\newcommand\hfllam{\hf^{\textnormal{(lasso)}}} 

\newcommand\ggroup{g^{\textnormal{(group)}}}
\newcommand\hggroup{\hg^{\textnormal{(group)}}}

\newcommand\flin{f^*_{\textnormal{lin}}}
\newcommand\thlin{\th^{\textnormal{(lin)}}}
\newcommand\Flin{\mathcal{F}_{\textnormal{lin}}}

\newcommand\demi{\frac{1}{2}}
\newcommand\demic{\fracc{1}{2}}

\newcommand\substa[2]{\substack{#1\\#2}}
\newcommand\substac[2]{{#1\,;\,#2}}
\newcommand\tpsi{\tilde{\psi}}
\newcommand\tzeta{\tilde{\zeta}}
\newcommand\ta{\tilde{a}}
\newcommand\chis{\chi_\sigma}
\newcommand\tchi{\tilde{\chi}}
\newcommand\tchis{\tchi_\sigma}
\newcommand\psis{\psi_\sigma}

\newcommand\loc{{\textnormal{loc}}}
\newcommand\dsG{\mathbb {G}}

\newcommand\tA{\tilde{A}}
\newcommand\tL{\tilde{L}}

\newcommand\hL{\hat{L}}
\newcommand\hcE{\hat{\cE}}
\newcommand\hDe{\hat{\cE}}
\newcommand\cEb{\cE^{\sharp}}
\newcommand\La{L^{\flat}}
\newcommand\cEa{\cE^{\flat}}
\newcommand\Lb{L^{\sharp}}

\newcommand\Pa{P^{\flat}}
\newcommand\Pb{P^{\sharp}}

\newcommand\ela{\tilde{\ell}}
\newcommand\hpig{\hpi}

\newcommand\sigmb{\phi}
\newcommand{\logdeps}{\log(4d\eps^{-1})}
\newcommand{\logddeps}{\log(2d^2\eps^{-1})}

\newcommand{\piobs}{\bar{\pi}}
\newcommand\pijj{[\pi^{(j)}]_j}

\newcommand\ppac{with $\Pdn$-probability at least $1-\be$}
\newcommand\Ppac{With $\Pdn$-probability at least $1-\be$}
\newcommand\pac{with $\Pun$-probability at least $1-\be$}
\newcommand\Pac{With $\Pun$-probability at least $1-\be$}

\newcommand\Esn{{\E'}^n}
\newcommand\Pun{\P^n}
\newcommand\Pdn{\P^{2n}}

\newcommand\kxi{K\left(\frac{X_i-x}{h}\right)}
\newcommand\lpl{LP(l) }
\newcommand\barO{\bar{B}}
\newcommand\xih{\left(\frac{X_i-x}{h}\right)}
\newcommand\xh{\left(\frac{x'-x}{h}\right)}
\newcommand\calPS{\calP_\Sigma}
\newcommand\calFXY{{\G(\X;\Y)}}
\newcommand\calPSI{\calP_{\Sigma,\infty}}
\newcommand\calCS{\calC_\Sigma}

\newcommand\pbms{{\bf (MS)}}
\newcommand\pbc{{\bf (C)}}
\newcommand\pbl{{\bf (L)}}
\newcommand\gms{g^*_{\text{\bf MS}}}
\newcommand\gc{g^*_{\text{\bf C}}}
\newcommand\gl{g^*_{\text{\bf L}}}
\newcommand\rhoc{\rho^*_{\text{\bf C}}}



%
%
%
%
%
%

\renewcommand*\contentsname{$ $ \vs{-3}\\
\LARGE 
PAC-Bayesian aggregation \vspace{.2cm}\\
and multi-armed bandits
\vs{1}\\
{\large \textnormal{(Habilitation thesis)}}
\vs{1}\\
{\large\bf Jean-Yves AUDIBERT}
\vs{1}\\
\begin{flushleft}
{\large $ $} \vs{-1.8}
\end{flushleft}
}


\tableofcontents

\chapter{Introduction}

This document presents the research I have undertaken since the beginning of my PhD thesis.
The Laboratoire de Probabilités et Modèles Aléatoires of Université Paris 6 hosted my PhD (2001-2004). I was then recruited in the Centre d'Enseignement et de Recherche en Traitement de l'Information et Signal de l'Ecole Nationale des Ponts et Chaussées, which is now a common research laboratory with the Centre Scientifique et Technique du Bâtiment and part of 
the Laboratoire d'informatique Gaspard Monge de l'Université Paris Est.
Besides, since 2007, part of my research has been done within the Willow team of the Laboratoire d'Informatique de l'École Normale Superiéure.

My main research directions are statistical learning theory and machine learning techniques for computer vision. Machine learning is a research field positioned between statistics, computer science and applied mathematics. Its goal is to bring out theories and algorithms
to better understand and deal with complex systems for which no simple, accurate and easy-to-use model exists. It has a considerable impact on a wide variety of scientific domains, including
text analysis and indexation, financial market analysis, search engines, bioinformatics, 
speech recognition, robotics, industrial engineering...
The development of new sensors to acquire data, the increasing capacity of storage and computational power of computers have brought new perspectives to understand more and more complex systems from observations. In particular, Machine learning techniques are used in computer vision tasks that are unsolvable using classical methods (object detection, handwriting recognition, image segmentation and annotation).

The core problem in statistical learning can be formalized in the following way.
We observe $n$ input-output (or object-label) pairs: $Z_1=(X_1,Y_1),\dots,Z_n=(X_n,Y_n)$. A new input $X$ comes.
The goal is to predict its associated output $Y$.
The input is usually high dimensional and highly structured (such as a digital image). The output is simple: it is typically a real number or an element in a finite set (for instance, 'yes' or 'no' in the case of the detection of a specific object in the digital image). The usual probabilistic modelling is that the observed data (or training set) and the input-output pair $Z=(X,Y)$ are independent and identically distributed random variables coming from some unknown distribution $P$, and that, for various possible reasons, the output is not necessarily a deterministic function of the input.

The lack of quality of a prediction $y'$ when $y$ is the true output is measured by 
its loss, denoted $\ell(y,y')$.
Typical loss functions are the 0/1 loss: $\ell(y,y')=\oo_{y\neq y'}$ (the loss is one if  
and only if the prediction differs from the true output) and the 
square loss: $\ell(y,y') = (y-y')^2$. The latter loss is more appropriate than the 0/1 loss 
when the output space is the real line, a small difference between the prediction and the true output generating a small loss.
The target of learning is to infer from the training set a function $g$ from the input space to 
the output space 
having a low risk, also called expected loss or generalization error:
  $$R(g)=\E_{(X,Y)\sim P} \ell\big(Y,g(X)\big).$$

Statistical learning theory aims at answering the following questions.
What are the conditions for (asymptotic) consistency of the learning scheme?
What can we learn from a finite sample of observations? Under which circumstances, can we expect the risk to be close to the risk of the best prediction function, that is the one 
we could have proposed had we a full knowledge 
of the probability distribution $P$ underlying the observations?
How accurate is the prediction built on the training set? For instance, how low is its risk?
What kind of guarantees can we ensure? Both theoretical and empirical (i.e., computable from the observed data) upper bounds on the risk or the excess risk are of interest. 
Can we understand/explain the success of some prediction schemes?
Besides, we also expect that a new theoretical analysis leads to the design of new prediction methods.

This document details my contributions to these issues, and specifically to:
\bi
\item the PAC-Bayesian analysis of statistical learning,
\item the three aggregation problems: given $d$ functions, how to predict as well as
  \bi
  \item the best of these $d$ functions (model selection type aggregation),
  \item the best convex combination of these $d$ functions,
  \item the best linear combination of these $d$ functions,
  \ei
\item the multi-armed bandit problems,
\ei
Being in computer science departments where image processing and computer vision are core research directions leads me to address a wide variety of topics in which machine learning plays a key role. It includes object recognition, content-based image retrieval, image segmentation and image annotation and vanishing point detection. This document will not detail my contributions on these topics. My related publications can be found on my webpage.

\chapter{The PAC-Bayesian analysis of statistical learning} \label{chap:PB}

\section{Introduction}

The natural target of learning is to predict as well as 
if we had known the distribution generating the input-output pairs.
In other words, we want to infer from the training set 
$Z_1^n =\{(X_1,Y_1),\dots,(X_n,Y_n)\}$
a prediction function $\hg$
whose risk is close to the risk of the Bayes predictor $g^* = \argmin_{g} R(g)$,
where the minimum is taken among all functions $g:\X\ra\Y$ (such that $\ell\big(Y,g(X)\big)$ is integrable). The goal is therefore to propose a good estimator $\hg$ of $g^*$,
where the quality of the estimator is not in terms of the functional proximity 
of the prediction functions but in terms of their risk similarity.

Since the distribution $P$ of the input-output pair is unknown, the risk is not observed,
and numerous core learning procedures have recourse to
its empirical counterpart: 
   $$r(g) = \frac1n \sum_{i=1}^n \ell(Y_i,g(X_i)),$$
either by minimizing it on a restricted class of functions, or almost equivalently
by minimizing a linear combination of this empirical risk and a penalty (or regularization) term whose role is to favor ``simple'' functions. The term ``simple'' typically refers
to some a priori of the statistician, and is often linked to either some smoothness 
property or some sparsity of the prediction function.
The traditional approach to statistical learning theory 
relies on the study of $R(\hg)-r(\hg)$.

In the PAC-Bayesian approach, randomized prediction schemes are considered.
Let $\M$ denote the set of distributions on the set $\G(\X;\Y)$ of functions from the input space to the output space. 
A distribution $\hrho$ in $\M$ is chosen from the data, and the quantity of interest is
$R(g)$, where $g$ is drawn from the distribution $\hrho$.
This risk is thus doubly stochastic: it depends on the realization
of the training set (which is a realization of the $n$-fold product distribution $\Pn$ of $P$)
and on the realization of the (posterior) distribution $\hrho$.

Basically, one can argue that the difference between the approaches seems minor: 
the understanding of $\E_{g\sim\hrho} R(g)$ for any distribution
$\hrho$ implies the understanding of $R(\hg)$ (simply by considering 
the Dirac distribution at $\hg$), and that the converse is also true
(to the extent that if $R(\hg) \le B(\hg)$ holds for any estimator $\hg$ and some real-valued function $B$, then $\E_{g\sim\hrho} R(g) \le \E_{g\sim\hrho} B(g)$ also holds
for any posterior distribution $\hrho$).

The main difference lies rather in the very starting point of
the PAC-Bayesian analysis. To detail it, let me
introduce a distribution $\pi\in\M$, that is non-random (as opposed to $\hrho$, which depends on the sample).
The central argument is (based on) the following property of the Kullback-Leibler (KL) divergence:
for any bounded function $h:\G(\X;\Y)\ra\R$,
we have
  \begin{equation} \label{eq:0}
  \sup_{\rho \in \M} \big\{ \E_{g\sim \rho} h(g) -K(\rho,\pi) \big\} 
    = \log \E_{g\sim \pi} e^{h(g)},
  \end{equation}
where $e$ denotes the exponential number, and $K(\rho,\pi)$ is the KL divergence between
the distributions $\rho$ and~$\pi$: $K(\rho,\pi)=\E_{g\sim\rho} \log\big(\frac{\rho}\pi(g)\big)$ if $\rho$
admits a density with respect to $\pi$, denoted $\frac{\rho}\pi$, and $K(\rho,\pi)=+\infty$ otherwise\footnote{See
Appendix \ref{app:kl} for a summary of the properties of the KL divergence.}.
To control the difference $\E_{g\sim\hrho} R(g) - \E_{g\sim\hrho} r(g)$,
putting aside integrability issues, one essentially uses: for any $\lam>0$,
  \begin{align} 
  \E_{Z_1^n \sim\Pn} e^{\lam [\E_{g\sim\hrho} R(g) - \E_{g\sim\hrho} r(g)] - K(\hrho,\pi)}
  & \le \E_{Z_1^n\sim\Pn} e^{\sup_{\rho\in\M} \lam [\E_{g\sim\rho} R(g) - \E_{g\sim\rho} r(g)] 
    - K(\rho,\pi)} \notag \\
  & = \E_{Z_1^n\sim\Pn} \E_{g\sim\pi} e^{\lam [R(g) - r(g)]} \notag \\
  & = \E_{g\sim\pi} \E_{Z_1^n\sim\Pn} e^{\lam [R(g) - r(g)]} \notag \\
  & = \E_{g\sim\pi} \bigg(\E_{(X,Y)\sim P} e^{\frac{\lam}n [R(g) - \ell(Y,g(X))]} \bigg)^n.
  \end{align}
A first consequence is that PAC-Bayes bounds are not (directly) useful
for posterior distributions with $K(\hrho,\pi)=+\infty$: this is in particular the case
when $\hrho$ is a Dirac distribution and $\pi$ assigns no probability mass to single functions.
So classical results of the standard approach does not derive from the PAC Bayesian approach.
On the other hand, the apparition of the KL term shows that the PAC-Bayesian analysis fundamentally differs from the simple analysis given in the previous paragraph. 

To illustrate this last point, consider the case of a prior distribution putting mass on a finite set $\G\subset\G(\X;\Y)$ of functions. For simplicity, consider bounded losses, say $0\le \ell(y,y')\le 1$ for any $y,y'\in\Y$.
By using Hoeffding's inequality and a weighted union bound, 
one gets that for any $\eps>0$, with probability at least $1-\eps$, we have
for any $g\in\G$
  $$
  R(g)-r(g) \le \sqrt{\frac{\log(\pi^{-1}(g)\eps^{-1})}{2n}},
  $$
hence for any distribution $\rho$ such that $\rho(\G)=1$,
  \begin{align}
  \E_{g\sim\rho} R(g)-\E_{g\sim\rho} r(g) & \le \E_{g\sim\rho} \sqrt{\frac{\log(\pi^{-1}(g)\eps^{-1})}{2n}} \notag\\
  & \le \sqrt{\frac{K(\rho,\pi)+H(\rho)+\log(\eps^{-1})}{2n}},\label{eq:basic}
  \end{align}
where the second inequality uses Jensen's inequality and Shannon's entropy: 
$H(\rho)=-\sum_{g\in\G} \rho(g) \log\rho(g).$
This is to be compared to the first PAC-Bayesian bound from the pioneering work of McAllester \cite{McA99}, which
states that with probability at least $1-\eps$, for any distribution $\rho\in\M$, we have
  \begin{align*}
  \E_{g\sim\rho} R(g)-\E_{g\sim\rho} r(g) \le \sqrt{\frac{K(\rho,\pi)+\log(n)+2+\log(\eps^{-1})}{2n-1}}.
  \end{align*}
The main difference is that the Shannon entropy has been replaced with a $\log n$ term.
In fact, the latter bound is not restricted to prior distributions putting mass on a finite set of
functions: it is valid for any distribution $\pi$. On the contrary, the basic argument leading to \eqref{eq:basic} does not extend to continuous set of functions because of the Shannon's entropy term (for $\rho$ putting masses on a continuous set of functions, this term diverges).

The previous discussion has shown the originality of the PAC-Bayesian analysis.
However it does not clearly demonstrate its usefulness.
Several works in the last decade have shown that the approach is indeed useful,
and that PAC-Bayesian bounds lead to tight bounds, which are often representative of the risk behaviour even for relatively small training sets
(see e.g. \cite{Lan03,McA03b,Lac09} for margin-based bounds from Gaussian prior distributions,
\cite{Lac07} for an Adaboost setting, that is majority vote of weak learners, \cite{Sel08} in a clustering setting, \cite[Chap.2]{Aud04},\cite{LavMar07} for compression schemes,
\cite{Dal08,Dal09} for PAC bounds with sparsity-inducing prior distributions).

My contributions to the PAC-Bayesian approach are
the use of relative PAC-Bayesian bounds to design estimators with minimax rates
(Section \ref{sec:rel}), 
the combination of the PAC-Bayesian argument with metric and (generic) chaining arguments
(Section \ref{sec:combin}), 
the use of PAC-Bayesian bounds to propose new estimators and minimax bounds under weak assumptions for the aggregation problems (Chapter \ref{chap:agg}).
Before detailing them, I give in the next section a global picture of PAC-Bayesian bounds,
with a particular emphasis on the relations between the
different works since they have not been underlined so far in the literature.

\section{PAC-Bayesian bounds} \label{sec:PBB}


We consider that the losses are between $0$ and $1$, unless otherwise stated. 
The symbol $C$ will be used to denote a constant that may differ from line to line.
The bounds stated here are the original ones, possibly up to minor improvements.
Most of them rely on a different use of the duality formula \eqref{eq:0}
and the Markov inequality, which allows to prove a Probably Approximatively Correct (PAC) bound
from the control of the Laplace transform of an appropriate random variable: precisely, if a
real-valued random variable $V$ is such that $\E e^V \le 1$, then 
for any $\eps>0$, with probability at least $1-\eps$,
  $V\le \log(\eps^{-1}).$

\subsection{Non localized PAC-Bayesian bounds}

McAllester's first bound states that for any $\eps>0$, with probability at least $1-\eps$, for any $\rho\in\M$, we have
  \begin{align} \label{eq:McA}\tag{McA}
  \E_{g\sim\rho} R(g)-\E_{g\sim\rho} r(g) \le \sqrt{\frac{K(\rho,\pi)+\log(2n)+\log(\eps^{-1})}{2n-1}}.
  \end{align}
In \cite{LSM01,See02}, Seeger has proposed a simplified proof and improved the bound when the losses take only two values $0$ or $1$ (classification losses). The result is that with probability at least $1-\eps$, for any $\rho\in\M$, we have
  \begin{align} \label{eq:S}\tag{S}
  K(\E_{g\sim\rho} r(g),\E_{g\sim\rho} R(g)) \le \frac{K(\rho,\pi)+\log(2\sqrt{n}\eps^{-1})}{n}.
  \end{align}
where, with a slight abuse of notation, $K(\E_{g\sim\rho} r(g),\E_{g\sim\rho} R(g))$ denotes
the KL divergence between the Bernoulli distributions of respective parameters $\E_{g\sim\rho} r(g)$
and $\E_{g\sim\rho} R(g)$. The concise proofs of \eqref{eq:McA} and \eqref{eq:S} are given in Appendices \ref{app:McA} and \ref{app:S}.

Since we have $\E_{g\sim\rho} R(g) - \E_{g\sim\rho} r(g) \le \sqrt{K\big( \E_{g\sim\rho} r(g) , \E_{g\sim\rho} R(g)\big)}$ (Pinsker's inequality), \eqref{eq:S} implies \eqref{eq:McA}.
Besides, when $\E_{g\sim\rho} R(g)$ is small, \eqref{eq:S} provides a much better bound than \eqref{eq:McA} 
since, from a cumbersome study of the function $t\mapsto K(\E_{g\sim\rho} r(g),\E_{g\sim\rho} r(g)+t)$, \eqref{eq:S}
implies 
  \begin{align} \label{eq:dS}\tag{S'}
  \big| \E_{g\sim\rho} R(g) - \E_{g\sim\rho} r(g) \big| \le \sqrt{\frac{2\E_{g\sim\rho} r(g) [1-\E_{g\sim\rho} r(g)] \K}{n}}+\frac{4\K}{3n},
  \end{align} 
  with $\K=K(\rho,\pi) + \log(2\sqrt{n}\eps^{-1}).$
In particular, when the empirical risk of the randomized estimator is zero, this last bound is of $1/n$ order, while
\eqref{eq:McA} only gives a $1/\sqrt{n}$ order.

Still in the classification setting, Catoni \cite{Cat03b} proposed a different bound:
for any $\eps>0$ and $\lam>0$ with $\frac{\lam}{n}\Psi(\frac{\lam}n)<1$, with probability at least $1-\eps$, for any $\rho\in\M$,
  \begin{align} \label{eq:C03}\tag{C1}
  \E_{g\sim\rho} R(g) \le \frac{\E_{g\sim\rho} r(g)}{1-\frac{\lam}{n}\Psi(\frac{\lam}n)} + \frac{K(\rho,\pi)+ \log(\eps^{-1})}{\lam [1-\frac{\lam}{n}\Psi(\frac{\lam}n)]},
  \end{align}
where $$\Psi(t) = \frac{e^t-1-t}{t^2}.$$
Since typical values of $\lam$ (the ones which minimizes the previous right-hand side) are in $[C\sqrt{n};Cn]$ and since $\Psi(\lam/n) \approx 1/2$ for $\lam/n$ close to $0$, we roughly have
  $$
  \E_{g\sim\rho} R(g) \lessapprox \E_{g\sim\rho} r(g) + \frac{\lam}{2n} \E_{g\sim\rho} r(g) + \frac{K(\rho,\pi)+ \log(\eps^{-1})}{\lam},
  $$
which gives by choosing $\lam$ optimally\footnote{Technically speaking, we are not allowed to choose $\lam$ depending on $\rho$, but using a union bound argument, the argument can be made rigorous at the price that the $\logeps$ term becomes   
$\log(C\log(C n) \eps^{-1})$.}
  \begin{align} \label{eq:C1p}\tag{C1'}
  \E_{g\sim\rho} R(g) \lessapprox \E_{g\sim\rho} r(g) + \sqrt{2\E_{g\sim\rho} r(g) \frac{K(\rho,\pi)+ \log(\eps^{-1})}{n}}.
  \end{align}
My PhD thesis used in variant ways the following Bernstein's type PAC-Bayesian bound, which is a direct extension of 
the argument giving \eqref{eq:C03}: for any $\lam>0$, with probability at least $1-\eps$, for any $\rho\in\M$,
  \begin{align} 
  \E_{g\sim\rho} R(g) \le \E_{g\sim\rho} r(g) + \frac{\lam}{n}\Psi\bigg(\frac{\lam}n\bigg) 
    \E_{g\sim\rho} & \Var_{Z} \ell(Y,g(X)) \notag\\
  & + \frac{K(\rho,\pi)+ \log(\eps^{-1})}{\lam}. \label{eq:A}\tag{A}
  \end{align}
The basic PAC-Bayesian bound used in Zhang's works \cite{Zha06,Zha07} does not require any boundedness assumption of the loss function and states that for any $\lam>0$,
with probability at least $1-\eps$, for any $\rho\in\M$,
  \begin{align} \label{eq:Z}\tag{Z}
  -\frac{n}\lam \E_{g\sim\rho} \log \E_{Z} e^{-\frac\lam{n} \ell(Y,g(X))} \le \E_{g\sim\rho} r(g) 
      + \frac{K(\rho,\pi)+ \log(\eps^{-1})}{\lam}.
  \end{align} 
Catoni's book \cite{Cat07} concentrates on the classification task.
Instead of using 
  $$
  \log \E e^{- \frac{\lam}{n}\ell(Y,g(X))} \le - \frac{\lam}{n} R(g) + \frac{\lam^2}{n^2} \Psi\Big(\frac{\lam}{n} \Big) R(g),
  $$
which would give \eqref{eq:C03} from \eqref{eq:Z}, Catoni used the equality
  $$
  \log \E e^{- \frac{\lam}{n}\ell(Y,g(X))} = \log \big( 1 - R(g) (1-e^{-\frac\lam{n}})\big),
  $$
and obtain that with probability at least $1-\eps$, for any $\rho\in\M$,
  \begin{align} \label{eq:C07}\tag{C2}
  -\frac{n}\lam \log[1-(1-e^{-\frac{\lam}{n}}) \E_{g\sim\rho} R(g)] \le \E_{g\sim\rho} r(g) 
      + \frac{K(\rho,\pi)+ \log(\eps^{-1})}{\lam}.
  \end{align} 

To compare Seeger's bound with the bounds having the free parameter $\lam$ in the classification framework,  
one needs to apply the same kind of analysis which leads from \eqref{eq:C03} to \eqref{eq:C1p}.
As a result, both \eqref{eq:A} and \eqref{eq:Z} lead to 
  \begin{align} \label{eq:dZ}\tag{Z'}
  \E_{g\sim\rho} R(g) \lessapprox\E_{g\sim\rho} r(g) + \sqrt{2\E_{g\sim\rho} \big( R(g) [1-R(g)] \big) \frac{K(\rho,\pi)+ \log(\eps^{-1})}{n}},
  \end{align} 
\eqref{eq:C03} leads to
  $$\E_{g\sim\rho} R(g) \lessapprox \E_{g\sim\rho} r(g) + \sqrt{2\E_{g\sim\rho} R(g) \frac{K(\rho,\pi)+ \log(\eps^{-1})}{n}},$$
\eqref{eq:S} gives, once more, from studying the function $t\mapsto K(\E_{g\sim\rho} r(g),\E_{g\sim\rho} r(g)+t)$,
$$\E_{g\sim\rho} R(g) \le \E_{g\sim\rho} r(g) + \sqrt{\frac{2\E_{g\sim\rho} R(g) [1-\E_{g\sim\rho} R(g)] \K}{n}}+\frac{2\K}{3n},$$
with $\K=K(\rho,\pi) + \log(2\sqrt{n}\eps^{-1}),$ and finally \eqref{eq:C07} leads to
  $$
  \E_{g\sim\rho} R(g) \lessapprox\E_{g\sim\rho} r(g) + \sqrt{2\big( \E_{g\sim\rho} R(g) [1-\E_{g\sim\rho} R(g)] \big) \frac{K(\rho,\pi)+ \log(\eps^{-1})}{n}}.
  $$
Although we have $\E_{g\sim\rho} R(g) [1-\E_{g\sim\rho} R(g)] \ge \E_{g\sim\rho} R(g) [1-R(g)]$ (from Jensen's inequality),
the two quantities will be of the same order, and also of the order of $\E_{g\sim\rho} R(g)$ 
for the typical posterior distributions, i.e., the ones  
which concentrate on low risk functions. As a consequence, in the classification setting, all these bounds are
similar (even if this similarity has not been exhibited so far in the literature).
 
In fact, the works which lead to \eqref{eq:C03}, \eqref{eq:A}, \eqref{eq:Z}
and \eqref{eq:C07} rather differ in the way these bounds are refined and used.
The main common refinement is the PAC-Bayesian localization,
which can be seen as a way to reduce the complexity term and the influence of the particular choice of 
the prior distribution $\pi$. 
Before detailing the localization idea, let us see how to design an estimator from PAC-Bayesian bounds.

\subsection{From PAC-Bayesian bounds to estimators} \label{sec:bound2est}

The standard way to exploit an upper bound on the risk of any estimators is 
to minimize it in order to get the estimator having the best guarantee in view of the bound.
This will be achievable if the bound is empirical, that is computable from the observations.
Bounds \eqref{eq:McA}, \eqref{eq:dS}, \eqref{eq:C03}, \eqref{eq:A} and \eqref{eq:C07} are of this type 
(unlike \eqref{eq:dZ} for instance).

When minimizing PAC-Bayesian bounds, one gets a posterior distribution corresponding to a randomized estimator.
The minimizer can be written in the following form
  $$
  \pi_h(dg) = \frac{e^{h(g)}}{\E_{g'\sim\pi} e^{h(g')} } \cdot \pi(dg)
  $$
for some appropriate function $h:\G\ra\R$.
This is essentially due to the equality
$\argmin_{\rho\in\M} \big\{ - \E_{g\sim\rho} h(g) + K(\rho,\pi) \big\} = \pi_h$.

Let us now detail the case of McAllester's bound as it is representative of what can be derived 
from the other PAC-Bayesian bounds.
Let 
$B(\rho) = \E_{g\sim\rho} r(g) + \sqrt{\frac{K(\rho,\pi) + \log(4n\eps^{-1})}{2n-1}}.$
McAllester's bound implies that for any distribution $\rho\in\M$, we have $\E_{g\sim\rho} R(g) \le B(\rho).$
From this, one can deduce that there exists 
$\hlam \in [\lam_1,\lam_2]$ s.t. $B(\pi_{-\hlam r}) = \min_{\rho} B(\rho)$
with
  $\lam_1= \sqrt{4(2n-1)\log(4n\eps^{-1})}$ and $\lam_2 = 2 \lam_1 + 4(2n-1)$.
Besides, the parameter $\hlam$ which can be called inverse temperature parameter by analogy with
the Boltzmann distribution in statistical mechanics satisfies 
$$\hlam = \sqrt{4(2n-1)[K(\pi_{-\hlam r},\pi)+\log(4n\eps^{-1})]}$$
and
$\hlam \in \und{\argmin}{\lam>0} \Big\{ -\frac{1}{\lam}\log \E_{g\sim\pi} e^{ -\lam r(g)} 
      + \frac{\lam}{4(2n-1)}+\frac{\log(4n\eps^{-1})}{\lam} \Big\}.$

The posterior distribution is thus a distribution which concentrates on low empirical risk functions, but
is still a bit diffuse since to avoid a high KL complexity term, the optimal parameter $\hlam$ cannot be larger
than $Cn$. The next section shows how to reduce the complexity term by tuning the prior distribution.

\subsection{Localized PAC-Bayesian bounds} \label{sec:loc}

Without prior knowledge, one may want to choose a prior distribution $\pi$ which is rather ``flat''.
Now for a particular choice of posterior distribution $\hrho$, from the equality
  $\E_{Z_1^n} K(\hrho,\pi) = \E_{Z_1^n} K(\hrho,\E_{Z_1^n}[\hrho])+K(\E_{Z_1^n}[\hrho],\pi),$
the prior distribution (recall that it is not allowed to depend on the training set) which minimize the expectation 
of the KL divergence is $\E_{Z_1^n} \hrho$, where the expectation is taken with respect to the training set distribution\footnote{
As noted by Catoni, $\E_{Z_1^n} K(\hrho,\E_{Z_1^n}[\hrho])$ is exactly the
mutual information of the random variable $\hg$ drawn according to the posterior distribution $\hrho$ and the training sample $Z_1^n$.
This makes a nice connexion between the learning rate of a randomized estimator and information theory.}.
Now using such a prior distribution does not lead to empirical bound. To alleviate this issue
and since the typical posterior distributions have the form $\pi_{-\lam r}$ for some $\lam>0$ (as seen in the previous section),
one may consider the prior distribution $\pi_{-\beta R}$ for some $\be >0$, use the expansion
  $$
  K(\rho,\pi_{-\be R})=K(\rho,\pi)+ \beta \E_{g\sim\rho} R(g)+\log\big( \E_{g\sim\pi} e^{-\be R(g)}\big),
  $$ 
and obtain an empirical bound by controlling the last non-observable term by its empirical version.

This leads to the following localized PAC-Bayesian bound which was obtained by Catoni in \cite{Cat03b}:
for any $\eps>0$, $\lam>0$ and $\xi\ge0$ such that $\frac{(1+\xi)\lam}{(1-\xi)n}\Psi(\frac{\lam}n)<1$, with probability at least $1-\eps$, for any $\rho\in\M$, we have
  \begin{align} \label{eq:C3}\tag{C3}
  \E_{g\sim\rho} R(g) \le \frac{\E_{g\sim\rho} r(g)}{1-\frac{(1+\xi)\lam}{(1-\xi)n}\Psi(\frac{\lam}n)} + 
    \frac{K(\rho,\pi_{-\xi\lam r})+ (1+\xi)\log(2\eps^{-1})}{(1-\xi)\lam [1-\frac{(1+\xi)\lam}{(1-\xi)n}\Psi(\frac{\lam}n)]}.
  \end{align}
The parameter $\xi$ characterizes the localization. For $\xi=0$, we recover \eqref{eq:C03} (up to a minor difference on the confidence level).
For $\xi>0$, the KL term is (potentially much) smaller when considering the posterior distribution $\pi_{-\gamma r}$ with $\gamma\ge \xi\lam$.

We use similar ideas in the case of the comparison of the risks of two randomized estimators as we will see in Section \ref{sec:rel}.
Zhang \cite{Zha06,Zha07} localizes by using $\pi_{h}$ with $h(g)=\alpha \log \E_Z e^{-\lam \ell(Y,g(X))}$ instead
of $\pi_{-\be R}$. The argument there is slightly different and does not lead to empirical bounds on
the risk of the randomized estimator with posterior distribution of the form $\pi_{-\lam r}$.
Nevertheless, it was sufficient to prove tight theoretical bounds for this estimator in different contexts: density estimation,
classification and least squares regression.

Ambroladze, Parrado-Hern{\'a}ndez and Shawe-Taylor \cite{Amb06} proposed a different way to reduce the influence
of a ``flat'' prior distribution. Their localization scheme is based on cutting the training set 
into two parts and learn from the first part the prior distribution to be used on the second part of the training set.
Catoni \cite{Cat07} uses $\pi_{-n\log[1+(e^{\fracc{\be}n}-1)R]}$ to obtain tighter localized bounds in the classification setting.
Alquier \cite{Alq06,Alq08} uses $\pi_{-\be R}$ for general unbounded losses with application to regression and density estimation.

\section{Comparison of the risk of two randomized estimators} \label{sec:rel}

\subsection{Relative PAC-Bayesian bounds}

My PhD (its second chapter) used relative bounds which compare the risk of two randomized estimators
to design new (randomized) estimators. 
The rationale behind developing this type of bounds is that the fluctuations
of $R(g_2)-R(g_1)+r(g_1)-r(g_2)$ can be much smaller than the fluctuations of $R(g_2)-r(g_2)$, and this can lead
to significantly tighter bounds.
Technically speaking, relative bounds are deduced from standard bounds by replacing $\G$ by $\G\times\G$, taking the loss $\ell(y,(g_1,g_2)(x))=\ell(y,g_2(x))-\ell(y,g_1(x))$ (with a slight abuse of notation) and by considering product distributions on $\G\times\G$, i.e. $\rho = \rho_1\otimes\rho_2$
with $\rho_1$ and $\rho_2$ distributions on $\G(\X;\Y)$.
This standard argument transforms \eqref{eq:A} into the following assertion holding for losses taking values in $[0,1]$. For any $\lam>0$ and (prior) distributions $\pi_1$ and $\pi_2$ in $\M$, 
with probability at least $1-\eps$, for any $\rho_1\in\M$ and $\rho_2\in\M$,
  \begin{align} 
  \E_{g_2\sim\rho_2} R(g_2) - & \E_{g_1\sim\rho_1} R(g_1) \le \E_{g_2\sim\rho_2} r(g_2) - \E_{g_1\sim\rho_1} r(g_1) \notag\\
    & + \frac{\lam}{n}\Psi\bigg(\frac{\lam}n\bigg) 
    \E_{g_2\sim\rho_2} \E_{g_1\sim\rho_1} \E_{Z} \big([ \ell(Y,g_1(X)) - \ell(Y,g_2(X)) ]^2 \big)\notag\\
  & + \frac{K(\rho_2,\pi_2)+ K(\rho_1,\pi_1)+ \log(\eps^{-1})}{\lam}. \label{eq:relnonloc}
  \end{align} 
Getting empirical relative bounds calls for controlling the variance term. This is achieved 
by plugging the following inequality, which holds with probability at least $1-\eps$, into the previous one
  \begin{align*} 
  \E_{g_2\sim\rho_2} & \E_{g_1\sim\rho_1} \E_{Z} [ \ell(Y,g_1(X)) - \ell(Y,g_2(X)) ]^2 \\
  & \le \bigg(1+\frac{\lam}{2n}\bigg)\E_{g_2\sim\rho_2} \E_{g_1\sim\rho_1} \frac1n\sum_{i=1}^n [ \ell(Y_i,g_1(X_i)) - \ell(Y_i,g_2(X_i)) ]^2 \\
  & \qquad\qquad\qquad + \bigg(1+\frac{\lam}{2n}\bigg)^2\frac{K(\rho_2,\pi_2)+ K(\rho_1,\pi_1)+ \log(\eps^{-1})}{\lam}.  
  \end{align*} 
Now, the localization argument described in Section \ref{sec:loc} no longer works as it would 
change the left-hand side of \eqref{eq:relnonloc} into
$(1+\xi_2)\E_{g_2\sim\rho_2} R(g_2) - (1-\xi_1) \E_{g_1\sim\rho_1} R(g_1)$ for some positive constants $\xi_1$ and $\xi_2$,
and would therefore fail to produce relative bounds.
To solve this issue, I proved the following uniform empirical upper bound on the KL divergence with respect to a localized prior:
for any $\epsilon > 0$ and $0 < \lam \leq 0.19 \, n$, with probability at least $1-2\eps$,
for any $\rho \in \M$, we have
	\begin{align*}
	K\big( \rho , \pi_{-\lam R} \big) & \le 2 K\big( \rho , \pi_{-\lam r} \big) \\
		& \ + 2 \log \E_{g_1\sim\pi_{-\lam r}} e^{\frac{4 \lam^2}{ n} \E_{g_2\sim\rho} \frac1n \sum_{i=1}^n [ \ell(Y_i,g_1(X_i)) - \ell(Y_i,g_2(X_i)) ]^2}
		+ \logeps ,
	\end{align*}
and get the following localized empirical PAC-Bayesian relative bound: for any $\lam>0$ and $0 < \lam_1,\lam_2 \leq 0.19 \, n$, 
with probability at least $1-\eps$,
  \begin{align} 
  \E_{g_2\sim\rho_2} R(g_2) - & \E_{g_1\sim\rho_1} R(g_1) \le \E_{g_2\sim\rho_2} r(g_2) - \E_{g_1\sim\rho_1} r(g_1) \notag\\
  & + a(\lam) \E_{g_2\sim\rho_2} \E_{g_1\sim\rho_1} \frac1n \sum_{i=1}^n [ \ell(Y_i,g_1(X_i)) - \ell(Y_i,g_2(X_i)) ]^2 \notag\\
  & + b(\lam) \bigg[K(\rho_2,\pi_{-\lam_2 r})+ K(\rho_1,\pi_{-\lam_1 r}) + 2\log(6\eps^{-1})\notag\\
  & \qquad\quad + \log \E_{g_1\sim\pi_{-\lam_2 r}} e^{\frac{4 \lam_2^2}{n} \E_{g_2\sim\rho_2} \frac1n \sum_{i=1}^n [ \ell(Y_i,g_1(X_i)) - \ell(Y_i,g_2(X_i)) ]^2}\notag\\
  & \qquad\quad + \log \E_{g_1\sim\pi_{-\lam_1 r}} e^{\frac{4 \lam_1^2}{ n} \E_{g_2\sim\rho_1} \frac1n \sum_{i=1}^n [ \ell(Y_i,g_1(X_i)) - \ell(Y_i,g_2(X_i)) ]^2}
    \bigg]. \label{eq:relloc}
  \end{align} 
with $a(\lam) = \frac{\lam}{n} \Psi(\frac{\lam}{n}) \big( 1 + \frac{ \lam }{ 2n } \big)$
and $b(\lam) = \frac{2}{\lam} \Big[ 1 + \frac{\lam}{n} \Psi(\frac{\lam}{n}) \big( 1 + \frac{ \lam }{ 2n } \big)^2 \Big].$
  
\subsection{From the empirical relative bound to the estimator}

In view of Section \ref{sec:bound2est}, it is natural to concentrate our effort on Gibbs estimators of the form $\pi_{-\lam r}$ for $\lam>0$.
Introduce for any $0 \le j \le \log n$ and $\eps>0$, 
	\begin{align*}
	\lam_j & = 0.19 \sqrt{n} e^{\frac{j}{2}}\\
	\calC( j ) & = \log \E_{g_1\sim\pi_{-\lam_j r}} e^{\frac{4 \lam_j^2}{ n} \E_{g_2\sim\pi_{-\lam_j r}} \frac1n \sum_{i=1}^n [ \ell(Y_i,g_1(X_i)) - \ell(Y_i,g_2(X_i)) ]^2}\\
	L & = \log[3\log^2(e n) \eps^{-1} ] 
	\end{align*}
and for any $0 \le i < j \le \log n$ and $\eps>0$, 
	\begin{align*}
	S( i , j ) = a(\lam_j) 	\E_{g_1\sim\pi_{-\lam_i r}} \E_{g_2\sim\pi_{-\lam_j r}} & \frac1n \sum_{i=1}^n [ \ell(Y_i,g_1(X_i)) - \ell(Y_i,g_2(X_i)) ]^2\\
		& + b(\lam_j) \big[ \calC( i ) + \calC( j ) + 2 L \big].
	\end{align*}
Inequality \eqref{eq:relloc} implies that with probability at least $1-\eps$, for any $0 \le i < j \le \log n$, we have
  \begin{align*}
  \E_{g_2\sim\pi_{-\lam_j r}} R(g_2) - 
    \E_{g_1\sim\pi_{-\lam_i r}} R(g_1) \le \E_{g_2\sim\pi_{-\lam_j r}} r(g_2) - \E_{g_1\sim\pi_{-\lam_i r}} r(g_1) 
    + S(i,j).
  \end{align*}
This leads me to consider in the chapter 2 of my PhD thesis the following choice of the temperature/complexity parameter in the classification setting.
%

{\bf Algorithm 1.}\label{alg1}
{\it Let $u(0) = 0$.
For any $k \ge 1$, define $\hat{\lam}_{k-1} = \lam_{u(k-1)}$ and $u(k)$ as the smallest integer $j \in ]u(k-1);\log n]$ such that
	$$\E_{g_2\sim\pi_{-\lam_j r}} r(g_2) 
	  - \E_{g_1\sim\pi_{-\hat{\lam}_{k-1} r}} r(g_1) + S\big( u(k-1) , j \big) \le 0.$$
Classify using a function drawn according to the posterior distribution 
associated with the last $u(k)$.}

This algorithm can be viewed in the following way: it ``ranks'' the estimator in the model
by increasing complexity (if we consider that $K(\pi_{-\lam_j r},\pi)$ is the complexity of the estimator associated with 
$\pi_{-\lam_j r}$), picks the ``first'' function in this list
and takes at each step the function of smallest complexity such that its risk is smaller than the one at the previous step. 
This is possible since we have \emph{empirical relative} bounds.
Subsequently to this work, different iterative schemes based on empirical relative 
PAC-Bayesian bounds have been proposed \cite{Alq06,Alq08,Cat07}.
The interest of the procedure lies in the following theoretical guarantee.	

\begin{thm} \label{th:indlocrel}
The iterative scheme is finite: there exists $K\in\N$ such that $u(K)$ exists but not $u(K+1)$. With probability at least $1-\eps$, for any $k\in\{1,\dots,K\},$ we have 
		$$\E_{g\sim\pi_{-\hat{\lam}_{k} r}} R(g) 
		  \le \E_{g\sim\pi_{-\hat{\lam}_{k-1} r}} R(g),$$
and
	\begin{align*}
	\E_{g \sim\pi_{-\hat{\lam}_{K} r}} & R(g) \le 
	  \underset{1 \le j \le \log n}{\min} \Bigg\{ 
	  \E_{g\sim\pi_{-\lam_{j-1} R}} R(g) + C \frac{\log[\log(en) \eps^{-1} ]}{\lam_j}\\
	& + \frac1{\lam_j} \underset{0\le i \le j}{\sup} \bigg\{ 
		\log \E_{g_1\sim\pi_{-\lam_i R}} \E_{g_2\sim\pi_{-\lam_i R}}
		e^{\frac{C\lam_i^2}{n} \P[g_1(X)\neq g_2(X) ]} \bigg\}
			  \Bigg\}.
	\end{align*}
\end{thm}

To illustrate this last theoretical guarantee, let us consider 
complexity and margin assumptions similar to the ones
used in the pioneering work of Mammen and Tsybakov \cite{Mam99}.
To detail these assumptions, let $d$ be the (pseudo-)distance on $\G(\X;\Y)$ defined by 
  $$d(g_1,g_2)=\P[g_1(X)\neq g_2(X)].$$
Let $\G\subset\G(\X;\Y)$. For $u>0$, the set $\cN\subset\G(\X;\Y)$ 
is called a $u$-covering net of $\G$ if we have
  $
  \G = \cup_{g\in\cN} \big\{g'\in\G; d(g,g')\le u\big\}.
  $
Let $H(u)$ denote the $u$-covering entropy, i.e. the logarithm of the smallest $u$-covering net of $\G$.
The complexity assumption is that there exist $C'>0$ and $q>0$ such that $H( u ) \le C' u^{-q}$
for any $u > 0$.
Let 
  $$g^* = \argmin_{g\in\G} R(g).$$
Without great loss of generality, we assume the existence of such a function. The margin assumption is that there exist $c'',C''>0$ and $\kappa\in[1,+\infty]$ such that for any function $g \in \G$,
	\begin{equation} \label{eq:margin}
	c'' \big[R( g ) - R(g^*)\big]^{\frac{1}{\kappa}} 
		\leq \P[g(X)\neq g^*(X)] \leq C'' \big[R( g ) - R(g^*)\big]^{\frac{1}{\kappa}}.
	\end{equation}
For any $k\in\N^*$, introduce $\pi_k$ the uniform distribution on the smallest 
$2^{-k}$ covering net. 

\begin{thm}
For the prior distribution $\pi = \sum_{k\ge 1} \frac{\pi_k}{k(k+1)}$,
the randomized estimator defined in Algorithm 1 (p.\pageref{alg1}) satisfies 
  $$
  \E_{g \sim\pi_{-\hat{\lam}_{K} r}} R(g) - R(g^*) \le
    C n^{-\frac{\kap}{2\kap-1+q}},
  $$
for some positive constant $C$.
\end{thm}
We also proved in \cite[Chap.3, Theorem 3.3]{Aud04} that the right-hand side is the minimax optimal convergence rates under such assumptions. 
Since the algorithm does not require the knowledge of the margin parameter $\kap$,
it is adaptive to this parameter.

Note that Assumption \eqref{eq:margin} is stronger than the usual assumption as the latter 
does not assume the left inequality.
In fact, to achieve minimax optimal rates under the usual margin assumption, while still
assuming polynomial covering entropies requires the chaining argument \cite[Chap.3]{Aud04}.
This leads us to study how to combine the chaining argument with the PAC-Bayesian approach
and make the connexion with majorizing measures from the generic chaining argument developed by Fernique and
Talagrand \cite{Tal96}, which we detail in the next section.

\section{Combining PAC-Bayesian and Generic Chaining Bounds} \label{sec:combin}

There exist many different risk bounds in
statistical learning theory. Each of these bounds contains an
improvement over the others for certain situations or algorithms.
In \cite{AudBou07}, Olivier Bousquet and I underline the links between these bounds,
and combine several different improvements into a single
bound. In particular, we combine the PAC-Bayes approach with the optimal union bound provided by the generic
chaining technique developed by Fernique and Talagrand,
in a way that also takes into account the variance of the combined
functions. We also show how this connects to Rademacher based
bounds.
The interest in generic chaining rather than just Dudley's chaining \cite{Dud78}
comes from the fact that it captures better the behaviour supremum of a Gaussian process \cite{Tal96}.
In statistical learning theory, the process of interest and which is asymptotically Gaussian is $g\mapsto R(g)-r(g)$.

I hereafter give a simplified version of the main results of \cite{AudBou07}. 
Let me first introduce the notation. 
We still consider a set $\G\subset\G(\X;\Y)$, $g^*=\argmin_{g\in\G} R(g)$, and that losses take their 
values in $[0,1]$.
We consider a sequence of nested
partitions $(\A_j)_{j\in\N}$ of the set $\G$, that is (i)
$\A_j$ is a partition of $\G$ either countable
    or equal to the set of all singletons of $\G$, and
(ii) the $\A_j$ are nested: each element of $\A_{j+1}$ is
    contained in an element of $\A_j$, and $\A_0 = \{\G\}$.
For the partition $\A_j$ and for $g\in\G$, we denote by $A_j(g)$ the unique element of
$\A_j$ containing $g$. Given a sequence of nested partitions $(\A_j)_{j\in\N}$, we can build a
collection $(S_j)_{j\in\N}$ of approximating subsets of $\G$ in
the following way: for each $j\in\N$, for each element $A$ of
$\A_j$, choose a unique element of $\G$ contained in $A$ and
define $S_j$ as the set of all chosen elements. We have $|S_0|=1$
and denote by $p_j(g)$ the unique element of $S_j$ contained in
$A_j(g)$.
Finally, we also consider that for each $j\in\N$, we have a distribution $\pij$ on $\G$ at our disposal.

Our bound will depend on the specific choices of the distributions $\pij$, the nested partitions $(\A_j)$,
the associated sequence of approximating sets $(S_j)$, and the corresponding 
approximating functions $p_j(g), g\in\G$.
Denote $\delta_g$ the Dirac measure on $g$. For a probability
distribution $\rho$ on $\G$, define its $j$-th projection as
  $$[\rho]_j = \sum_{g \in S_j} \rho [A_j(g)] \delta_g,$$
when $S_j$ is countable and $[\rho]_j=\rho$ otherwise.
For any $\eps>0$ and $\rho\in\M$, define the complexity of
$\rho$ at scale $j$ by
  $$
  \K_j(\rho)=K([\rho]_j,\pijj) + \log[j(j+1)\eps^{-1}],
  $$
and introduce the average distance between the $(j-1)$-th and $j$-th projections by
  \begin{align*}
  \D_j(\rho)=\E_{g\sim\rho} \bigg\{ \frac12 &\E_{Z\sim P} \Big\{ \ell\big(Y,[p_j(g)](X)\big)
    -\ell\big(Y,[p_{j-1}(g)](X)\big) \Big\}^2\\
  & + \frac1{2n} \sum_{i=1}^n \Big\{ \ell\big(Y_i,[p_j(g)](X_i)\big)-\ell\big(Y_i,[p_{j-1}(g)](X_i)\big) \Big\}^2 \bigg\}
  \end{align*}

\begin{thm} 
If the following condition holds
    \begin{equation} \label{eq:condind}
    \underset{j \rightarrow +\infty}{\lim} \sup_{g\in\G} \,
        \big\{ R(g) - R[p_j( f )] - r(g) + r[p_j( f )] \big\} = 0, \text{\qquad a.s.}
    \end{equation}
then for any $0 < \be \leq 0.7$, with probability at least $1-\eps$, for any $\rho\in\M$, we have
    \begin{align}
    \E_{g\sim\rho} R(g) - R(g^*) \le \E_{g\sim\rho} r(g) -  r(g^*) & +
        \frac{4}{\sqrt{n}} \sum_{j = 1}^{+\infty} \sqrt{ \D_j(\rho) \K_j(\rho) } \notag \\
    & + \frac{4}{\sqrt{n}} \sum_{j = 1}^{+\infty}
        \sqrt{ \frac{\D_j(\rho) }{\K_j(\rho)} }\log \log\left( 4 e^2 \frac{\K_j(\rho)}{\D_j(\rho)} \right). \label{eq:combin}
    \end{align}
\end{thm}

Assumption \eqref{eq:condind} is not very restrictive. For instance,
it is satisfied when one of the following condition holds:
\begin{itemize}
\item there exists $J\in\N^*$ such that $S_J=\G$, \item almost
surely
    $\lim_{j \rightarrow +\infty} \sup_{g\in\G,x\in\X,y\in\Y} \, | \ell\big(y,g(x)\big) - \ell\big(y,[p_j(g)](x)\big) |
    = 0$
(it is in particular the case when the bracketing entropy of the
set $\G$ is finite for any radius and when the $S_j$'s and $p_j$'s
are appropriately built on the bracketing nets of radius going to
$0$ when $j\rightarrow +\infty$).
\end{itemize}

The bound \eqref{eq:combin} combines several previous improvements. 
It contains an optimal union bound, both in the sense of optimally taking
into account the metric structure of the set of functions (via the majorizing measure approach) and
in the sense of taking into account the averaging distribution. It is sensitive to the variance of
the functions and consequently will lead to fast convergence rates (that is faster than $1/\sqrt{n}$),
under margin assumptions such as the ones considered in the works of Nédélec and Massart \cite{MasNed06}
or Mammen and Tsybakov \cite{Mam99}.
It holds for randomized classifiers but contrarily to usual PAC-Bayesian bounds, it remains finite 
when the averaging distribution is concentrated at a single prediction function.
On the negative side, there still remains work in order to get a 
fully empirical bound (it is not the case here since $\D_j(\rho)$ is not observable) 
and to better understand the connection with Rademacher averages.

Independently of the generic chaining argument, we use 
a carefully weighted union bound argument, which is at the origin of the $\log\log$ term in \eqref{eq:combin}
and leads to the following corollary of the main result in \cite{AudBou07}.

\begin{thm}
For any $\eps>0$, with probability at least $1-\eps$, for any $\rho\in\M$, we have
  \begin{align*} 
  \E_{g\sim\rho} R(g)-\E_{g\sim\rho} r(g) \le C \sqrt{\frac{K(\rho,\pi)+\log(2\eps^{-1})}{n}},
  \end{align*}
for some numerical constant $C>0$ \cite[Section 4.3]{AudBou07}.
\end{thm}

This result means that neither the $\log(n)$ term in \eqref{eq:McA} (p.\pageref{eq:McA}) or the Shannon's entropy term in \eqref{eq:basic} (p.\pageref{eq:basic}) is 
needed if we are allowed to have a numerical factor slightly larger in front of the square root term.




\chapter{The three aggregation problems} \label{chap:agg}

\section{Introduction}

Aggregation is about combining different prediction functions in order
to get a better prediction. It has become popular and has been intensively studied
these last two decades partly thanks to the success of boosting algorithms, 
and principally of the AdaBoost algorithm, introduced by Freund and Schapire \cite{Fre97}.
These algorithms use linear combination of a large number of simple functions
to provide a classification decision rule.

In this chapter, we focus on the least squares setting, 
in which the outputs are real numbers and the risk of a prediction function $g:\X\ra \R$ is 
  $$R(g) = \E [Y-g(X)]^2.$$ 
Our results are nevertheless of interest for classification also as
any estimate of the conditional expectation of the output knowing the input
leads by thresholding to a classification decision rule, and
the quality of this plug-in estimator is directly linked 
to the quality of the least squares regression estimator 
(see \cite[Section 6.2]{Dev96}, \cite{AudTsy07}
and specifically the comparison lemmas of its section 5,
and also \cite{LugVay04,BarJorMcA06,BarTra07} for consistency results in classification
using other surrogate loss functions).

Boosting type classification methods usually aggregate simple functions, but the aggregation is also of interest when some potentially complicated functions are aggregated. 
More precisely, when facing the data, the statistician has often to choose 
several models which are likely to be relevant for the task.
These models can be of similar structures (like embedded balls of functional spaces) or
on the contrary of very different nature (e.g., based on kernels, splines, wavelets or on parametric approaches). 
For each of these models, we assume that we have a learning scheme which produces 
a 'good' prediction function in the sense that its risk is as small as the risk of the best 
function of the model up to some small additive term\footnote{The learning procedure
could differ for each model, or on the contrary, be the same but using 
different values of a tuning parameter.}. Then the question is
to decide on how we use or combine/aggregate these schemes. 
One possible answer is to split the data into two groups, use
the first group to train the prediction function (i.e. compute the estimator) associated with each model, and
then use the second group to build a prediction function which is as good as (i) the
best of the previously learnt prediction functions, (ii) the best convex combination
of these functions or (iii) the best linear combination of these functions, in terms of risk, up to 
some small additive term. The three aggregation problems we will focus on in this chapter
concern the second part of this scheme. 
The idea of mixing (or combining or aggregating) the estimators 
originally appears in \cite{Nem98,Jud00,Yan99,Yan01b}.

We hereafter treat the initial estimators as fixed functions, which
means that the results hold conditionally on the data set on which they have been obtained,
this data set being independent of the $n$ input-output observations $Z_1^n$.
Specifically, let $g_1$,\dots,$g_d$ be $d$ prediction functions, with $d\ge 2$.
Introduce
  $$\gms\in\und{\argmin}{g\in\{g_1,\dots,g_d\}} \, R(g),$$
  $$\gc\in\und{\argmin}{g\in\{\sum_{j=1}^d \th_j g_j;
    \th_1\ge0,\dots,\th_d\ge 0,\sum_{j=1}^d\th_j= 1\}} \, R(g),$$
and
  $$\gl\in\und{\argmin}{g\in\{\sum_{j=1}^d \th_j g_j;\th_1\in\R,\dots,
    \th_d\in\R\}} R(g).$$
The model selection aggregation task \pbms\ is to find an estimator
$\hg$ based on the observed data $Z_1^n$ for which the excess risk $R(\hg)-R(\gms)$
is guaranteed to be small.
Similarly, the convex (resp. linear) aggregation task \label{pbc}
\pbc\ (resp. \pbl) is to find an estimator
$\hg$ for which the excess risk $R(\hg)-R(\gc)$ (resp. $R(\hg)-R(\gl)$)
is guaranteed to be small.

The minimax optimal rates of aggregation are given in \cite{Tsy03} and references within.
Under suitable assumptions, it is shown that there exist estimators 
$\hg_\text{\bf MS}$, $\hg_\text{\bf C}$ and $\hg_\text{\bf L}$ such that
  \beglab{eq:ms}
  \E R(\hg_\text{\bf MS})-R(\gms) \le C \min\left(\frac{\log d} n,1\right),
  \endlab
  $$\E R(\hg_\text{\bf C})-R(\gc) \le C \min\left(\sqrt{\frac{\log(1+ d/\sqrt{n})} n},\frac{d} n,1\right),$$
  $$\E R(\hg_\text{\bf L})-R(\gl) \le C \min\left(\frac{d} n,1\right),$$
where $\hg_L$ (and for $d\le n$, $\hg_C$) require the knowledge of the input distribution. 
We recall that $C$ is a positive constant that may differ from line to line.
Tsybakov \cite{Tsy03} has shown that these rates cannot be uniformly improved in
the following sense. Let $\sigma>0$, $L>0$ and
Let $\cP_{\sigma,L}$ be the set of probability distributions
on $\X\times \R$ such that we almost surely have
$Y=g(X)+\xi$, with $\|g\|_\infty \le L$, and $\xi$ 
a centered Gaussian random variable independent of $X$ and with variance $\sigma^2$.
For appropriate choices of $g_1,\dots,g_d$, the following lower bounds hold:
  $$\und{\inf}{\hg} \und{\sup}{P\in\cP_{\sigma,L}} \big\{ \E R(\hg)-R(\gms) \big\}
    \ge C \min\left(\frac{\log d} n,1\right),$$
  $$\und{\inf}{\hg} \und{\sup}{P\in\cP_{\sigma,L}} \big\{ \E R(\hg)-R(\gc) \big\}
    \ge C \min\left(\sqrt{\frac{\log(1+ d/\sqrt{n})} n},\frac{d} n,1\right),$$
  $$\und{\inf}{\hg} \und{\sup}{P\in\cP_{\sigma,L}} \big\{ \E R(\hg)-R(\gl) \big\}
    \ge C \min\left(\frac{d} n,1\right),$$
where the infimum is taken over all estimators.
The three aggregation tasks have also been studied in the least
squares regression with fixed design, where similar rates are obtained \cite{Bun07,Dal08,Dal09}.

This chapter will provide my contributions to the aggregation problems (in the random design setting)
summarized as follows.
\begin{itemize}
\item The expected excess risk $\E R(\hg)-R(\gms)$
of the empirical risk minimizer on $\{g_1,\dots,g_d\}$ (or its penalized variants)
cannot be uniformly smaller than $C \sqrt\frac{\log d}{n}$. Since the minimax optimal rate 
is $\frac{\log d}n$, this shows that these estimators are inappropriate for
the model selection task (Section \ref{sec:subopt}).
\item Catoni \cite{Cat99} and Yang \cite{Yan00} have independently shown that
the optimal rate $\frac{\log d}n$ in the model selection problem
is achieved for the progressive mixture rule. In \cite{Aud09a},
I provide a variant of this estimator coming from the field of sequential 
prediction of nonrandom sequences, and 
called the progressive indirect mixture rule. It has the benefit of
satisfying a tighter excess risk bound in a bounded setting (outputs in $[-1,1]$). I also study the case when the outputs have heavy tails
(much thicker than exponential tails), and show how the noise influences the 
minimax optimal convergence rate. I also provide
refined lower bounds of Assouad's type with tight constants (Section \ref{sec:aos}).
\item In \cite{Aud07b}, I show a limitation of the
algorithms known to satisfy \eqref{eq:ms}: despite having an expected excess risk
of order $1/n$ (if we drop the dependence in $d$), the excess risk of
the progressive (indirect or not) mixture rule suffers deviations of order $1/\sqrt{n}$
(Section \ref{sec:devsubopt}).
\item This last result leads me to define a new estimator $\hg$
which does not suffer from this drawback: the deviations of the excess risk
$\E R(\hg)-R(\gms)$ is of order $\frac{\log d}{n}$ (Section \ref{sec:star}).
\item In my PhD (its first chapter), I provide
an estimator $\hg$ based on empirical bounds of any aggregation procedures for which with high probability 
$$R(\hg)-R(\gc) \le \left\{ \begin{array}{lll}
C\sqrt{\frac{\log(d\log n)}n} & \text{always,}\\
C{\frac{\log(d\log n)}n} & \text{if $R(\gms)=R(\gc)$.}
\end{array} \right.$$ 
This means that for $n^{\frac12+\delta}\le d\le e^n$ with $\delta>0$,
the estimator has the minimax optimal rate of task \pbc, and 
is adaptive to the extent that it
has also the minimax optimal rate of task \pbms\ when $R(\gms)=R(\gc)$ (Section \ref{sec:aggconv}).
\item Finally, Olivier Catoni and I \cite{Aud09b} provide minimax results for \pbl,
and consequently also for \pbc\ when $d\le \sqrt{n}$.
The strong point of these results is that it does not require the
knowledge of the input distribution, nor uniformly bounded exponential moments 
of the conditional distribution of the output knowing the input and has no
extra logarithmic factor unlike previous results. In particular, provided that we know 
$H$ and $\sigma$ such that $\|\gl\|_\infty \le H$ and
$\sup_{x\in\X} \E\big\{[Y-\gl(X)]^2 \big| X=x\big\} \le \sigma^2$,
we propose an estimator $\hg$ satisfying
$\E R(\hg) - R(\gl) \le 68 (\sigma+H)^2\ \frac{d+2}n$ (Section \ref{sec:agglin}).
\end{itemize}

I should conclude this introductory section by emphasizing
that we will not assume that the regression function $\greg:x\mapsto \E(Y|X=x)$, which minimizes
the risk functional, is in 
the linear span of $\{g_1,\dots,g_d\}$.
This means that bounds of the form
  \beglab{eq:cgrand}
  \E R(\hg) - R(g^*) \le c [ R(\greg) - R(g^*) ] + \text{residual term},
  \endlab
with $c>1$ are not of interest in our setting\footnote{These last bounds, which are relatively common in
the literature, 
are nonetheless useful
in a nonparametric setting in which the statistician is allowed to take $\{g_1,\dots,g_d\}$ large enough
so that $R(\greg) - R(g^*)$ is of the same order as the residual term.}, as they would not provide the minimax learning
rate when $R(\greg) \gg R(g^*)$. 

\section{Model selection type aggregation}

\subsection{Suboptimality of empirical risk minimization} \label{sec:subopt}

Any empirical risk minimizer and any of its penalized variants are really poor algorithms 
in this task since their expected convergence rate cannot be uniformly faster than $\sqrt{\fracl{\log d}{n}}$.
The following lower bound comes from \cite{Aud07b} (see \cite{Lee98}, \cite[p.14]{Cat99}, \cite{Lec07,Jud08,Men08} for similar results and variants). 

\begin{thm} 
For any training set size $n$, there exist $d$ prediction functions $g_1,\dots,g_d$ taking their values in $[-1,1]$ such that
for any learning algorithm $\hg$ producing a prediction function in $\{g_1,\dots,g_d\}$  
there exists a probability distribution generating the data for which
$Y\in[-1,1]$ almost surely, and
	$$
	\E R(\hg) - R(\gms) \ge \min\Big( \sqrt{\frac{\lfloor \log_2 d \rfloor}{4n}} , 1 \Big),
	$$
where 
$\lfloor \log_2 d\rfloor$ denotes the largest integer smaller or equal to the
logarithm in base 2 of $d$. 
\end{thm}

%

\subsection{Progressive indirect mixture rules} \label{sec:aos}

The result of the previous section 
shows that, to obtain the minimax optimal rate given in \eqref{eq:ms},
an estimator has to look at an enlarged set of prediction functions.
Until our work, the only known optimal estimator was based on a Cesaro mean of 
Bayesian estimators, also referred to as progressive mixture rule. 

To define it, let $\pi$ be the uniform distribution on the finite set 
$\{g_1,\dots,g_d\}$. For any $i\in\{0,\dots,n\}$, the cumulative loss suffered by the prediction function
 $g$ on the first $i$ pairs of input-output, denoted $Z_1^i$ for short, is 
        $$\Sigma_i(g) \eqdef \sum_{k=1}^i [Y_k-g(X_k)]^2,$$
where by convention we take $\Sigma_0$ identically equal to zero.
Let $\lam>0$ be a parameter of the estimator. 
Recall that $\pi_{-\lam \Sigma_i}$ is the distribution on $\{g_1,\dots,g_d\}$ admitting a density
with respect to $\pi$ that is proportional to $e^{-\lam \Sigma_i}$.

The \emph{progressive mixture rule} (PM) predicts according to $\frac{1}{n+1} \sum_{i=0}^n \expec{g}{\pi_{-\lam \Sigma_i}} g$. 
In other words, for a new input $x$, the predicted output is
        $$
        \frac{1}{n+1} \sum_{i=0}^n \frac{\sum_{j=1}^d g_j(x) e^{-\lam \Sigma_i(g_j)}}{\sum_{j=1}^d e^{-\lam \Sigma_i(g_j)}}.
        $$

A specificity of PM is that its proof of optimality is not 
achieved by the most prominent tool in statistical learning theory:
bounds on the supremum of empirical processes (see \cite{Vap82}, 
and refined works as \cite{BarBouMen05,Kol06,Mas00,bou05} and references within).
The idea of the proof, which comes back to Barron \cite{Bar87}, is based 
on a chain rule and appeared to be successful for least squares and entropy losses
\cite{Cat97,Cat99,Bar99,Yan00} and for general loss in \cite{Jud08}.

Here my first contribution was to take ideas coming from the field of 
sequential prediction of nonrandom sequences
(see e.g. \cite{Mer98,CesLug06} for a general overview and \cite{Hau98,Ces97,Ces99,Yar04} for more specific results with
sharp constants) and propose a slight generalization of progressive mixture rules,
that I called progressive indirect mixture rules.

The \emph{progressive indirect mixture rule} (PIM) is also parameterized by a real number $\lam>0$,
and is defined as follows.
For any $i\in\{0,\dots,n\}$, let $\hh_i$ be a prediction function such that 
        \beglab{eq:pimdef}
        [Y-\hh_i(X)]^2 \le -\frac{1}{\lam} \log \expe{g}{\pi_{-\lam \Sigma_i}} e^{-\lam [Y-g(X)]^2} \qquad\text{ a.s.}
        \endlab
If one of the $\hh_i$ does not exist, the algorithm is said to fail. Otherwise it
predicts according to $\frac{1}{n+1} \sum_{i=0}^n \hh_i$.

This estimator is a direct transposition from the sequential prediction algorithm
proposed and studied in \cite{Vov90,Hau98,Vov98} to our ``batch'' setting.
The functions $\hh_i$ do not necessarily exist, but are also not necessarily unique when they exist.
A technical justification of \eqref{eq:pimdef} comes from the analysis of PM synthetically written in Appendix \ref{app:proofpm}.

When $\max\big(|Y|,|g_1(X)|,\dots,|g_d(X)|\big)\le B$ a.s. for some $B>0$ and for $\lam$ large enough, the
functions $\hh_i$ exist (so the algorithm does not fail). Still in this uniformly bounded setting,
it can be shown that PM is a PIM for $\lam$ large enough.
On the other hand, there exists $\lam>0$ small enough for which the algorithm does not fail
and such that PM is not a particular case of PIM,
that is one cannot take $\hh_i=\expec{g}{\pi_{-\lam \Sigma_i}} g$ to satisfy \eqref{eq:pimdef}
(see \cite[Example 3.13]{Hau98}). In fact, it is also shown there that 
PIM will not generally produce a prediction function in the convex hull 
of $\{g_1,\dots,g_d\}$ unlike PM. The following amazingly sharp upper bound on the excess risk of PIM holds.

\begin{thm} \label{th:pim}
Assume that $|Y|\le 1$ a.s. and $\|g_j\|_\infty\le 1$ for any $j\in\{1,\dots,d\}$.
Then, for $\lam\le\frac{1}{2}$, \emph{PIM} does not fail and its expected
excess risk is upper bounded by $\frac{\log d}{\lam(n+1)}$, that is
  \beglab{eq:pim}
  \E_{Z_1^n} R\bigg(\frac{1}{n+1} \sum_{i=0}^n \hh_i\bigg) - R(\gms)\le \frac{\log d}{\lam(n+1)}.
  \endlab
\end{thm}
  
It essentially comes from a result in sequential prediction 
and the fact that results expressed in cumulative loss can be transposed to our setting
since the expected risk of the randomized procedure based on sequential 
predictions is proportional to the expectation of the cumulative loss of the sequential procedure.
Precisely, the following statement holds.

\begin{lemma} \label{lem:rand}
Let $\A$ be a learning algorithm which produces the prediction function $\A(Z_1^i)$ at time $i+1$, i.e. from the data 
$Z_1^i=(Z_1,\dots,Z_i)$.
Let $\cL$ be the randomized algorithm which produces a prediction function $\cL(Z_1^n)$ 
drawn according to the uniform
distribution on $\{\A(\emptyset),\A(Z_1),\dots,\A(Z_1^n)\}$. 
The (doubly) expected risk of $\cL$ is equal to $\frac{1}{n+1}$ times 
the expectation of the cumulative loss of $\A$ on the sequence $Z_1,\dots,Z_{n+1}$,
where $Z_{n+1}$ denotes a random variable independent of the training set 
$Z_1^n\eqdef(Z_1,\dots,Z_n)$ and with the same distribution $P$.
\end{lemma}

My second contribution to model selection aggregation in \cite{Aud09a} is to 
provide a different viewpoint of the progressive mixture rule from the one in \cite{Jud08},
leading to a slight improvement in the moment condition of the initial version of \cite{Jud08}.
The result is the following and is extended to the $L_q$ loss functions for $q\ge 1$ in
\cite[Section 7]{Aud09a}.
\begin{thm}
Assume that $\|g_j\|_\infty\le 1$ for any $j\in\{1,\dots,d\}$,
and $\E |Y|^s\le A$ for some $s\ge 2$ and $A>0$.
For $\lam=C_1 \big(\frac{\log d}{n}\big)^{2/(s+2)}$ with $C_1>0$, the expected excess risk of PM
is upper bounded by $C \big(\frac{\log d}{n}\big)^{s/(s+2)}$, that is
        $$
        \E_{Z_1^n} R\bigg(\frac{1}{n+1} \sum_{i=0}^n \expec{g}{\pi_{-\lam \Sigma_i}} g\bigg) - R(\gms) 
          \le C \bigg(\frac{\log d}{n}\bigg)^{s/(s+2)},
        $$
for a quantity $C$ which depends only on $C_1$, $A$ and $s$.
\end{thm}

The convergence rate cannot be improved in a minimax sense \cite[Section 8.3.2]{Aud09a}.
These results show how heavy output tails influence the learning rate: for the limiting case $s=2$,
the bounds are of order $n^{-1/2}$ while for $s$ going to infinity, it is of order of $n^{-1}$, that is the rate
in the bounded case, or in the uniformly bounded conditional exponential moment setting.

The lower bounds developed to prove the minimax optimality of the above
result are based on a refinement of Assouad's lemma, which 
allows to get much tighter constants. For instance, 
they improve the lower bounds for Vapnik-Cervo\-nenkis classes
\cite[Chapter 14]{Dev96} by a factor greater than $1000$, and lead to the
following simple bound.


\begin{thm} 
Let $\cF$ be a set of binary classification functions of VC-dimension $V$.
For any classification rule $\hf$ trained on a data set of size $n\ge \frac{V}4$, there exists a probability distribution generating the data for which
	\beglab{eq:vcnew}
	\E R(\hf) - \und{\inf}{f\in\cF} R(f) 
		\geq \frac{1}{8} \sqrt{\frac{V}{n}}.
	\endlab
\end{thm}

\subsection{Limitation of progressive indirect mixture rules} \label{sec:devsubopt}

Let $\hg_\lam$ denote a progressive indirect mixture rule (it could be a progressive mixture or not) for some $\lam>0$.
Under boundedness assumptions (and even under some exponential moment assumptions) and appropriate choice of $\lam$,
$\hg_{\lam}$ satisfies an expected excess risk bound
of order $\frac{\log d}n$. Then one would also expect the excess risk 
$R(\hg) - R(\gms)$ to be of order $\frac{\log d}n$ with high probability. In fact, this does not necessarily happen
as the following theorem holds for $d=2$.

\begin{thm}
Let $g_1$ and $g_2$ be the constant functions respectively equal to $1$ and $-1$.
For any $\lam>0$ and any training set size $n$ large enough, there exist $\eps>0$ and a distribution
generating the data for which
$Y\in[-1,1]$ almost surely, and with probability larger than $\eps$, we have
	$$
	R(\hg_{\lam})- R(\gms) \ge c \sqrt{\frac{\log(e\eps^{-1})}{n}}
	$$
where $c$ is a positive constant only depending on $\lam$. 
\end{thm}

More precisely, in \cite{Aud07b}, it is shown that for large enough $n$, and some constants $c_1>1/2$ and $c_2>0$ only depending on $\lam$, 
with probability at least $1/n^{c_1}$, we have
$R(\hg_{\lam})- R(\gms) \ge c_2 \sqrt{\fracc{(\log n)}{n}}$.
Since $c_1>1/2$, there is naturally no contradiction with the fact that, in expectation, the excess risk is of order $\frac{\log d}n$.

\subsection{Getting round the previous limitation} \label{sec:star}

I now present the algorithm introduced in \cite{Aud07b}, and called the empirical star estimator,
which has both expectation and deviation convergence rate of order $\frac{\log d}n$.
The empirical risk of a prediction function $g:\X\ra \R$ is defined by
  $$r(g)=\frac1n\sum_{i=1}^n [Y_i-g(X_i)]^2.$$
Let $\hgerm$ be an empirical risk minimizer among the reference functions: 
  $$\hgerm\in\und{\argmin}{g\in\{g_1,\dots,g_d\}} r(g).$$
For any prediction functions $g,g'$, let $[g,g']$ denote the set of functions which are convex combination of $g$ and $g'$:
$[g,g']=\big\{\alpha g+(1-\alpha)g':\alpha\in[0,1]\big\}.$
The empirical star estimator $\hgstar$ minimizes the empirical risk over a star-shaped set of functions, precisely:
  $$\hgstar\in\und{\argmin}{g\in[\hgerm,g_1]\cup\cdots\cup[\hgerm,g_d]} r(g).$$
The main result concerning this estimator is the following.

\begin{thm} \label{th:star}
Assume that $|Y|\le B$ almost surely and $\|g_j\|_\infty\le B$ for any $j\in\{1,\dots,d\}$.
Then the empirical star algorithm satisfies:
for any $\eps>0$, with probability at least $1-\eps$, 
  $$
  R(\hgstar) - R(\gms)\le \frac{200 B^2 \log[3d(d-1)\eps^{-1}]}n\le \frac{600 B^2 \log(d\eps^{-1})}n.
  $$
Consequently, we also have
  $$
  \E R(\hgstar) - R(\gms)\le \frac{400 B^2 \log(3d)}n.
  $$
\end{thm}

An additional advantage of this empirical star estimator is that it does not need to know the constant $B$. In
other words, it is adaptive to the smallest value of $B$
for which the boundedness assumptions hold. This was not the case of 
the progressive mixture rules in which we need to take $\lam \le 1/(2B^2)$
for the indirect ones and $\lam \le 1/(8B^2)$ for the ``direct'' one in order to state 
Inequality \eqref{eq:pim}.
On the negative side, the theoretical guarantee on the expected excess risk 
is $200$ times larger than the one stated for the best PIM. However, this is more an artefact of the intricate proof 
of Theorem \ref{th:star} than a drawback of the algorithm.

Another difference between progressive mixture rules is that the function output 
by the estimator is inside $\cup_{1\le j<k\le d} [g_j,g_d]$, which is
not in general the case for the progressive (indirect) mixture rules.
We have already seen in Section \ref{sec:subopt} that
the empirical risk minimizer on $\{g_1,\dots,g_d\}$ has not the minimax optimal rate.
A natural question in view of the empirical star algorithm is whether
empirical risk minimization on $\cup_{1\le j<k\le d} [g_j,g_d]$ would reach 
the $(\log d)/n$ rate. It can be proved for $d=3$ that, even under boundedness assumptions, the rate cannot be
better than $n^{-2/3}$ for an adequate choice of the functions and the distribution 
(proof omitted by lack of interest in negative results).

Interestingly, Lecué and Mendelson \cite{Lec09} proposed a variant of the empirical star algorithm,
which also uses the empirical risk minimizer $\hgerm$ to define a set of functions on which 
the empirical risk is minimized. Precisely, for a confidence level $\eps>0$, let $\hat{\G}$
be the set of functions $g\in\{g_1,\dots,g_d\}$ satisfying
  \begin{align*}
  r(g) \le r(\hgerm) & + C B \sqrt{\frac{\log(2 d\eps^{-1})}n}
    \sqrt{\frac{\sum_{i=1}^n [g(X_i)-\hgerm(X_i)]^2}n}\\
  & + C \frac{B^2 \log(2 d\eps^{-1})}n .
  \end{align*}
where $C$ is a positive constant. The final estimator is the empirical risk minimizer in the convex hull of $\hat{\G}$.
It is also shown there that the selection of a subset of functions $\hat{\G}$ before taking the convex hull is necessary 
to achieve the minimax convergence rate since
the empirical risk minimizer on the convex hull of $\{g_1,\dots,g_d\}$ has an excess risk at least of
order $1/\sqrt{n}$ for an appropriate distribution and $d$ of order $\sqrt{n}$.

The advantage of the empirical star algorithm over the empirical risk minimizer on the
convex hull of $\hat{\G}$ is its adaptivity to both the confidence level and the constant $B$,
and a theoretical guarantee of the form $C \frac{\log(d\eps^{-1})}n$ instead of 
$C \frac{\log(d)\log(\eps^{-1})}n$ for the empirical risk minimizer on the
convex hull of $\hat{\G}$.

\section{Convex aggregation} \label{sec:aggconv}

When $d\le \sqrt{n}$, the minimax learning rate for problem \pbc\
and \pbl\ are both of order $\frac{d}n$, meaning that estimators solving problem \pbl\
are solutions to problem \pbc\ for $d\le \sqrt{n}$. So, estimators for
$d\le \sqrt{n}$ are given in the section devoted to linear aggregation (Section \ref{sec:agglin}),
and this section focuses on the case when $d\ge n^{\frac12+\delta}$. 

The literature contains few results for problem \pbc\ with constant 
$c=1$ in \eqref{eq:cgrand} and minimax optimal residual term for $d\ge n^{\frac12+\delta}$, 
with $\delta>0$.
The first type of results is to apply the progressive mixture rule on an appropriate
grid of the simplex \cite{Tsy03}. Another solution is to use the exponentiated 
gradient algorithm introduced and studied by Kivinen and Warmuth \cite{Kiv97} in 
the context of sequential prediction for the quadratic loss, and then extended 
to general loss functions by Cesa-Bianchi \cite{Ces99b}. Lemma \ref{lem:rand} 
has to be invoked to convert these algorithms and the bounds to our statistical framework.
Juditsky, Nazin, Tsybakov and Vayatis \cite{JudNazTsyVay05} has viewed the
resulting algorithm as a stochastic version of the mirror descent algorithm, and
proposed a different choice of the temperature parameter, while still reaching the optimal
convergence rate.
All the above results hold in expectation, and it is not clear that
the deviations of the excess risk bounds are sub-exponential. The 
estimator presented hereafter
does not share this drawback. 

To address problem \pbc\ (defined in page \pageref{pbc}), the first chapter of my PhD thesis establishes empirical excess risk bounds for 
any estimator that produces a prediction function in the convex hull of
$g_1, \dots,g_d$ whatever the empirical data are. 
Any such estimator $\hg$ can be associated with 
a function $\hrho$ mapping a training set to a distribution on 
$\{g_1,\dots,g_d\}$ such that $\hg(Z_1^n) = \E_{g\sim\hrho(Z_1^n)} g.$
Conversely, any mapping $\hrho$ from $\Z^n$ (the set of training sets of size $n$)
to the set $\M$ of distributions on $\{g_1, \dots,g_d\}$ defines
an estimator   
  $$\hg = \E_{g\sim\hrho} g,$$  
where we have dropped the training set $Z_1^n$ for sake of compactness.
Similarly, there exists a distribution $\rhoc$ on $\{g_1,\dots,g_d\}$ such that 
  $$\gc=\E_{g\sim\rhoc} g.$$

The assumptions are boundedness of the functions $g_1,\dots,g_d$ and of the regression function
$\greg:x\mapsto \E(Y|X=x)$ and uniform boundedness of the conditional exponential moments of
the output knowing the input. Precisely, there exist $B>0$, $\alpha > 0,$ and $M>0$
such that for any $g',g''$ in $\{\greg,g_1,\dots,g_d\}$, $\|g'-g''\|_\infty\le B$
and for any $x\in\X$,
            $$
            \E \big(e^{\alpha|Y-\greg(X)|}\big|X=x\big) \leq M.
            $$

\begin{thm}
Under the above assumptions, there exist $C_1 , C_2 >0$ depending only 
on the constant $M$ and the product $\alpha B$ such that for any (prior) distribution
$\pi\in\M$, any $\eps > 0$, and any aggregating procedure $\hrho$ : $\Z^n \rightarrow \M$,
with probability at least $1-\epsilon$, 
  \begin{align} \label{eq:solc}
  R(\E_{g\sim\hrho} & g) - R(\gc) \le \min_{\lam\in[0,C_1]} \bigg\{
        (1+\lam) \big[ r(\E_{g\sim\hrho} g) - r(\gc) \big] \notag\\
        & + \frac{2 \lam}n \sum_{i=1}^n \Var_{g\sim\hrho} g(X_i)
        + C_2\frac{B^2}{n} \frac{K(\hrho,\pi) + \log(2\log(2n)\eps^{-1})}{\lam} \bigg\}.
  \end{align}
\end{thm}

This bound comes from the PAC-Bayesian analysis, and consequently,
the complexity of an aggregating procedure is measured by the Kullback-Leibler
divergence of $\hrho$ with respect to some prior distribution $\pi$ on $\{g_1,\dots,g_d\}$.
In absence of prior knowledge, $\pi$ is chosen as the uniform distribution,
which allows to bound uniformly the KL divergence by $\log d$.
Besides the usual empirical excess risk, Inequality \eqref{eq:solc} depends 
on the empirical variance of $g(x)$ when $g$ is drawn according
to $\hrho$. Unlike the Kullback-Leibler term, this term is small for concentrated posterior distributions.

All previous results of this chapter were easily generalizable to loss of quadratic type under boundedness assumptions,
that is loss with second derivative with respect to its second argument uniformly lower and upper bounded by positive constants.
To my knowledge, the generalization cannot be done here as the analysis strongly relies
on the remarkable identity\footnote{To be precise, \cite[Chap.1]{Aud04} used  
  $
  R(\E_{g\sim \rho} g) 
   = \E_{g\sim \rho} R(g) - \frac12 \E_{g'\sim \rho} \E_{g''\sim \rho} \E [g'(X)-g''(X)]^2,
  $
but it would have been more direct to use \eqref{eq:rem}.}  
  \begin{align} \label{eq:rem}
  R(\E_{g\sim \rho} g) & = \E_{(g',g'')\sim \rho\otimes\rho} \E [Y-g'(X)][Y-g''(X)],
  \end{align}
which is specific to the quadratic loss and allows to apply the 
PAC-Bayesian analysis for distributions on the product space 
$\{g_1,\dots,g_d\}\times \{g_1,\dots,g_d\}$.

Let $\hrc$ be the distribution minimizing the right-hand side of \eqref{eq:solc}
with $\pi$ the uniform distribution on $\{g_1,\dots,g_d\}$ and where $-(1+\lam)r(\gc)$
is replaced by its upper bound $-r(\gc) - \lam \min_{g\in\{\sum_{j=1}^d \th_j g_j;
    \th_1\ge0,\dots,\th_d\ge 0,\sum_{j=1}^d\th_j= 1\}} r(g)$. When defining $\hrc$, for sake of computability of the estimator 
\cite[Chap.1, Theorem 4.2.2]{Aud04}, one can also replace the minimum over $[0,C_1]$ 
by a minimum over a geometric grid of the interval $[n^{-1},C_1]$ without altering 
the validity of the following theorem.

\begin{thm}
For any $\eps > 0$, with probability at least $1-\epsilon$, we have
  \begin{align*}
  R(\E_{g\sim\hrc} g) - R(\gc) \le C B & \sqrt{\frac{\log(d\log(2n)\eps^{-1})}{n}
    \E \Var_{g\sim\rhoc} g(X)}\\
    & + C B^2 \frac{\log(d\log(2n)\eps^{-1})}{n},
  \end{align*}
for some constant $C>0$ depending only on $\alpha B$ and $M$.
\end{thm}

Since $\E \Var_{g\sim\rhoc} g(X)\le B^2/4$, the excess risk is at most of order 
$\sqrt{\frac{\log(d\log(2n))}{n}}$, that is the minimax convergence rate of the 
convex aggregation task for $d\ge n^{\frac12+\delta}$, with $\delta>0$.
Besides, when the best convex combination
occurs to be a vertex of the simplex defined by $\{g_1,\dots,g_d\}$
the variance term equals zero, and thus, the convergence rate is $\frac{\log(d\log(2n))}{n}$,
that is the minimax convergence rate of model selection type aggregation 
(at least for $d\ge \log(2n)$).


\section{Linear aggregation} \label{sec:agglin}

To handle problems \pbc\ and \pbl\ in the same framework and also
to incorporate other possible constraints on the coefficients of the linear 
combination, let us consider
$\Theta$ a closed convex subset of $\R^d$, and define
  $$
  \G = \bigg\{\sum_{j=1}^d \th_j g_j;(\th_1,\dots,
    \th_d)\in\Theta\bigg\}.
  $$
Introduce the vector-valued function 
  $
  \vp: x\mapsto\left( g_1(x),\dots,g_d(x) \right)^T.
  $
The function 
  $\sum_{j=1}^d \th_j g_j$ can then be simply written 
$\langle \th , \vp\rangle$ with $\th=(\th_1,\dots,\th_d)^T$.
Let 
  $$
  g^*\in\und{\argmin}{g\in\G} R(g).
  $$
Thus, when $\Theta$ is the simplex of $\R^d$,
we have $g^*=\gc$ and when $\Theta=\R^d$, we have $g^*=\gl$.

Aggregating linearly functions to design a prediction function with
low quadratic risk is just the problem of linear least squares regression. 
It is a central task in statistics, since
both linear parametric models and
nonparametric estimation with linear approximation spaces (piecewise polynomials based 
on a regular partition, wavelet expansions, trigonometric polynomials, \dots)
are popular.
It has thus been widely studied.

Classical statistical textbooks often only state results
for the fixed design setting as a bound of order $d/n$ can
be rather easily obtained in this case. 
This can be misleading since it does not imply a $d/n$ 
upper bound in the random design setting.
For the truncated ordinary least squares estimator, 
Gy\"{o}rfi, Kohler, Krzy$\dot{\text{z}}$ak and Walk
\cite[Theorem 11.3]{Gyo04} give a bound of the form of \myeq{eq:cgrand}
with residual term of order $\frac{d \log n}n$ and $c=8$.
When the input distribution is known,
Tsybabov \cite{Tsy03} provides a  
bound of order $d/n$ on the expected risk of
a projection estimator on an orthonormal basis of $\G$
for the dot product $(f,g)\mapsto \E[f(X)g(X)]$.

Catoni \cite[Proposition 5.9.1]{Cat01} and Alquier \cite{Alq08} have used the PAC-Bayesian 
approach to prove high probability excess risk bounds of order $d/n$ 
involving the conditioning of
the Gram matrix $Q=\B{E} \bigl[ \vp(X) \vp(X)^T\bigr]$.
Both results require at least exponential moments on
the conditional distribution of the output $Y$ knowing the 
input vector $\vp(X)$.

It can be derived from the work of Birg\'e and Massart \cite{BirMas98}
an excess risk bound for the empirical risk minimizer 
of order at worst $\frac{d\log n}n$, and asymptotically of
order $d/n$. It holds with high probability, for a bounded set $\Theta$
and requires bounded input vectors and conditional exponential moments 
of the output.
Localized Rademacher complexities \cite{Kol06,BarBouMen05} also allows
to study the empirical risk minimizer on a bounded set of functions.
They lead to a high probability $d/n$ convergence rate of the
empirical risk minimizer under strong assumptions: uniform boundedness
of the input vector, the output and the parameter set~$\Theta$.

Penalized least squares estimators using the $L^2$-norm of 
the vector of coefficients, or more recently, its $L^1$-norm
have also been widely studied. A common characteristic of the 
excess risk bounds obtained for these estimators is that
it is of order $d/n$ only under strong assumptions on 
the eigenvalues (of submatrices) of $Q$. 

In \cite{Aud09b}, Olivier Catoni and I provide new risk bounds for the ridge estimator and
the ordinary least squares estimator (Section \ref{sec:ridge}). 
We also propose a min-max estimator
which has non-asymptotic guarantee of order $d/n$ under weak assumptions
on the distributions of the output $Y$ and the random variables $g_j(X)$,
$j=1,\dots,d$ (Section \ref{sec:minmax}). 
Finally, we propose a sophisticated PAC-Bayesian estimator which
satisfies a simpler $d/n$ bound (Section \ref{sec:sophis}).

The key common surprising factor of these results is the absence of exponential
moment condition on the output distribution while achieving exponential deviations. All risk bounds are obtained through a PAC-Bayesian analysis on truncated differences of losses. 
Our results tend to say that truncation leads to more robust algorithms.
Local robustness to contamination is usually invoked 
to advocate the removal of outliers, 
claiming that estimators should be made insensitive to 
small amounts of spurious data.
Our work leads to a different theoretical explanation.
The observed points having unusually large outputs when compared
with the (empirical) variance should be down-weighted in the estimation 
of the mean, since they contain less information than noise. 
In short, huge outputs should be truncated because of their low 
signal to noise ratio.

\subsection{Ridge regression and empirical risk minimization} \label{sec:ridge}

The ridge regression estimator on $\G$ is defined by
$\hfrlam=\langle \thrlam , \vp \rangle$ with
$$
\thrlam \in \arg \min_{\th \in \Theta} 
r(\langle \theta,\vp\rangle) + \lambda \lVert \theta \rVert^2,
$$
where $\lambda$ is some nonnegative real 
parameter and $r(\langle \theta,\vp\rangle)$
is the empirical risk of the function $\langle \theta,\vp\rangle$. 
In the case when $\lambda = 0$, 
the ridge regression $\hfrlam$ is nothing
but the empirical risk minimizer $\hgerm$.

In the same way we consider the optimal ridge function  
optimizing the expected ridge risk: $\frid=\langle \thrid , \vp \rangle$ with
$$
\thrid \in \arg \min_{\theta \in \Theta} \big\{ R(\langle \theta,\vp\rangle) 
+ \lambda \lVert \theta \rVert^2 \big\}.
$$

Our first result is of asymptotic nature.
It is stated under weak hypotheses, taking advantage
of the weak law of large numbers.

\begin{thm} \label{th:hfrlam}
Let us assume that
\begin{gather}
\B{E}\bigl[ \lVert \vp(X) \rVert^4 \bigr] < + \infty,\\
\text{and }\quad \B{E} \Bigl\{ \lVert \vp(X) \rVert^2 \bigl[ \frid(X) - Y \bigr]^2 \Bigr\} < + \infty.
\end{gather}

Let $\nu_1,\dots,\nu_d$ be the eigenvalues of 
the Gram matrix $Q=\B{E} \bigl[ \vp(X) \vp(X)^T\bigr]$,
and let $Q_{\lambda} = Q + \lambda I$ be the ridge regularization 
of $Q$. 
Let us define the {\em effective ridge dimension}
$$
D = \sum_{i=1}^d \frac{\nu_i}{\nu_i+\lam} \dsone_{\nu_i>0} 
= \Tr \bigl[ (Q+ \lambda I)^{-1} Q \bigr] 
= \B{E} \bigl[ \lVert Q_{\lambda}^{-1/2} \vp(X) \rVert^2 \bigr].
$$
When $\lambda = 0$, $D$ is equal to the rank 
of $Q$ and is otherwise smaller.
For any $\eps > 0$, there is $n_{\eps}$, 
such that for any $n \geq n_{\eps}$, 
with probability at least $1 - \eps$, 
\begin{align*}
R(\hfrlam) & + \lambda \lVert \hat{\th}^{\textnormal{(ridge)}} 
\rVert^2 \leq  \min_{\theta \in \Theta} \big\{ R(\langle \theta,\vp\rangle) 
+ \lambda \lVert \theta \rVert^2 \big\}
\\
& \qquad
+ C \ess \sup \E\big\{[Y-\frid(X)]^2 \big| X\big\} \, \frac{D + \log(3\eps^{-1})}{n}, 
\end{align*}
for some numerical constant $C>0$.
\end{thm}

This theorem shows that the ordinary least squares estimator
(obtained when $\Theta = \B{R}^d$ and $\lambda = 0$), 
as well as the empirical risk minimizer on any closed convex 
set, asymptotically reach a $d/n$ speed of convergence
under very weak hypotheses. It shows also the regularization 
effect of the ridge regression. There emerges an {\em effective
dimension} $D$, 
where the ridge penalty has a threshold effect on the eigenvalues
of the Gram matrix. 

On the other hand, the weakness of this result is 
its asymptotic nature : $n_\eps$ may be arbitrarily large
under such weak hypotheses, and this shows even in the 
simplest case of the estimation of the mean of a real-valued
random variable by its empirical mean, which is the case when 
$d = 1$ and $\vp(X) \equiv 1$ \cite{Cat09}.
Typically, the proof of Theorem \ref{th:hfrlam}
shows that $n_\eps$ is of order $1/\eps$.
To avoid this limitation, we were conducted to
consider more involved algorithms as described in 
the following two sections. 

\subsection{A min-max estimator for robust estimation} 
\label{sec:minmax}

This section provides an alternative to the empirical risk minimizer 
with non asymptotic exponential risk deviations of order $d/n$ for 
any confidence level.
Moreover, we will assume only a second order moment condition 
on the output and cover the case of unbounded inputs, the requirement on the random variables $g_j(X)$
being only a finite fourth order moment. On the other hand, we assume 
that the set $\Theta$ of the vectors of coefficients is bounded.
(This still allows to solve problem \pbl\ as soon as we know a bounded set in
which $\gl$ lies for sure.)

Let $\alpha>0$ and consider the truncation function:
$$
T(x) = \begin{cases} 
- \log \bigl(1 - x + x^2/2 \bigr) & 0 \leq x \leq 1, \\
\log(2) & x \geq 1, \\ 
- T(-x) & x \leq 0, 
\end{cases} 
$$
For any $g,g'\in\G$, introduce 
  $$
  \cD(g,g')= \sum_{i=1}^n T\Big(\alpha\big[Y_i-g(X_i)\big]^2-\alpha\big[Y_i-g'(X_i)\big]^2\Big).
  $$
Let us assume in this section that 
for any $j\in\{1,\dots,d\}$, 
  \beglab{eq:as1}
  \E\big\{g_j(X)^2[Y-g^*(X)]^2\big\}<+\infty,
  \endlab
and
  \beglab{eq:as2}
  \E\big[g_j^4(X)\big]<+\infty.
  \endlab

Define 
  \begin{align} 
  \cS & = \{ g\in\Span\{g_1,\dots,g_d\}: \E[g(X)^2]=1\}, \label{eq:cs}\\
\sigma & = \sqrt{\E\big\{[Y-g^*(X)]^2\big\}}= \sqrt{R(g^*)} , \\ 
\chi & = \max_{g\in\cS} \sqrt{\E [g(X)^4] },\\ 
\kappa & = \frac{ \sqrt{\E\big\{[\vp(X)^TQ^{-1}\vp(X)]^2\big\}}}
  {\E\big[\vp(X)^TQ^{-1}\vp(X) \big]},\\ 
\kappa' & = \frac{\sqrt{\E\big\{[Y-g^*(X)]^4\big\}}}{\E\big\{[Y-g^*(X)]^2\big\}} 
  = \frac{\sqrt{\E\big\{[Y-g^*(X)]^4\big\}}}{\sigma^2},\\
\cR & = \max_{g',g''\in\G} \sqrt{\E\big\{[g'(X)-g''(X)]^2\big\}}. \label{eq:kapp}
  \end{align}

\begin{thm} \label{th:3.1}
Let us assume that \eqref{eq:as1} and \eqref{eq:as2} hold.
For some numerical constants $c$ and $c'$, 
for
  $$
  n > c \kappa \chi d,
  $$
by taking
  \beglab{eq:alpha}
  \alpha = \frac{1}{2 \chi \bigl[ 2 \sqrt{\kappa'} \sigma
  + \sqrt{\chi} \cR
  \bigr]^2} \biggl(1  - \frac{c \kappa \chi d}{n} \biggr), 
  \endlab
for any estimator $\hg$ satisfying $\hg\in\G$ a.s., for any $\eps>0$,  
with probability at least $1-\eps$, we have
  \begin{align*}  
  R(\hg) - R(g^*) \le 
& \frac1{n\alpha}\bigg( \und{\max}{g'\in\G} \cD(\hg,g')
    - \und{\inf}{g\in{\G}} \und{\max}{g'\in{\G}} \cD(g,g')  \bigg) \\
& \qquad + \frac{c \kappa \kappa' d \sigma^2}{n}  + \frac{8 \chi \bigl( \frac{\log(\epsilon^{-1})}{n} + 
\frac{c' \kappa^2 d^2}{n^2} \bigr) \bigl[ 2 \sqrt{\kappa'}\sigma 
+ \sqrt{\chi} \cR \bigr]^2}{ 1 - \frac{c\kappa \chi d}{n} }.
  \end{align*}  
\end{thm}

The above theorem suggest to look for function realizing the min-max of $(g,g')\mapsto\cD(g,g')$.
More precisely, an estimator such that 
  $$\und{\max}{g'\in\G} \cD(\hg,g') < \und{\inf}{g\in{\G}} \und{\max}{g'\in{\G}} \cD(g,g') + \sigma^2 \frac{d}n,$$
has a non asymptotic bound for the excess 
risk with a $d/n$ convergence rate and an exponential tail even 
when neither the output $Y$ nor the input vector $\vp(X)$ has exponential moments. This stronger non asymptotic bound compared to the bounds of the previous section comes 
at the price of replacing the empirical risk
minimizer by a more involved estimator. Nevertheless,
reasonable heuristics can be developed to compute it approximately 
\cite[Section 3]{Aud09b}, and leads to a significantly better estimator 
of $\gl$ than the ordinary least squares estimator when there is some heavy-tailed noise (see Appendix \ref{app:exper}).

\subsection{A simple tight risk bound for a sophisticated PAC-Bayes algorithm} \label{sec:sophis}

A disadvantage of the min-max estimator proposed in the previous section
is that its theoretical guarantee depends (implicitly) on kurtosis like coefficients.
We provide in \cite[Section 4]{Aud09b}
a more sophisticated estimator, 
having the following simple excess risk bound independent of these 
kurtosis like quantities, and still of order $\frac{d}n$.
It holds under stronger assumption on the input vector $\vp(X)$ 
(precisely, uniform boundedness), still assumes that the set $\Theta$ is 
bounded, and holds under a second order moment condition on the output.

\begin{thm} \label{th:sophis}
Assume that $\G$ has a diameter $H$ for $L^\infty$-norm:
	\beglab{eq:hh}
	\sup_{g',g''\in \G,x\in\X} |g'(x)-g''(x)| = H
	\endlab
and that, for some $\sigma>0$, 
	\[
	\sup_{x\in\X} \E\big\{[Y-g^*(X)]^2 \big| X=x\big\} \le \sigma^2 < +\infty.
	\]
There exists an estimator $\hg$ such that for any $\eps>0$, with probability 
at least $1-\eps$, we have
	\[
	R( \hg ) - R(g^*) \le 17 (2\sigma+H)^2 \, \frac{d 
    		+ \log(2\eps^{-1})}{n}.	
	\]
\end{thm}

On the negative side, when the target is to solve problem \pbl, it requires the knowledge 
of a $L^\infty$-bounded ball in which $\flin$ lies and an upper bound
on $\sup_{x\in\X} \E\big\{[Y-\flin(X)]^2 \big| X=x\big\}$.
The looser this knowledge is, the bigger the constant in front of $d/n$ is.
On the positive side, the convergence rate is of order $d/n$, without
neither extra logarithmic factor, nor constant factors involving the conditioning
of the Gram matrix $Q$ or some Kurtosis like coefficients.

To conclude this section, let us add that,
when the output admits uniformly bounded conditional exponential moments,
 a relatively simple Gibbs estimator also achieves the $d/n$ convergence rate.
Precisely we have the following theorem.

\begin{thm} \label{th:v1c}
Assume that \eqref{eq:hh} holds for $H<+\infty$, and that there exist $\alpha > 0$ and $M>0$
such that for any $x\in\X$,
            $$
            \E \big(e^{\alpha|Y-\gl(X)|}\big|X=x\big) \leq M.
            $$
Consider the probability distribution $\hpig$ on $\G$ defined by its density
with respect to the uniform distribution $\pi$ on $\G$:
	\[
	\frac{\hpig}{\pi} (g)  
    =\frac{e^{ -\lam \sum_{i=1}^n [Y_i-g(X_i)]^2}}
		{\E_{g'\sim\pi} e^{ -\lam \sum_{i=1}^n [Y_i-g'(X_i)]^2}},
	\]
where $\lam>0$ is appropriately chosen (depending on $\alpha$, $H$ and $M$).
For any $\eps>0$, with probability 
at least $1-\eps$, we have 
	\[
	R( \E_{g\sim \hpig} g ) - R(g^*) \le C \, \frac{d 
    		+ \log(2\eps^{-1})}{n},
	\]
where the quantity $C>0$ only depends on $\alpha$, $H$ and $M$.
\end{thm}

\section{High-dimensional input and sparsity}

From the minimax rates of the three aggregation problems, 
we see that for $n\ll d\ll e^n$, 
one can predict as well as the best
convex combination up to a small additive term, which
is at most of order $\sqrt{\frac{\log d}n}$,
but one cannot expect to predict in general as well as the best
linear combination up to a small additive term.
In this setting, one may want to reduce its target by trying 
to predict as well as (still up to a small additive term) the best linear combination 
of at most $s\ll d$ functions, that is the function
  \beglab{eq:newg}
  g^*\in\und{\argmin}{g\in\{\sum_{j=1}^d \th_j g_j;\th_1\in\R,\dots,
    \th_d\in\R,\sum_{j=1}^d {\dsone}_{\th_j\neq 0} \le s\}} R(g).
  \endlab
It is well-established that $L^1$ regularization allows to perform
this task. The procedure is known as Lasso \cite{Tib94,Osb00} 
and is defined by $\hfllam= \langle\thllam,\vp\rangle$ with
	\[
	\thllam \in \und{\argmin}{\th\in\R^d} 
		\frac{1}{n}\sum_{i=1}^n \big( Y_i-\langle\th,\vp(X_i) \rangle \big)^2 
		  + \lam \|\theta\|_1,
	\]
where $\lam>0$ is a parameter to be tuned to retrieve the desired number of relevant 
variables/functions\footnote{
The functions $g_1,\dots,g_d$ can be called the explanatory variables of the output.
Note also that we can consider without loss of generality that the input space is $\R^d$ and that the functions $g_1,\dots,g_d$ are the coordinate functions.}.
As the $L^2$ penalty used in ridge regression, the $L^1$ penalty shrinks the coefficients. The difference is that for coefficients 
which tend to be close to zero, the shrinkage makes them equal to zero. This allows to select relevant 
variables/functions (i.e., find the $j$'s such that $\th^*_j\neq 0$). 

If we assume that the regression function $\greg$ is a linear combination of only $s \ll d$ variables 
among $\{g_1,\dots,g_d\}$,
the typical result is to prove that the expected excess risk 
of the Lasso estimator for $\lam$ of order $\sqrt{\fracl{\log d}{n}}$
is of order $\fracl{s \log d}{n}$ \cite{Bun07,van08,Mei09,Lou08}. 
Since this quantity is much smaller than $d/n$, this makes a huge improvement (provided that
the sparsity assumption is true). This kind of results usually requires strong conditions on the
eigenvalues of submatrices of $Q$, essentially assuming that the functions 
$g_j$ are near orthogonal. Here we will argue
that by combining the estimators solving \pbms\ and \pbl, one can achieve minimax optimal
learning rate without requiring such conditions. The guarantees presented here are also stronger
than the ones associated with $L_0$-regularization (penalization proportional to the number of nonzero coefficient) whatever criterion (Mallows' $C_p$ \cite{Mal73}, AIC \cite{Aka73} or BIC \cite{Sch78}) is used to tune the penalty constant. 
Recent advances on theoretical guarantees of $L_0$-regularization can be found in the works of Bunea, Tsybakov and Wegkamp \cite{Bun07}
and of Birgé and Massart \cite{Bir07} for the fixed design setting and in the work of Raskutti, Wainwright and Yu
\cite{Ras09} for the random design setting considered here. These results for $L_0$-regularization are not as good for the ones for the estimator described in this section since the $(s \log d) / n$ excess risk bound holds only when the conditional expectation of the output knowing the input is inside the model.

%

Precisely, let us assume\footnote{We make boundedness assumptions for sake of simplicity. The
results can be generalized to outputs having exponential conditional moments since
both building blocks of the estimator can handle this type of noisy outputs: for the empirical star algorithm,
see the supplemental material of \cite{Aud07b}. Further generalizations are open problems.} that for some $B>0$,
$\|g^*\|_\infty\le B$ and $|Y| \le B$.
Let $\cL_1$ denote the first half of the training set $\{Z_1,\dots,Z_{n/2}\},$ 
and $\cL_2$ denote the second half of the training set $\{Z_{n/2+1},\dots,Z_{n}\},$ 
where for simplicity we have assumed that $n$ is even.
For any $I\subset\{1,\dots,d\}$ of size $s$, let $\hg_I$
be the sophisticated estimator that satisfies Theorem \ref{th:sophis} 
\emph{trained on $\cL_1$} and associated with the set 
$\G_I = \big\{\langle \th,\vp\rangle: \|\langle \th,\vp\rangle\|_\infty \le B,
\th_j=0 \text{ for any }j\notin I \big\}.$ (One can alternatively consider the Gibbs estimator of Theorem \ref{th:v1c}.)
Let $\hg$ be the empirical star estimator (defined in Section \ref{sec:star})
\emph{trained on $\cL_2$} and associated with the $\binom{d}{s}$ functions $\hg_I$
(that are non-random given $\cL_1$).
This two-stage estimator satisfies the following theorem.

\begin{thm}
For any $\eps>0$, with probability at least $1-\eps$,
  \beglab{eq:sparse}
  R(\hg) - R(g^* ) \le C B^2 \frac{s \log(d/s)+\log(2\eps^{-1})}{n},
  \endlab 
for some numerical constant $C>0$.  
\end{thm}
\begin{proof}
From Theorem \ref{th:star}, since we have $\binom{d}{s}\le \big(\fracc{ed}s\big)^{s}$, 
with probability at least $1-\eps/2$,
we have
  $$
  R(\hg) - \min_{I\subset\{1,\dots,d\}:|I|=s} R(\hg_I)
    \le 1200 B^2 \frac{s\log(ed/s)+\log(2\eps^{-1})}n.
  $$
Let $I^*$ be a set of $s$ variables containing the set of at most 
$s$ variables involved in $g^*$. 
From Theorem \ref{th:sophis}, with probability at least $1-\eps/2$,
we have
  $$
  R(\hg_{I^*})- R(g^*) \le 1224 B^2 \frac{s + \log(4\eps^{-1})}n.
  $$
By using an union bound, we obtain 
  $$
  R(\hg) - R(g^* ) \le 1224 B^2 \bigg( \frac{s\log(ed/s)+\log(2\eps^{-1})}n + \frac{s +\log(4\eps^{-1})}{n} \bigg).
  $$
which gives the desired result.  
\end{proof}

Due to the particular structure of the empirical star algorithm,
the estimator $\hg$ can be written as a linear combination 
of at most $2s$ functions among $\{g_1,\dots,g_d\}$,
so that the estimator can be used for variable selection.
The functions involved in $g^*$ do not necessarily
belong to this set of at most $2s$ functions. I do not believe
that achieving such identifiability of these particular relevant variables
should be the goal, since pursuing this target would definitely require that the
different variables are not too much correlated, a situation
which will rarely occurs in practice. 

Adaptivity with respect to the sparsity level of $g^*$ can
also be obtained. Indeed, let $s^*$ be the number of nonzero coefficients of 
the function $g^*$ defined by \eqref{eq:newg}. By using 
a three-stage estimator procedure  
using successively the empirical star algorithm
at a given level of sparsity and then on the $s$ functions thus designed, it is 
easy that \eqref{eq:sparse} still holds with $s$ replaced by $s^*$.
Note that $g^*$ defined in \eqref{eq:newg} depends on a sparsity level $s$, 
which can be taken equal to $d$. Then we have $g^*=\gl$, and 
the three-stage procedure is adaptive to the sparsity level of $\gl$.
In the fixed design setting, Bunea, Tsybakov and Wegkamp \cite{Bun07}
have shown that these rates are minimax optimal, and it is natural
to consider that their lower bound extends to our random design case.

Another possible use of the algorithms solving problems \pbms\ and \pbl\ is 
when we consider sparsity with group structure. This occurs when
the variables are naturally organized into groups: in computer vision,
this naturally occurs since there exist different families of image descriptors,
and the grouping can be done by family, scale and/or position.
Let $I_1,\dots,I_D \subset \{1,\dots,d\}$ be $D$ sets of grouped variables.
For a vector $\th$, let us say that a group $I_k$ is active if there exists $j\in I_k$
such that $\th_j\neq0$. let $S(\th)$ be the number of active groups among $I_1,\dots,I_D$.
  
For a given sparsity level $s\in\{1,\dots,D\}$, the target is
  \beglabc{eq}
  \ggroup\in\und{\argmin}{g\in\{\langle \th, \vp\rangle;\th\in\R^d,S(\th)\le s\}} R(g).
  \endlabc
There exist only $\binom{D}{s}$ different sets of $s$ groups that could be active.
So a two-stage estimator $\hggroup$ similar to the one described before satisfies
that with probability at least $1-\eps$, 
  $$
  R(\hggroup) - R(\ggroup) \le C B^2 \frac{s \log(D/s)+J+\log(2\eps^{-1})}{n},
  $$
where $J$ denotes the number of nonzero coefficients in the linear combination defining $\ggroup$. 
This type of results has not been obtained yet for the group Lasso \cite{Yua06} even when assuming
low correlation between the variables, except for the fixed design setting \cite{Hua09,LouPon09}.

We have presented in this section an example of theoretical results easily obtainable from
the estimators solving problems \pbms\ and \pbl. The results are expressed
in terms of sub-exponential excess risk bounds, which were not obtainable
before the introduction of the empirical star algorithm.
An advantage of the approach is its genericity: it is not
restricted to particular families of estimators.

There are yet some limitations. First, there is no 
variable selection consistency with this approach, but
as stated before, this stronger type of results would require strong
assumptions on the input vector distribution, that are often not met in practice.
In the fixed design setting, for overlapping groups, Jenatton, Bach and I \cite{Jen09} 
have proved a high dimensional variable consistency result extending the corresponding result for the Lasso \cite{ZhaoYu06,Wai09}.

Second, the approach does not extend easily to the case of generalized additive models,
in which linear combinations of a fixed number of functions are replaced by
functional spaces \cite{MeiVanBuh09}, such as reproducing kernel Hilbert spaces in the cases of
multiple kernel learning \cite{Lan04,Bach04,Ong05,Mic06,Bach08,Kol08}.

Finally, the most important limitation, which is often encountered when using classical model selection approach, is its computational intractability. So this leaves open the following fundamental problem: is it possible to design a computationally efficient algorithm with the above guarantees (i.e., without assuming low correlation between the explanatory variables)?

\chapter{Multi-armed bandit problems}

\section{Introduction}

Bandit problems illustrate the fundamental difficulty of decision making in the face of uncertainty:
a decision maker must choose between following what seems to be the best choice in view of the past (``exploiting'') or 
testing (``exploring'') some alternative,
hoping to discover a choice that beats the current best choice.
More precisely, in the multi-armed bandit problem, at each stage,
an agent (or decision maker) chooses one action (or arm), and receives a reward from it.
The agent aims at maximizing his rewards. Since he does not know the process
generating the rewards, he needs to explore (try) the different actions and yet, exploit (concentrate its draws
on) the seemingly most rewarding arms.

The multi-armed bandit problem is the simplest setting where one encounters the 
exploration-exploitation dilemma. It has a wide range of applications
including advertizement \cite{Baba09,Dev09}, economics \cite{BerVal08,LamPagTar04}, games \cite{Gel06} and
optimization \cite{Kle05,Coq07,Kle08,Bub08}.
It can be a central building block of larger systems, like in evolutionary programming \cite{Hol92} and reinforcement learning \cite{Sut98}, in particular in large state space Markovian Decision Problems \cite{KocSze06}.
The name ``bandit'' comes from imagining a gambler in a casino playing with $K$ slot machines, where at each round, 
the gambler pulls the arm of any of the machines and gets a payoff as a result.
The seminal work of Robbins \cite{ro52} casts the bandit problem in a stochastic setting
in which essentially the rewards obtained from an arm are independent and identically distributed random variables 
that are also independent from the rewards obtained from the other arms.
Since the work of Auer, Cesa-Bianchi, Freund and Schapire \cite{ACFS03}, it was also studied in an adversarial setting. 

To set the notation, let $K\ge 2$ be the number of actions (or arms) and $n\ge K$ be the time horizon.
A $K$-armed bandit problem is a game between an agent and an environment in which, at each time step $t \in \{1,\dots,n\}$,
(i) the agent chooses a probability distribution $p_t$ on a finite set $\{1,\dots,K\}$,
(ii) the environment chooses a reward vector $g_t=(g_{1,t},\dots,g_{K,t}) \in [0,1]^K$
(possibly through some external randomization),
and simultaneously (independently), the agent draws the arm $I_t$ according to the distribution $p_t$,
(iii) the agent only gets to see his own reward $g_{I_t,t}$.
The goal of the decision maker is to maximize his cumulative reward $\sum_{t=1}^n g_{I_t,t}$.

In the stochastic bandit problem, the environment cannot choose any reward vectors:
the reward vectors $g_t$ have to be independent and identically distributed, and its components 
should be independent random variables\footnote{The independence 
of the components is always made in the literature, but is not fundamentally useful 
(up to rare modifications of the numerical constants).}.  
So an environment is just parameterized by a $K$-tuple 
of probability distributions $(\nu_1,\hdots,\nu_K)$ on $[0,1]$. 
Note that the term ``stochastic bandit'' can be a bit misleading
since the assumption is not just stochasticity but rather an i.i.d. assumption.

In the adversarial bandit problem, no such restriction is put
so that past gains have no reason to be representative of future ones. This contrasts with 
the stochastic setting in which confidence bounds on the mean reward of the arms can be 
deduced from the rewards obtained so far.

A policy is a strategy for choosing the drawing probability distribution $p_t$ based on the history formed by the past plays and the associated rewards.
So it is a function that maps any history to a probability distribution on $\{1,\dots,K\}$. 
We define the regret of a policy with respect to the best constant decision as
  \beglab{eq:regret}
  R_n = \max_{i=1,\dots,K} \sum_{t=1}^n \big( g_{i,t} - g_{I_t,t} \big). 
  \endlab

To compare to the best constant decision is a reasonable target
since it is well-known that (i) there exist randomized policies
ensuring that $\E R_n/n$ tends to zero as $n$ goes to infinity,
(ii)~this convergence property would not hold if the maximum and the sum would be inverted in the definition of $R_n$.
This chapter will first present my contributions to the stochastic bandit problems, essentially:
  \begin{itemize}
  \item how to use empirical variance estimates in upper confidence based policies? (Section \ref{sec:4})
  \item how thin is the tail distribution of the regret of standard policies, and how can we improve it? (Section \ref{sec:5})
  \item provide a minimax optimal policy (Section \ref{sec:6}), 
  \item propose a model and an arm-increasing rule to deal with bandit problems with more arms than draws: $K\ge n$
    (Section \ref{sec:7}),
  \item design and use a Bernstein's bound with estimated variances to have better stopping rules
    (Section \ref{sec:8}),
  \item provide a policy to identify the best arm at the end of the $n$ time steps (Section \ref{sec:9}).
  \end{itemize}
Sébastien Bubeck and I \cite{AudBub10} contribute to the adversarial setting
by designing a new type of weighted average forecaster characterized
by an implicit normalization of the weights, and for which a new type of
analysis can be developed. The advantage of the policy and the analysis
is that it allows to bridge the long open logarithmic gap
in the characterization of the minimax rate for the multi-armed bandit problem,
and to have a common framework for addressing other sequential prediction problems
(full information, label efficient, tracking the best expert) (Section \ref{sec:adv}).

\section{The stochastic bandit problem}

\subsection{Notation}

Let $T_i(t)$ denote the number of times arm $i$ is chosen by the policy during the first $t$ plays.
Define $\mu_i=\int x \nu_i(dx)$ the expectation and $V_i=\int (x-\mu_i)^2 \nu_i(dx)$ the variance
of the distribution $\nu_i$ characterizing arm $i$. Let $i^*\in\argmin_{i\in\{1,\dots,K\}} \mu_i$ denote an index of an optimal arm. 
The suboptimality of an arm $i$ is measured by:
  $$
  \Delta_i= \max_{j=1,\dots,K} \mu_j - \mu_i = \mu_{i^*} - \mu_i.
  $$
Let $X_{i,t}$ be the $t$-th reward obtained from arm $i$ if $T_i(n)\ge t$, and for $t>T_i(n)$, let $X_{i,t}$ be other
independent realizations of $\nu_i$.
For any $i\in\{1,\dots,K\}$ and $s\in\N$, introduce $\oX_{i,s}$ and $\bV_{i,s}$ the empirical mean and variance 
of $X_{i,1},\dots,X_{i,s}$. 
    $$
    \oX_{i,s} \eqdef \frac{1}{s} \sum_{j=1}^s X_{i,j} \quad {\rm and} \quad
    \bV_{i,s} \eqdef \frac{1}{s} \sum_{j=1}^s (X_{i,j}-\oX_{i,s})^2.
    $$

\subsection{Regret notion}

Previous works in the stochastic bandit problem do not use the regret defined by \eqref{eq:regret}, which is a 
regret with respect to the best constant decision, but a (pseudo-)regret that compares the reward of the policy
to the reward of an optimal arm in expectation, that is $i^*\in\argmin_{i\in\{1,\dots,K\}} \mu_i$:
  $$
  \oR_n = \sum_{t=1}^n \big( g_{i^*,t} - g_{I_t,t} \big) \le R_n.
  $$
Results concerning this regret are easier to state, and we will follow hereafter the trend of 
previous works to state the results in terms of $\oR_n$. In this section, we gather results
showing how to go from an upper bound on $\oR_n$ to an upper bound on $R_n$.
The following lemma shows that logarithmic regret bounds on $\E \oR_n$ extend to logarithmic regret bounds on $\E R_n$ 
\emph{when the optimal arm is unique}, that is $\mu_i<\mu_{i^*}$ for any $i\neq i^*$.
Besides, unlike known upper bounds for $\E \oR_n$, the ones on $\E R_n$ depends on the variance $V_{i^*}$
of the reward distribution of the optimal arm.
(When there are several optimal arms, it is the smallest variance of the optimal arms distributions which appears in the 
expected regret bound.)
\begin{lemma}[\cite{AudBub10}]
For a given $\delta\ge 0$, let $I=\big\{i\in\{1,\dots,K\}:\Delta_i\le\delta\big\}$ be
the set of arms ``$\delta$-close'' to the optimal ones, and $J=\{1,\dots,K\}\setminus I$ the remaining set of arms.
In the stochastic bandit game, we have
$$\E R_n - \E \oR_n \leq \sqrt{\frac{n \log |I|}{2}}
  + \sum_{i\in J} \frac{1}{2 \Delta_i} \exp(- n \Delta_i^2),$$
and also
$$\E R_n - \E \oR_n \leq \sqrt{\frac{n \log |I|}{2}}
  + \sum_{i\in J} \frac{2V_{i^*}+2V_{i}+2 \Delta_i/3}{\Delta_i} 
    \exp\left(- \frac{n\Delta_i^2}{2V_{i^*}+2V_{i}+2\Delta_i/3}\right).$$

In particular when there exists a unique arm $i^*$ such that $\Delta_{i^*}=0,$
we have
$$\E R_n - \E \oR_n \leq 2 \sum_{i\neq i^*} \frac{V_{i^*}+V_{i}+\Delta_i/3}{\Delta_i},$$
and also for any $t> 0$
$$\P\big( R_n - \oR_n > t \big) \le \sum_{i\neq i^*} \exp\left(- \frac{(t+n \Delta_i)^2}{n\min(1,2V_{i^*}+2V_{i}+2(t/n+ \Delta_i)/3)}\right).$$
\end{lemma}

The uniqueness of the optimal arm is really needed to have logarithmic (in $n$) bounds on the expected regret.
This can be easily seen by considering a two-armed bandit in which both reward distributions are identical (and non degenerated).
In this case, the expected pseudo-regret is equal to zero while the expected regret will be at least of order $\sqrt{n}$ for any forecaster.
This reveals a fundamental difference between the expected regret and the pseudo-regret.

Previous works on stochastic bandits use the expected pseudo-regret criterion since it satisfies
  $$\E \oR_n = \sum_{i=1}^K \Delta_i \E T_i(n),$$
meaning that one has only to control the expected sampling times of suboptimal arms to
understand how the expected pseudo-regret behaves.

\subsection{Introduction to upper confidence bounds policies}

Early papers have studied stochastic bandit problems under Bayesian assumptions (e.g., Gittins \citealt{gittins89}).
On the contrary, Lai and Robbins \citet{LaiRo85} have considered a parametric minimax framework.
They have introduced an algorithm that follows what is now called the ``optimism in the face of uncertainty principle''.
At time $t\equiv k_t$ (mod $K$) with $k_t\in\{1,\dots,K\}$, their policy 
compares an {\em upper confidence bound} (UCB) of the mean reward $\mu_{k_t}$
of arm $k_t$ to a reasonable target defined as the highest empirical mean of ``sufficiently'' drawn arms.
If the upper confidence bound exceeds the target, arm $k_t$ is drawn, and otherwise, the arm
defining the reasonable target is drawn.
Lai and Robbins proved that the expected regret of this policy increases at most at a logarithmic rate with the number of trials
 and that the algorithm achieves the smallest possible regret up to some sub-logarithmic additive term (for the considered family of distributions).
Agrawal \citet{Agr95} proposed computationally easier UCB algorithms in a more general setting that have also logarithmic expected regret
(at the price of a higher numerical constant in the upper bound on the regret).
More recently, Auer, Cesa-Bianchi and Fischer \citet{AueCes02} have proposed even simpler policies achieving 
logarithmic regret \emph{uniformly over time} rather than just for a fixed number $n$ of rounds known in advance by the agent.
Besides, unlike previous works, they have provided non asymptotic bounds.

%
 
Upper confidence bounds policies can be described as follows. From time $1$ to $K$, draw each arm once.
At time $t\ge K+1$, draw the arm maximizing $B_{i,T_i(t-1),t}$,
where $B_{i,s,t}$ is a high probability bound on $\mu_i$ computed from
the i.i.d. sample $X_{i,1},\dots,X_{i,s}$. 
The confidence level of this high probability bound might depend on the current round $t$.
For instance, the UCB$1$ policy of Auer, Cesa-Bianchi and Fischer \citet{AueCes02} uses 
  $$
  B_{i,s,t}=\oX_{i,s}+\sqrt{ \frac{2 \log t}{s}},
  $$
which is an upper bound on $\mu_i$ holding with probability at least $1-t^{-4}$
according to Hoeffding's inequality. 

Auer, Cesa-Bianchi and Fischer \citet{AueCes02} also noted that plugging 
an upper confidence bound of the variance in the square root term performs
empirically substantially better than UCB$1$. Precisely, their experiments used
  \beglab{eq:aue}
  B_{i,s,t}=\oX_{i,s}+\sqrt{ \min\bigg(\bV_{i,s}+\sqrt{\frac{2\log t}s},\frac14 \bigg)\frac{\log t}{s}}.
  \endlab
My first contribution to the multi-armed bandit problem was to provide a 
theoretical justification of these empirical findings, as described in the following section.

\subsection{UCB policy with variance estimates} \label{sec:4}
 
Rémi Munos, Csaba Sze\-pesv\'ari and I \cite{AuMuSz09} have 
proposed the following slight modification of the arm indexes given by \eqref{eq:aue}:
  \beglab{eq:ucbv}
  B_{i,s,t}=\oX_{i,s}+\sqrt{ \frac{2\zeta\bV_{i,s}\log t}s}+\frac{3 \zeta \log t}s,
  \endlab
with $\zeta>1$. The associated policy achieves a logarithmic regret as UCB$1$ with a constant factor potentially
much smaller than the one of UCB$1$.
Indeed, from \citet{AueCes02}, UCB$1$ satisfies 
  \beglab{eq:ucbun}
  \E \oR_n \le \sum_{i:\Delta_i>0}\frac{10}{\Delta_i} \log n,
  \endlab
whereas our algorithm, called UCB-V (V for variance), satisfies for $\zeta>1$,
  \beglab{eq:czeta}
  \E \oR_n \le c_\zeta \sum_{i:\Delta_i>0} \bigg( \frac{\sigma_k^2}{\Delta_k} + 2 \bigg) \log n,
  \endlab
with $c_\zeta>0$ a function of $\zeta$ satisfying $c_{1.2}\le 10$ and
$c_\zeta\le C \Big(\sum_{t=1}^{+\infty} t^{-\zeta} + \zeta \Big)$
for some numerical constant $C>0$. 
%
We also proved that for specific distributions of the rewards,
UCB-V with $\zeta<1$ suffers a polynomial expected (pseudo-)regret, that is $\E \oR_n \ge C n^C$
for some $C>0$.
The argument proving this later assertion also implies that
using exactly the upper bound \eqref{eq:aue} 
can dramatically fail in some specific situations\footnote{For instance,
when the optimal arm 
concentrates its rewards on $0$ and $1$ (Bernoulli distribution with parameter $1/2$), and when the other arms 
always provide a reward equal to $1/2-1/n^{1/6}$, the expected regret is lower bounded by 
$C n^{1/7}$.}.

\subsection{Deviation of the regret of UCB policies} \label{sec:5}

In this section, we consider that there is a unique optimal arm $i^*$. 
In \cite{AuMuSz09}, we show that the UCB-V policy defined by \eqref{eq:ucbv} satisfies
  \beglab{eq:noconc}
  \P(\oR_n \ge C \log n) \le \bigg(\frac{C'}{\log n}\bigg)^{\zeta/2}.
  \endlab
for quantities $C$ and $C'$ depending on $K, \zeta, \sigma_1,\ldots,\sigma_K, \Delta_1,\ldots,\Delta_K$, but not on $n$.
The ``polynomial'' rate in \eqref{eq:noconc} is not due to the looseness of the bound.
It can be shown that as soon as the essential infimum 
of the optimal arm's distribution $\tmu = \sup\{v\in\R: \nu_{i^*}([0,v) ) = 0 \}$
is smaller than the mean reward of the second best arm, the pseudo-regret admits a polynomial tail only: there exists $C'>0$ (depending on the distributions $\nu_1,\dots,\nu_K$)
such that for any $C>0$, there exists $n_0>0$ such that for any $n\ge n_0$, 
$\P\big( \oR_n > C \log n\big) \ge \big(\frac{1}{C'C\log n}\big)^{C'}$. In particular,
there is no positive quantities $C,C'$ for which for any $n$, we have\footnote{
An entirely analogous result holds for UCB1: using the variance estimates or not does not change the form of the
tail distribution of the regret.} $\P\big( \oR_n > C \log n\big) \le \frac{C'}n$.

The regret concentration, although it improves as $\zeta$ grows, is thus pretty slow.
The slow concentration happens when the first draws $\Omega$ of the optimal arm are unlucky (yielding small rewards) 
in which case the optimal arm will not be selected any more during roughly the first $e^{\Omega}$ steps. 
As a result, the distribution of the regret can be seen as a mixture of a peaky mode corresponding
to situations in which the optimal arm has a ``normal'' behaviour (with small variations due to 
the suboptimal arms) and a very thick-tailed mode corresponding to the unlucky start described above.
Our theoretical study shows that the mass of this mode decays only at a polynomial rate controlled by $\zeta$.
Recall that the larger $\zeta$ is, the more all arms are explored, the larger the bound on the expected regret is (see 
\eqref{eq:czeta}). In our experiments, this mode does appear (see Figure \ref{fig:graph}).

\begin{figure}[thbp]
\begin{center}
  \includegraphics[width=0.47\textwidth]{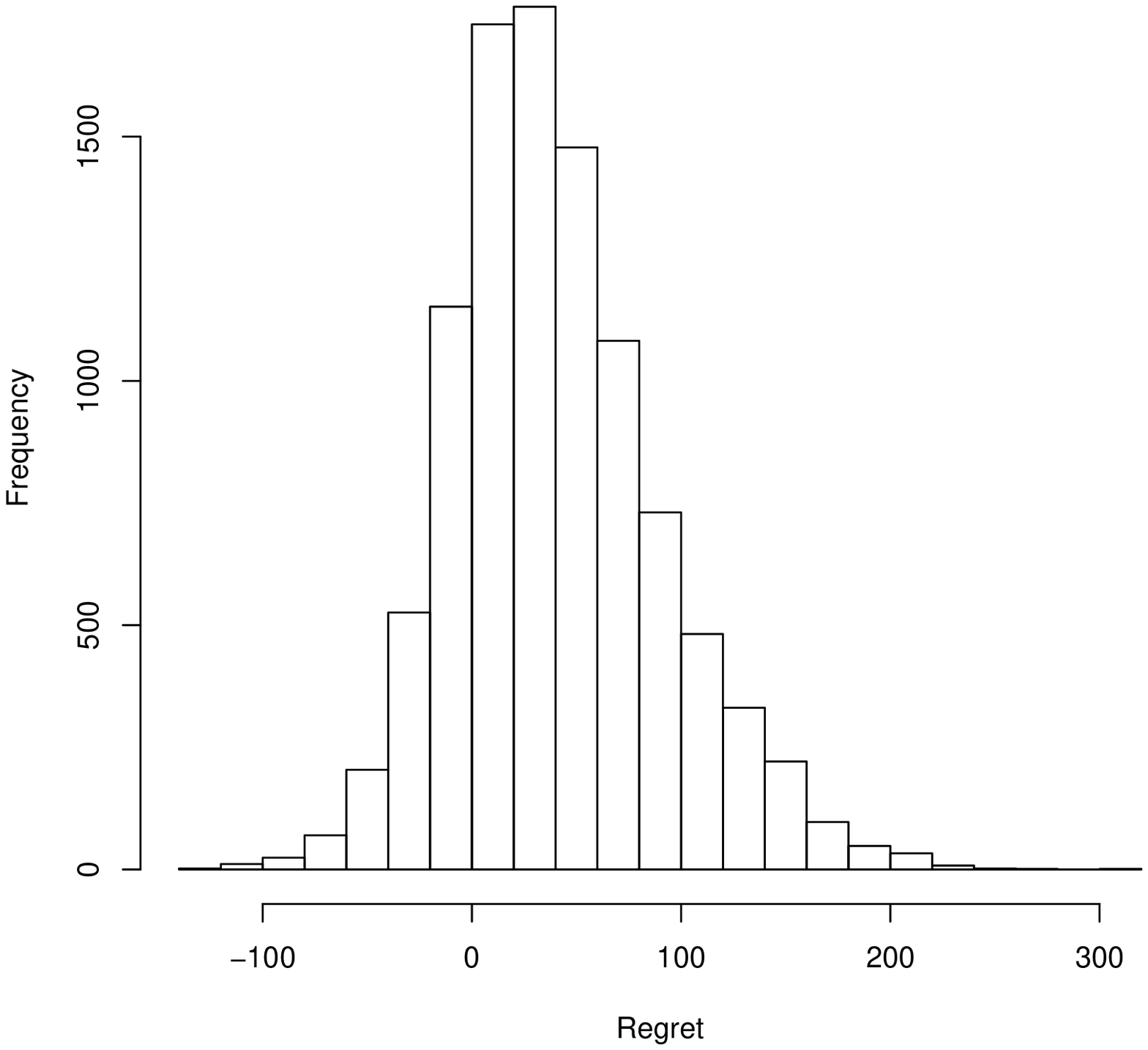}
  \includegraphics[width=0.47\textwidth]{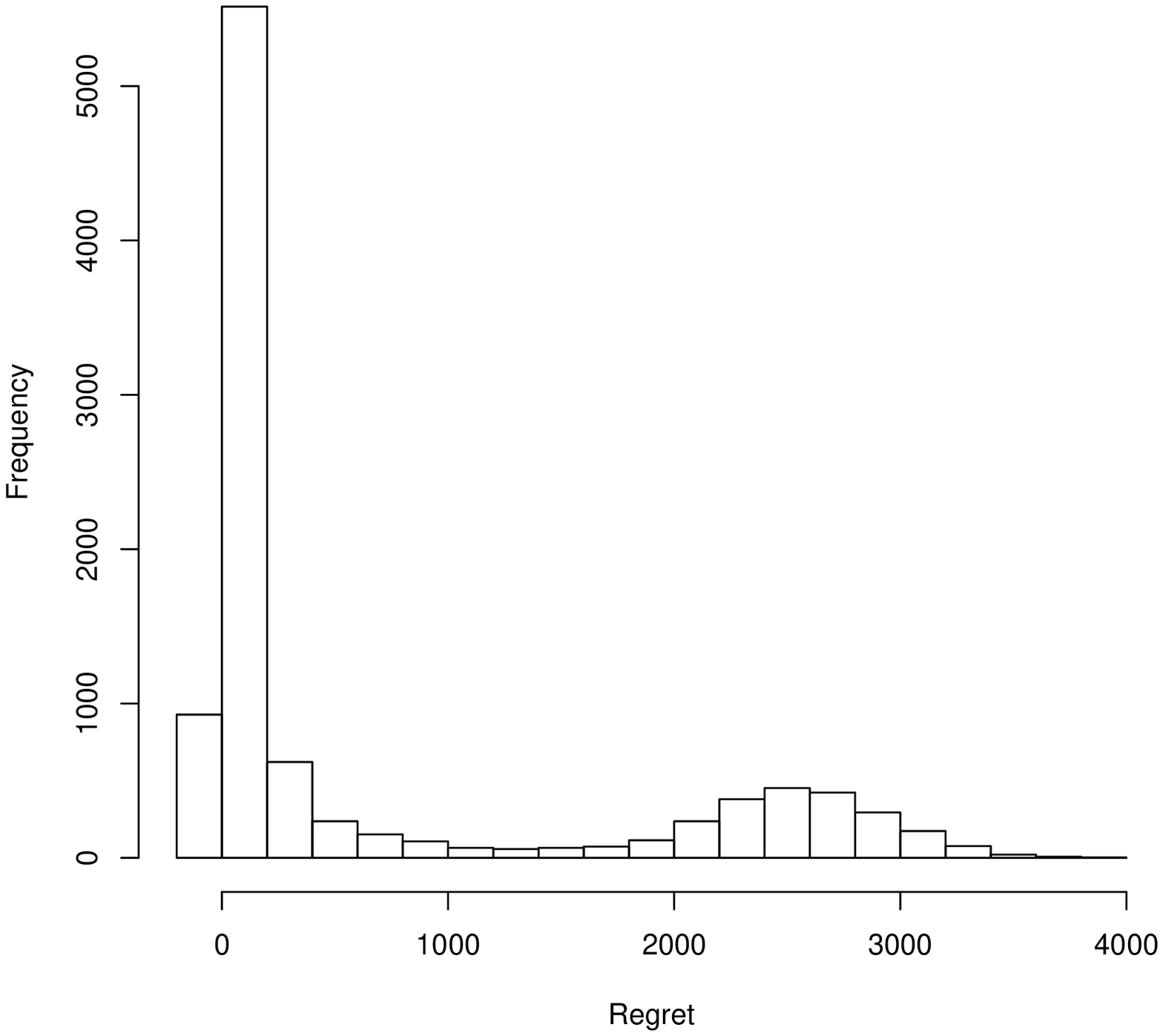}
\end{center}
\vs{-.9}
\caption{Distribution of the pseudo-regret for UCB-V ($\zeta=1$) 
for horizon $n=16,384$ (l.h.s. figure) and $n=524,288$ (r.h.s. figure).
The bandit problem is defined by $K=2$, a Bernoulli distribution with parameter $0.5$ and a Dirac distribution at $0.495$.}
\label{fig:graph}
\end{figure}

When the time horizon  $n$ is known, one may consider the UCB policy with
  \beglab{eq:logn}
  B_{i,s,t}=\oX_{i,s}+\sqrt{ \frac{6\bV_{i,s}\log n}s}+\frac{9 \log n}s,
  \endlab
which is an upper bound on $\mu_i$ which holds with probability at least $1-n^{-3}.$
The associated UCB policy, called hereafter UCB-Horizon, concentrates its exploration phase at the beginning of the plays,
and then switches to the exploitation mode. 
On the contrary, the UCB-V induced by \eqref{eq:ucbv}, which looks deceptively similar to UCB-Horizon (with $\zeta=3$),
explores and exploits at any time during the interval $\clint{1}{n}$.
Both policies have similar guarantee on their expected regret. However, 
on the one hand, UCB-Horizon
always satisfies
\beglab{eq:surprise}
    \P\big( \oR_n > C \log n\big) \le \frac{C'}{n},
    \endlab
where $C$ and $C'$ are quantities depending only on $K, \zeta, \sigma_1,\ldots,\sigma_K, \Delta_1,\ldots,\Delta_K$,
which contrasts with the significantly worse tail distribution of UCB-V.
On the other hand, unlike UCB-Horizon, UCB-V has the anytime property: the policy satisfies
the logarithmic expected regret bound for any time horizon $n$ (since its pulling 
strategy does not depend on the time horizon).
The open question here is thus: could we have both properties? In other words, is there an algorithm that does not need to know the time horizon
and which regret has a tail distribution satisfying~\eqref{eq:surprise}?
We conjecture that the answer is no.

%
%
%



\subsection{Distribution-free optimal UCB policy} \label{sec:6}

The inequalities \eqref{eq:ucbun} and \eqref{eq:czeta} may have surprised the reader
since the right-hand sides diverge for $\Delta_i$ going to $0$.
For $\Delta_i=\text{o}(n^{-1/2})$, this is an artefact of the bounds, which is easily rectifiable.
For instance, for UCB$1$, the more general bound (but less readable one) is 
  \[
  \E \oR_n \le \max_{t_i\ge 0, \sum_i t_i=n} \sum_{i:\Delta_i>0} \min\bigg(\frac{10}{\Delta_i}\log n\,, \ t_i \Delta_i\bigg).
  \]
In the worst case \big(i.e., $\Delta_1=0$ and $\Delta_2=\cdots=\Delta_K=\sqrt{10 K (\log n)/n}$\big), the right-hand side
of the bound is equal to $\sqrt{10 n(K-1) \log n}$. 
This has to be compared with the following lower bound of Auer, Cesa-Bianchi, Freund and Schapire 
\cite{ACFS03}:
  $$
  \inf \sup \E \oR_n \geq \frac{1}{20} \sqrt{n K},
  $$
where the infimum is taken over all policies and the supremum is taken over all 
$K$-tuple of probability distributions on $[0,1]$. 
We thus observe a logarithmic gap.
In \cite{AudBub09,AudBub10}, Sébastien Bubeck and I close this logarithmic gap,
by using a different UCB policy based on
  $$
  B_{i,s,t}=\oX_{i,s}+\sqrt{ \frac{\log\max(\frac{n}{Ks},1)}{s}},
  $$
which, for $s< n/K$, is an upper bound on $\mu_i$ which holds with probability at least $1-(Ks/n)^{-2}$
according to Hoeffding's inequality. In this policy,
an arm that has been drawn more than $n/K$ times has an index equal 
to the empirical mean of the rewards obtained from the arm, and when it has been drawn close to $n/K$ times, the logarithmic
term is much smaller than the one of UCB$1$, implying less exploration of this already intensively drawn arm.
For this policy, we prove
\begin{thm} \label{th:mossb}
For $\Delta= \und{\min}{i\in\{1,\dots,K\}:\Delta_i>0}\, \Delta_i$, the above policy satisfies
  \beglab{eq:beaua}
  \oR_n \le \frac{23 K}{\Delta} \log\left( \max\left(\frac{110n\Delta^2}{K} , 10^4 \right)\right),
  \endlab
and
  \beglab{eq:beauc}
  \E \oR_n \le 24 \sqrt{nK}.
  \endlab 
\end{thm}

This means that this UCB policy has the minimax rate $\sqrt{nK}$, while 
still having a distribution-dependent bound increasing logarithmically
in $n$.

\subsection{UCB policy with an infinite number of arms} \label{sec:7}

When the number of arms is infinite (or larger than the available number of experiments), 
the exploration of all the arms is impossible: if no additional assumption is made, it may be arbitrarily hard to 
find a near-optimal arm. In \cite{Wang08}, Yizao Wang, Rémi Munos and I consider a stochastic assumption on the \textit{mean-reward} of any new selected arm. 
When a new arm $i$ is pulled, its mean-reward $\mu_i$ is assumed to be an independent sample from a fixed distribution. 
Our assumptions essentially characterize the probability of pulling near-optimal arms. 
That is, given $\mu^*\in[0,1]$ as the best possible mean-reward and $\beta\geq0$ 
a parameter of the mean-reward distribution, the probability that a new arm is 
$\delta$-optimal is of order $\delta^\beta$ for small $\delta$,
 i.e. $\P(\mu_k\geq\mu^*-\delta) =\Theta(\delta^\beta)$ for $\delta\rightarrow 0$\footnote{
 We write 
 $f(\delta)=\Theta(g(\delta))$ for $\delta\rightarrow 0$ when $\exists c_1,c_2,\eps_0>0$ such that $\forall \delta\leq\delta_0$,
$c_1 g(\delta) \leq f(\delta)\leq c_2 g(\delta)$.}. 
In contrast with the previous many-armed bandits \cite{BCZHS97,teytaud07anytime}, our setting allows 
general reward distributions for the arms, under a simple assumption on the mean-reward. 

When there is more arms than the available number of experiments, the 
exploration takes two forms: 
discovery (pulling a new arm that has never been tried before) 
and sampling (pulling an arm already discovered in order to gain information about its actual mean-reward).

Numerous applications can be found e.g. in \cite{BCZHS97}. 
It includes labor markets (a worker has many opportunities for jobs), mining for valuable resources (such as gold or oil) when there are many areas available for exploration (the miner can move to another location or continue in the same location, depending on results), and 
path planning under uncertainty in which the path planner 
has to decide among a route that has proved to be efficient in the past (exploitation), or a known route that has not been explored many times (sampling), or a brand new route that has never been tried before (discovery).

In \cite{Wang08}, we propose an arm-increasing rule policy. It has the anytime property and 
consists in adding a new arm from time to time into the set of sampled arms.
It is done such that at time $t$, the number of sampled arms is of order $n^{\beta/2}$ if $\mu^*<1$ 
and $\beta<1$, and of order $n^{\beta/(1+\beta)}$ otherwise. It uses a modified version of the UCB-V
policy on this set of arms: specifically, the policy associated with
  $$
  B_{i,s,t}=\oX_{i,s}+\sqrt{ \frac{4\bV_{i,s}\log(10 \log t)}s}+\frac{6\log(10 \log t)}s.
  $$

The pseudo-regret of this policy is still defined as the difference between
the rewards we would have obtained by drawing an optimal arm (an arm having a mean-reward
equal to $\mu^*$) and the rewards we did obtain during the time steps $1,\dots,n$, hence, from the tower rule,
$\E \oR_n=n\mu^* - \sum_{t=1}^n \mu_{I_t}$.
Its behaviour depends on whether $\mu^*=1$ or $\mu^*<1$. 
Let us write $v_n=\tilde O(u_n)$ when for some $n_0,C>0$, $v_n \le C u_n (\log(u_n))^2$, for all $n\geq n_0$.
For $\mu^*=1$, our algorithms are such that $\E \oR_n=\tilde O(n^{\beta/(1+\beta)})$. For $\mu^*<1$, we have $\E \oR_n=\tilde O(n^{\beta/(1+\beta)})$ if $\beta>1$, and (only) $\E \oR_n=\tilde O(n^{1/2})$ if $\beta\leq 1$. 
Moreover we derive the lower bound: for any $\beta> 0$, $\mu^*\leq 1$, any algorithm satisfies
$\E \oR_n \ge C n^{\beta/(1+\beta)}$ for some $C>0$.


In continuum-armed bandits (see e.g. \cite{Agr95:continuumbandit,Kleinberg04:NIPS,auer07improvedrates}), an infinity of arms is also considered.
The arms lie in some Euclidean (or metric) space and their mean-reward is a deterministic and smooth (e.g. Lipschitz) 
function of the arms. This setting is different from ours since 
our assumption is stochastic and does not consider regularities of the mean-reward w.r.t. the arms.
However, if we choose an arm-pulling strategy which consists in selecting randomly the arms, then our setting encompasses continuum-armed bandits. For example, consider the domain $[0,1]^d$ and a mean-reward function $\mu$ assumed to be locally equivalent to a H\"older function (of order $\alpha\in[0,+\infty)$) around any maximum $x^*$ (the number of maxima is assumed to be finite), i.e. 
\begin{equation}\label{eq:holder}
\mu(x^*)-\mu(x) = \Theta( \left\|x^*-x\right\|^\alpha) \mbox{ when } x\rightarrow x^*.
\end{equation}
Pulling randomly an arm $X$ according to the Lebesgue measure on $[0,1]^d$, we have: 
$\P(\mu(X) > \mu^*-\eps) = \Theta( \P(\left\|X-x^*\right\|^\alpha<\eps) ) =\Theta ( \eps^{d/\alpha} )$, for $\eps\rightarrow 0$. Thus our assumption holds with $\beta=d/\alpha$, and our results say that if $\mu^*=1$, we have $\E \oR_n = \tilde O(n^{\beta/(1+\beta)})=\tilde O(n^{d/(\alpha+d)})$. 

For $d=1$, under the assumption that $\mu$ is $\alpha$-H\"older (i.e. $|\mu(x)-\mu(y)| \le c\left\|x-y\right\|^\alpha$ for $0<\alpha\leq 1$), \cite{Kleinberg04:NIPS} provides upper and lower bounds on the pseudo-regret $\oR_n=\Theta(n^{(\alpha+1)/(2\alpha+1)})$. Our results gives $\E \oR_n = \tilde O(n^{1/(\alpha+1)})$ which is better for all values of $\alpha$. The reason for this apparent contradiction is that the lower bound in \cite{Kleinberg04:NIPS} is obtained by the construction of a very irregular function, which actually does not satisfy our local assumption (\ref{eq:holder}). 

Now, under assumptions (\ref{eq:holder}) for any $\alpha>0$ (around a finite set of maxima), \cite{auer07improvedrates} provides the rate $\E \oR_n = \tilde O(\sqrt{n})$. Our result gives the same rate when $\mu^*<1$ but in the case $\mu^*=1$ we obtain the improved rate $\E \oR_n = \tilde O(n^{1/(\alpha+1)})$ which is better whenever $\alpha>1$ (because we are able to exploit the low variance of the good arms).
Note that like our algorithm, the algorithms in \cite{auer07improvedrates} as well as in \cite{Kleinberg04:NIPS}, do not make an explicit use (in the procedure) of the smoothness of the function. They just use a ``uniform'' discretization of the domain. 

On the other hand, the zooming algorithm of \cite{kleinberg08} adapts to the smoothness of $\mu$ (more arms are sampled at areas where $\mu$ is high). For any dimension $d$, they obtain $\E \oR_n = \tilde O(n^{(d'+1)/(d'+2)})$, where $d'\leq d$ is their ``zooming dimension''. Under assumptions (\ref{eq:holder}) we deduce $d'=\frac{\alpha-1}{\alpha}d$ using the Euclidean distance as metric, thus their pseudo-regret is $\E \oR_n = \tilde O(n^{(d(\alpha-1)+\alpha)/(d(\alpha-1)+2\alpha)})$. For locally quadratic functions (i.e. $\alpha=2$), their rate is $\tilde O(n^{(d+2)/(d+4)})$, whereas ours is $\tilde O(n^{d/(2+d)})$. Again, we have a smaller pseudo-regret although we do not use the smoothness of $\mu$ in our algorithm. Here the reason is that the zooming algorithm does not make full use of the fact that the function is locally quadratic (it considers a Lipschitz property only). 
However, in the case $\alpha<1$, our rates are worse than algorithms specifically designed for continuum armed bandits.
 
Hence, the comparison between the many-armed and continuum-armed bandits settings is not easy because of the difference in nature of the basis assumptions. Our setting is an alternative to the continuum-armed bandit setting which does not require the existence of an underlying metric space in which the mean-reward function would be smooth. Our assumption naturally deals with possibly very complicated functions where maxima may be located in any part of the space. For the continuum-armed bandit problems when there are relatively many near-optimal arms, our algorithm will be also competitive compared to the specifically designed continuum-armed bandit algorithms. This result matches the intuition that in such cases, a random selection strategy will perform well. 

Another contribution of our work is to show that, 
for infinitely many-armed bandits, we need much less exploration of each arm than for finite-armed bandits:
as shown in the next section, the index $B_{i,s,t}$ is an upper bound on $\mu_i$ which holds with probability at least $1-[\log(10t)]^{-2}$.
The use of this low confidence upper bound (compared to the ones of UCB$1$ and UCB-V for instance) 
can be explained by the fact that many sampled arms have a mean really close 
to the optimal one, and consequently exploiting not the best one but just one of the best arms
is enough to achieve the minimax pseudo-regret.

\subsection{The empirical Bernstein inequality} \label{sec:8}

A key lemma to analyze the policies using variance estimates as UCB-V
and the one used in the previous section 
is the following maximal inequality,
which in particular implies that the arm index \eqref{eq:ucbv} of UCB-V is an upper bound 
on $\mu_i$ which holds with probability at least $1-3 t^{-\zeta}$.
The interest of the lemma goes beyond the particular setting of the multi-armed bandit problems
as it provides a \emph{non asymptotic} confidence interval on the expectation of a distribution for which we observe a sample
(and for which we know a bounded interval containing its support).

\begin{lemma} \label{le:empber}
Let $U,U_1,\dots,U_n$ be independent and identically distributed random variables taking
their values in $[0,1]$. Let 
  $$
  \bar{U}_{t} = \frac1t\sum_{i=1}^{t}U_i \text{\qquad and\qquad} \bV_t = \frac1t\sum_{i=1}^{t}(U_i-\bar{U}_t)^2.
  $$
\begin{enumerate}
\item For any $\eps>0$, with probability at least $1-\eps$, for any $t\in \{1,\dots,n\}$ and $\ell_t=\frac{n\log(2 \eps^{-1})}{t^2}$,
we have
   \beglab{eq:eb}
   \bar{U}_{t}-\E U < \min\Bigg(\sqrt{2 \ell_t (\bV_t +\ell_t)}+\ell_t\Big(\frac13 +\sqrt{1- 3 \bV_t}\Big)\ ,\ \sqrt{\frac{\ell_t}2}\Bigg).
   \endlab
\item For any $\eps>0$, with probability at least $1-\eps$, for any $t\in \{1,\dots,n\}$ and $\tilde{\ell}_t=\frac{n\log(3 \eps^{-1})}{t^2}$,
we have
   \beglab{eq:eba}
   \big| \bar{U}_{t}-\E U \big| < \min\Bigg(\sqrt{2 \tilde{\ell}_t (\bV_t +\tilde{\ell}_t)}+\tilde{\ell}_t\Big(\frac13 +\sqrt{1- 3 \bV_t}\Big)\ ,\ \sqrt{\frac{\tilde{\ell}_t}2}\Bigg).
   \endlab   
\end{enumerate}
\end{lemma}
In particular, for any $\eps>0$, with probability at least $1-\eps$, for any $t\in \{1,\dots,n\}$, we have
  \beglab{eq:empber}
  \big| \bar{U}_{t}-\E U \big| < \sqrt{\frac{2n \bV_t \log(3\eps^{-1})}{t^2}}+\frac{3 n \log(3\eps^{-1})}{t^2}.
  \endlab
Inequality \eqref{eq:empber} is the one used in \cite{AuMuSz09,Mni08}, but
its tighter version \eqref{eq:eba} should be preferred. The proof of this lemma is given in Appendix \ref{sec:eb}. 
Fot $t=n$, the lemma is an empirical version of Bernstein's inequality, which differs from
the latter to the following extent: the true variance has been replaced by its empirical estimate (at the price
of having $\log(3\eps^{-1})$ terms instead of $\log(\eps^{-1})$, and 
a factor $3$ in the last term in the right-hand side instead of $1/3$. 
Inequality \eqref{eq:empber} relies on the following empirical upper bound of the variance $V$ of $U$, which
simultaneously holds with probability at least $1-\eps$: for any $t\in \{1,\dots,n\}$, we have
  $$
  V \le \Bigg(\sqrt{\bV_t+\frac{n\log(3\eps^{-1})}{t^2}}+\sqrt{\frac{n\log(3\eps^{-1})}{2t^2}(1-3\bV_t)}\Bigg)^2.
  $$  
This bound can be seen as an improvement of Inequality (5.27)  of Blanchard~\cite{Bla01}.
For $t=n \ge 2$, i.e. without the stopping time argument due to Freedman~\cite{Fre75}
allowing to have the inequality uniformly over time, Maurer and Pontil \cite{Mau09} improves on the constants of the above inequality
when the empirical variance is close to $0$.
Considering the unbiased variance estimator $\bV'_t = \frac1{t-1}\sum_{s=1}^{t}(U_s-\bar{U}_t)^2=\frac{t}{t-1}\bV_t$, they obtain that
with probability at least $1-\eps$,
  $$
  V \le \Bigg(\sqrt{\bV'_t+\frac{\log(\eps^{-1})}{2(t-1)}}+\sqrt{\frac{\log(\eps^{-1})}{2(t-1)}}\Bigg)^2.
  $$
Combined with Bernstein's bound, this gives  
that with probability at least $1-\eps$, 
  $$\big| \bar{U}_{t}-\E U \big| \le \sqrt{\frac{2 \log(3\eps^{-1})}t \bigg(\bV'_t+\frac{\log(3\eps^{-1})}{2(t-1)}\bigg)}+\frac{4 \log(3\eps^{-1})}{3(t-1)},$$ 
where the gain is on the factor of the logarithmic term when the empirical variance is much smaller than $\fracc{\log(3\eps^{-1})}t$.

Volodymyr Mnih, Csaba Szepesv\'ari and I \cite{Mni08} have used Lemma \ref{le:empber}
to address the problem of stopping the sampling of an unknown distribution $\nu$
as soon as we can output an estimate $\hmu$ of the mean $\mu$ of $\nu$ with relative error $\delta$
with probability at least $1-\eps$, that is
\begin{equation}
\label{eq:epsdelta}
\P\big( |\hmu - \mu| \leq \delta|\mu| \big) \geq 1 - \eps,
\end{equation}

\begin{figure}[t]
\bookbox{
Parameters of the problem: $\delta$, $\eps$ and the unknown distribution $\nu$.\\
Parameters of the algorithm: $q>0$, $t_1\ge1$ and $\alpha>1$ defining 
the geometric grid $t_{k}=\lceil \alpha t_{k-1}\rceil$.
(In our simulations, we take $q=0.1$, $t_1=20$ and $\alpha=1.1$.)

\medskip\noindent
Initialization:\\
$c= \frac3{\eps t_1^q(1-\alpha^{-q})}$\\
LB $\leftarrow 0$\\
UB $\leftarrow \infty$\vs{.2}\\
For $t=1,\dots,t_1-1$,\\
\hs{.3} sample $U_{t}$ from $\nu$\\
End For\vs{.2}\\
For $k=1,2,\dots$,\\
\hs{.3} For $t=t_k,\dots,t_{k+1}-1$,\\
\hs{.6} sample $U_{t}$ from $\nu$ and compute the empirical mean $\bar{U}_t=\frac1t\sum_{s=1}^{t}U_s$\\
\hs{.6} $\ell_t=\frac{t_{k+1}}{t^2} \log( c t_k^q).$\\
\hs{.6} $c_t=\min\Big(\sqrt{2 \ell_t (\bV_t +\ell_t)}+\ell_t\Big(\frac13 +\sqrt{1- 3 \bV_t}\Big),\sqrt{\frac{\ell_t}2}\Big)$\\ 
\hs{.6} LB $\leftarrow$ max(LB, $|\bar{U}_{t}|-c_t$)\\
\hs{.6} UB $\leftarrow$ min(UB, $|\bar{U}_{t}|+c_t$)\\
\hs{.6} If $(1+\delta)$LB $< (1-\delta)$UB, Then \\
\hs{1}   stop simulating $U$ and return the mean estimate $\sign(\bar{U}_{t}) \frac{(1+\delta)\text{LB}+ (1-\delta)\text{UB}}{2}$ 
\hs{.6} End If\\
\hs{.3} End For\\
End For}
\caption{Empirical Bernstein stopping (EBGStop* in our experiments).} \label{fig:ebs}
\end{figure}

For a distribution $\nu$ supported by $[a,a+1]$ for some $a\in\R$, 
we have proposed the empirical Bernstein stopping algorithm described in Figure~\ref{fig:ebs}.
It uses a geometric grid and parameters ensuring that the event $\cE = \{ \left|
\bar{U}_{t}-\mu\right| \le c_t, \,t \ge t_1 \}$ occurs with probability at least $1-\eps$.
It operates by maintaining a lower bound, LB, and an upper bound, UB,
on the absolute value of the mean of the random variable being sampled,
terminates when $(1+\delta)$LB $< (1-\delta)$UB, and
returns the mean estimate $\hmu=\sign(\bar{U}_{t}) \frac{(1+\delta)\text{LB}+ (1-\delta)\text{UB}}{2}$.
We prove that this output indeed satisfies \eqref{eq:epsdelta} 
and that the stopping time $T$ of the algorithm is upper bounded by
  $$
  T \le C\cdot\max\left(\frac{\sigma^2}{\delta^2\mu^2},\frac{1}{\delta|\mu|}\right) 
    \left(\log\bigg(\frac{2}{\eps}\bigg)+\log\bigg(\log\frac{3}{\delta|\mu|}\bigg) \right).
  $$
Up to the $\log \log$ term, this is optimal according to the work of Dagum, Karp, Luby and Ross \cite{dagum00}.

Besides, our experimental simulations show that it significantly outperforms previously known stopping rules,
in particular $\AAalg$ \cite{dagum00} and the Nonmonotonic Adaptive Sampling
(NAS) algorithm due to Domingo, Gavalda and Watanabe \cite{watanabe2000simple,Dom02}.
Figure~\ref{fig:stopmu} shows the results of running different stopping rules 
for the distribution $\nu$ of the average of $10$ uniform random variables on $[\mu-1/2,\mu+1/2]$
with varying $\mu$ and also on Bernoulli distributions. 
The experience is repeated a hundred times so that the differences observed in
Figure~\ref{fig:stopmu} are statistically significant.

\begin{figure}[t]
\begin{center}
{\includegraphics[width=.48\columnwidth]{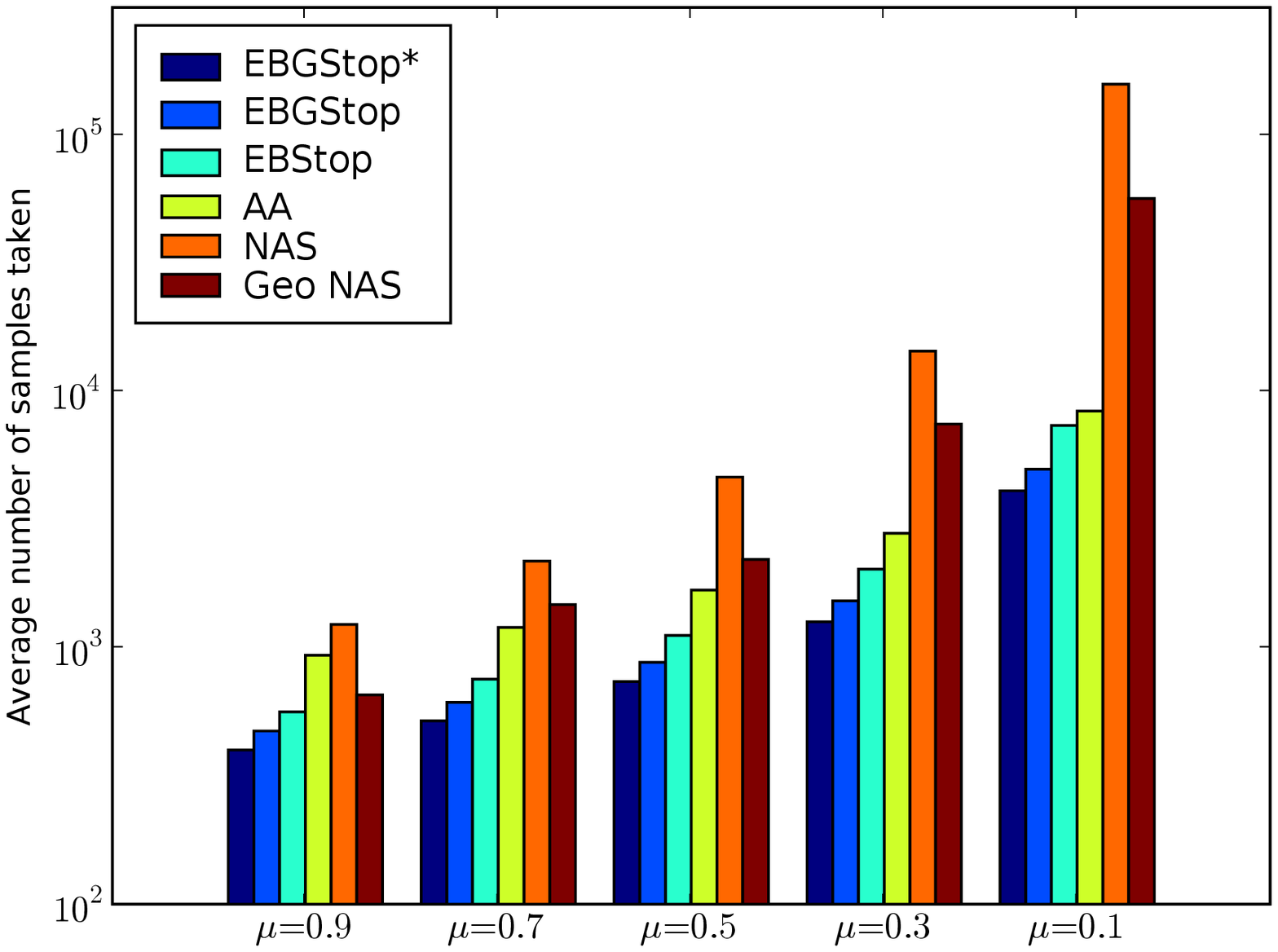}}
{\includegraphics[width=.48\columnwidth]{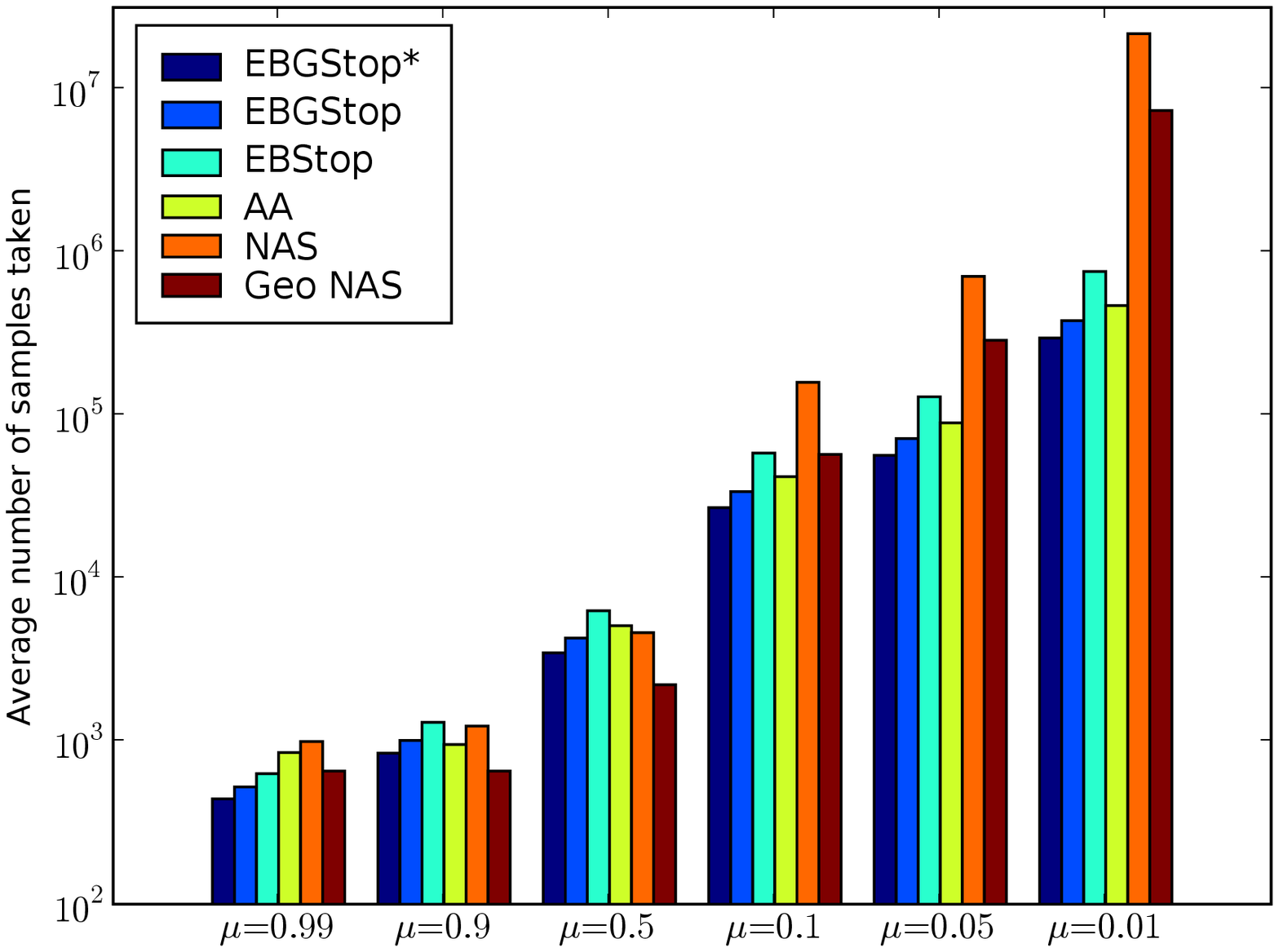}}
\begin{center}\end{center}
\vs{-1.6}
\caption{Comparison of stopping rules on (l.h.s. figure) averages of uniform random variables with 
varying means and (r.h.s. figure) Bernoulli random variables with means $0.99$, $0.9$, $0.5$, $0.1$, $0.05$, and $0.01$, averaged over $100$ runs. The average number of samples is shown in log scale.}
\label{fig:stopmu}
\end{center}
\vskip -0.4in
\end{figure}

We also use the empirical Bernstein bound in the context of racing algorithms.
Racing algorithms aim to reduce the computational burden of performing tasks
such as model selection using a hold-out set by discarding poor models quickly
\cite{MarMoo93,ortiz00}.  The context of racing algorithms is the one of multi-armed
bandit problems. 
Let $\eps>0$ be the confidence level parameter.
A racing algorithm either terminates when it runs out of time (i.e. at the end of the $n$-th round)
or when it can say that with probability at least $1-\eps$,
it has found the best option, i.e. an option $i^*\in\argmax_{i \in \{1,\dots,K\}} \mu_i$.

The Hoeffding race introduced by \cite{MarMoo93} is an algorithm based on discarding options which are likely to have smaller mean
than the  optimal one until only one option remains.
Precisely, for each time step and each distribution, $\frac{\delta}{nK}$-confidence intervals are
constructed for the mean. Options with upper confidence smaller than the lower confidence bound of
another option are discarded. The algorithm samples one by one all the options that have not been discarded yet.
Our empirical and theoretical study show that replacing the Hoeffding's inequality
by the empirical Bernstein bound leads to significant improvement.
In particular, Table~\ref{tab:1} shows the
percentage of work saved by each method ($1 - $ number of samples taken by
method divided by $Kn$), as well as the number of options remaining after termination (see \cite{Mni08}
for a more detailed description of the experiments).

\begin{table}[!ht]
\caption{Percentage of work saved / number of options left after termination.}
\label{tab:1}
\begin{center}
\begin{tabular}{lcccr}
\hline
Data set &  Hoeffding & Empirical Bernstein \\
\hline
SARCOS		& 0.0\% / 11 & 44.9\% / 4 \\
Covertype2	& 14.9\% / 8& 29.3\% / 5\\
Local		& 6.0\% / 9 & 33.1\% / 6 \\
\hline
\end{tabular}
\end{center}
\vskip -0.2in
\end{table}

\subsection{Best arm identification} \label{sec:9}

Racing algorithms \cite{MarMoo93} try to identify the best
action at a given confidence level while consuming the minimal number of pulls. 
They essentially try to optimize the exploration ``budget'' for a given confidence level.
In some applications, the budget size is fixed (say $n$ rounds), and one may want to predict
the best arm at the end of the $n$-th round. 
A motivating example concerns channel allocation for mobile phone communications. During a
very short time before the communication starts, a cellphone can explore the set of channels to find the
best one to operate. Each evaluation of a channel is noisy and there is a limited number of evaluations
before the communication starts. The connection is then launched on the channel which is believed to be the
best. 

More formally, the setting of identifying the best arm is summarized in Fi\-gure~\ref{fig:descr}.
It differs from the traditional multi-armed bandit problem by its target: the cumulative regret 
is no longer appropriate to measure the performance of a policy. The aim is rather to minimize
the simple regret:
  $$
  r_n = \Delta_{J_n},
  $$
where $J_n$ is the final recommendation of the algorithm and $\Delta_i$
still denotes the gap between the mean reward of the best arm or the mean reward
of the selected arm.
Let $i^*$ still denote the optimal arm. The simple regret is linked to the probability of error
  $$e_n= \P(J_n\neq i^*),$$ 
since, from $\E r_n = \sum_{i\neq i^*} \P(J_n=i) \Delta_i$, we have 
  $ 
  \min_{i:\Delta_i>0} \Delta_{i} e_n \le \E r_n 
    \le e_n.
  $

\begin{figure}[t]
\bookbox{
Parameters available to the forecaster: the number of rounds $n$ and the number of arms $K$.\\

Parameters unknown to the forecaster: the reward distributions $\nu_1,\dots,\nu_K$ of the arms.\\
  
For each round $t=1,2,\ldots,n$;
\begin{itemize}
\item[(1)]
the forecaster chooses $I_t \in \{ 1,\ldots,K \}$,
\item[(2)]
the environment draws the reward $X_{I_t, T_{I_t}(t)}$ from $\nu_{I_t}$ and independently of the past given $I_t$.
\end{itemize}

\medskip\noindent
At the end of the $n$ rounds, the forecaster outputs a recommendation $J_n \in \{ 1,\ldots,K \}$.
}
\caption{\label{fig:descr}
Best arm identification in multi-armed bandits.}
\end{figure}
In \cite{AuBuMu10}, Sébastien Bubeck, Rémi Munos and I prove that UCB policies can still be used provided that the
exploration term is taken much larger: precisely, for $H=\sum_{i:\Delta_i>0} \Delta_i^{-2}$
and a numerical constant $c>0$, we introduce the UCB-E (E for exploration) policy characterized by
  $$
  B_{i,s,t}=\oX_{i,s}+\sqrt{ \frac{c n}{2s H}},
  $$
which is an extremely high confidence upper bound on $\mu_i$ (probability at least $1-\exp(-\frac{c n}{H})$,
hence much higher than the confidence level of UCB$1$ and UCB-V),
and by taking $J_n$ as the arm with the largest empirical mean.
We also propose a new algorithm, called SR, based on successive rejects. 
We show that these algorithms are essentially
optimal since their simple regret decreases exponentially at a rate which is, up to a logarithmic factor, the
best possible. However, while the UCB policy needs the tuning of a parameter depending on the
unobservable hardness of the task, the successive rejects policy benefits from being parameter-free,
and also independent of the scaling of the rewards. As a by-product of our analysis, we show
that identifying the best arm (when it is unique) requires a number of samples of order $H$
(up to a $\log(K)$ factor). 
This generalizes the well-known fact that one needs of order of $1/\Delta^2$ samples to 
differentiate the means of two distributions with gap $\Delta$.
A precise understanding of both SR and the UCB-E policy leads
us to define a new algorithm, Adaptive UCB-E. It comes without guarantee of optimal rates, 
but performs slightly better than SR in practice as shown in Figure~\ref{fig:exp}.

Another variant of the best arm identification task is the problem of minimal sampling
times required to identify an $\epsilon$-optimal arm with a given confidence level, see in particular \cite{Dom02} 
and \cite{Eve06}. 
In \cite{Gru10}, Steffen Gr\"unew\"alder, Manfred Opper, John Shawe-Taylor and I also study a non-cumulative regret notion,
but in the context of a continuum of arms. Precisely, we consider the scenario in which the reward distribution for arms is modelled by a Gaussian process and there is no noise in the observed reward, and provide upper and lower bounds under
reasonable assumptions about the covariance function defining the Gaussian process. 

\begin{figure}[t]
\begin{center}
\includegraphics[width=0.47\textwidth]{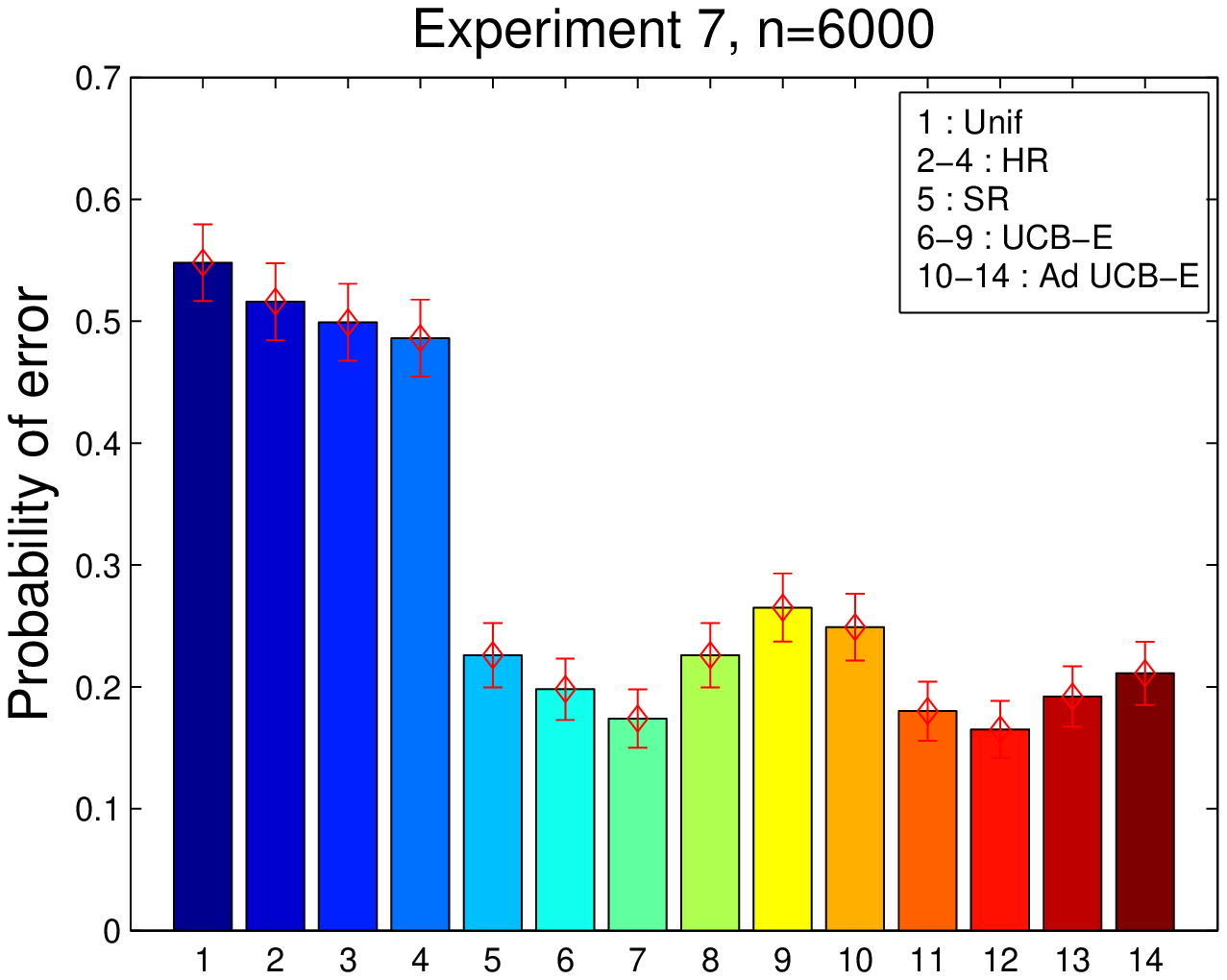}\hfill \includegraphics[width=0.47\textwidth]{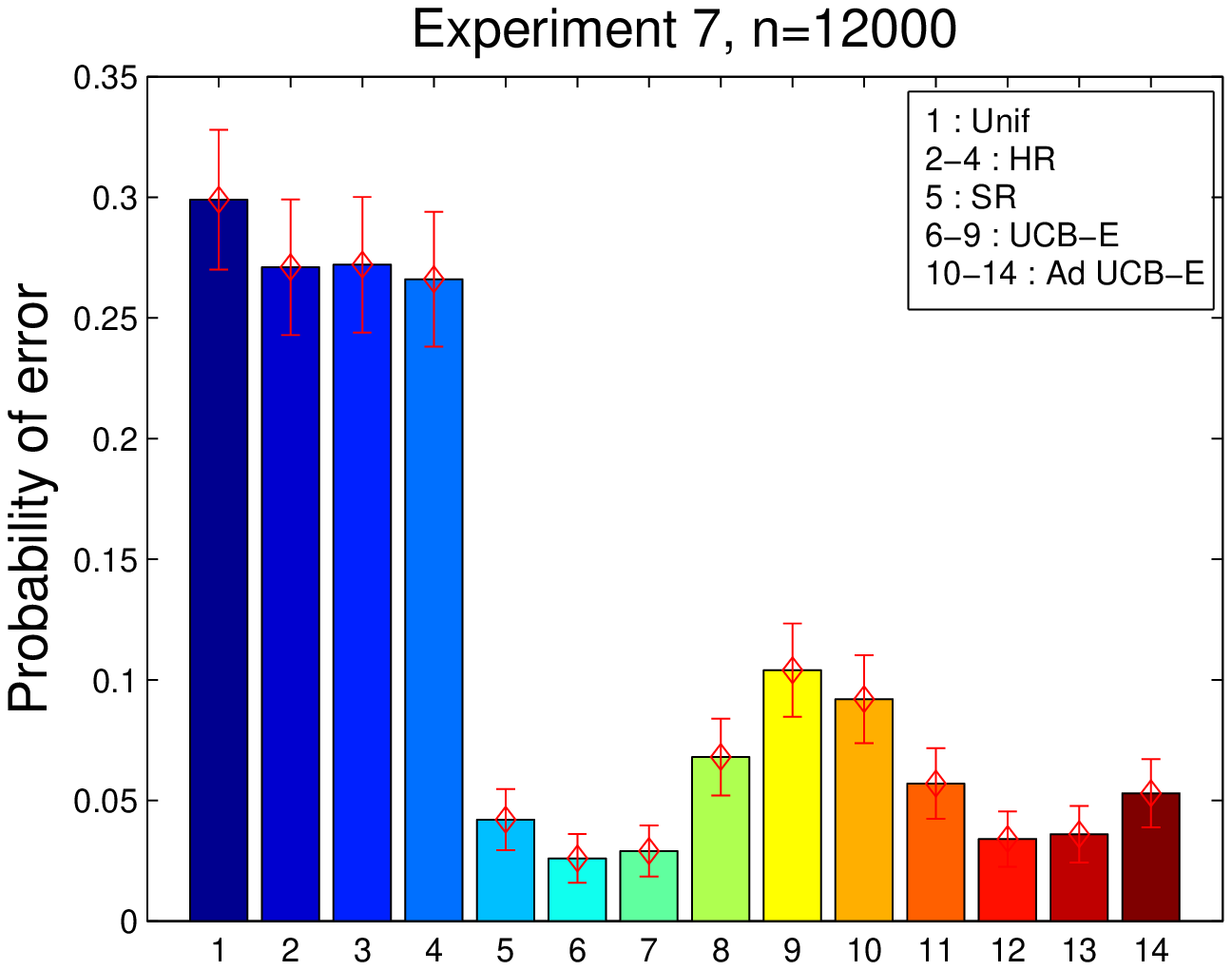}
\end{center}
\vspace{-.7cm}
\caption{
Probability of error of different algorithms for $n=6000$ (l.h.s.) and $n=12000$ (r.h.s.), 
and  $K=30$ arms having 
Bernoulli distributions with parameters $0.5$ (one arm), $0.45$ (five arms), $0.43$ (fourteen arms), 
$0.38$ (ten arms). Each bar represents a different algorithm and the bar's height 
represents the probability of error of this algorithm. 
``Unif'' is the uniform sampling strategy, ``HR'' is the Hoeffding Race algorithm
(run for three different values of the confidence level parameter), UCB-E is tested for four different
values of $c$: $2, 4, 8, 16$, Adaptive UCB-E is tested for five different values of its parameter.
More extensive experiments are presented in \cite{AuBuMu10} and confirm the ranking of algorithms 
observed on these simulations: Ad UCB-E > SR > HR > Unif, where '>' means 'has better performance than'. (UCB-E is not ranked as 
it requires the knowledge of $H$.)
%
}
\label{fig:exp}
\end{figure}

\section{Sequential prediction} \label{sec:adv}

This section summarizes my work with Sébastien Bubeck \cite{AudBub10}. It starts with
the adversarial bandit problem, and goes on with the extension to other sequential prediction games.

\subsection{Adversarial bandit} \label{sec:band}

In the general bandit problem, the environment is not constrained 
to generate the reward vectors independently as in the stochastic bandit problem.
However, the target is still to minimize the regret
      $$
      R_n = \max_{i=1,\dots,K} \sum_{t=1}^n \big( g_{i,t} - g_{I_t,t} \big). 
      $$
In the most general form of the game, called the non-oblivious/adaptive adversarial game,
the adversary may choose the reward vector $g_t$ as a function of the past decisions $I_1,\dots,I_{t-1}$.
Upper bounds on the regret $R_n$ for this type of adversary have a less
straightforward interpretation since the target cumulative reward 
is now depending on the agent's policy! 
I will not provide here results for this type of adversary but the extension of the results presented hereafter
can be found in \cite{AudBub10}.

Thus we will focus on the oblivious adversarial bandit game, in which 
the reward vector $g_t$ is \emph{not} a function of the past decisions $I_1,\dots,I_{t-1}$.
The environment is then simply defined by a distribution on $[0,1]^{nK}$, while
the agent's policy is still defined by a mapping, denoted $\varphi$ from $\cup_{t\in\{1,\dots,n-1\}} \big(\{1,\dots,K\} \times[0,1]\big)^t$
to the set of distributions of $\{1,\dots,K\}$.
Now we can see the game a bit differently.
The ``master'' of the game draws a matrix
$(g_{i,t})_{1\le i \le K,1\le t\le n}$ from the distribution defining the environment,
and at each time step $t$, draws the arm $I_t$ according to the distribution $p_t=\varphi(\cH_t)$
chosen by the agent,
where $\cH_t = \{(I_1,g_{I_1,1}),\dots,(I_{t-1},g_{I_{t-1},t-1})\}$ is the past information.
The regret $R_n$
is a random variable since it depends on the draw of the reward matrix and the draws from
the distributions $p_t$'s.

In \cite{ACFS03}, Auer, Cesa-Bianchi, Freund and Schapire have shown that a forecaster 
based on exponentially weighted averages has a regret upper bounded
by $2.7 \sqrt{n K \log K}$.
As stated before, they also show that this is optimal up to the logarithmic factor:
precisely, there is no forecaster satisfying $\E \oR_n \le \frac{1}{20} \sqrt{n K},$
for any environment.
In \cite{AudBub09,AudBub10}, we close the logarithmic gap between the above upper and lower bounds
by introducing a new 
class of randomized policies. 
To define it, consider a function
$\psi: \R_-^* \rightarrow \R^*_+$ such that
\beglab{eq:psiconds}
\begin{array}{lll}
\psi \text{ increasing and continuously differentiable,}\\
\psi'/\psi \text{ nondecreasing,}\\
\lim_{u\ra -\infty} \psi(u)<1/K, \text{ and}
\lim_{u\ra 0} \psi(u)\ge 1.
\end{array}
\endlab
It can be easily shown that there exists a continuously differentiable function $C: \R_+^K \rightarrow \R$ satisfying
for any $x=(x_1,\hdots,x_K) \in \R_{+}^K$,
  \beglab{eq:cmax}
  \max_{i=1,\dots,K} x_i < C(x) \le \max_{i=1,\dots,K} x_i - \psi^{-1}\left(\fracc{1}{K}\right),
  \endlab
and
  \beglab{eq:sum}
  \sum_{i=1}^K \psi(x_i - C(x)) = 1.
  \endlab
\begin{figure}
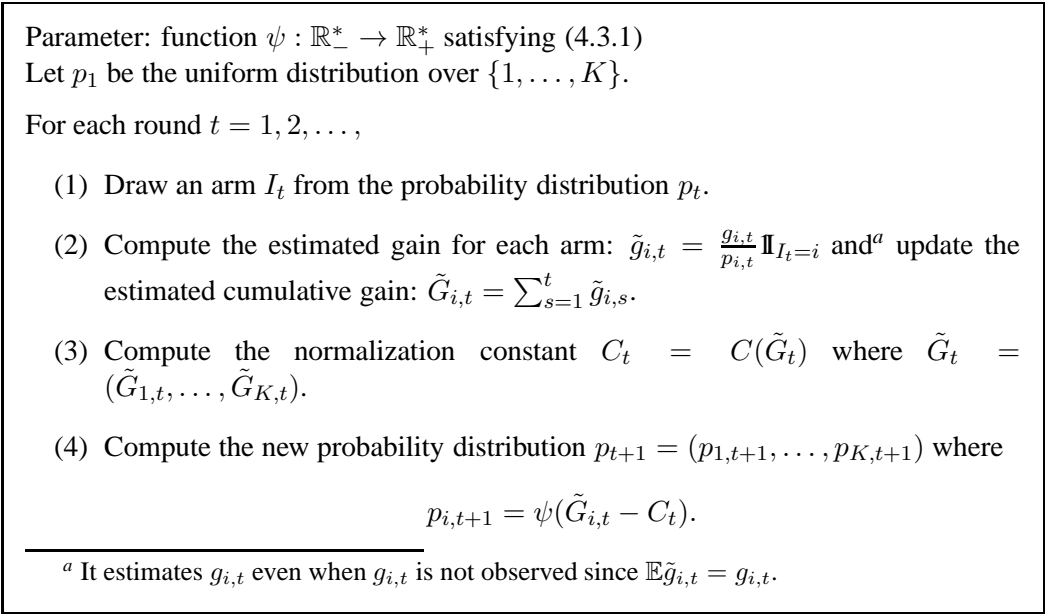
 
\bookbox{

Parameter: function $\psi: \R_-^* \rightarrow \R^*_+$ satisfying \eqref{eq:psiconds}\\
Let $p_1$ be the uniform distribution over $\{1,\hdots,K\}$.

\medskip\noindent
For each round $t=1,2,\ldots,$
\begin{itemize}
\item[(1)]
Draw an arm $I_t$ from the probability distribution $p_t$.
\item[(2)]
Compute the estimated gain for each arm: $\tilde{g}_{i,t} = \frac{g_{i,t}}{p_{i,t}} \dsone_{I_t = i}$ and\footnote{
It estimates $g_{i,t}$ even when $g_{i,t}$ is not observed since $\E \tg_{i,t}=g_{i,t}$.} update the estimated cumulative gain:
$\tilde{G}_{i,t} = \sum_{s=1}^t \tilde{g}_{i,s}$.
\item[(3)]
Compute the normalization constant $C_t = C(\tilde{G}_{t})$ where $\tilde{G}_{t}=(\tilde{G}_{1,t},\hdots,\tilde{G}_{K,t})$.
\item[(4)]
Compute the new probability distribution $p_{t+1}=(p_{1,t+1},\hdots,p_{K,t+1})$ where
$$p_{i,t+1}=\psi(\tilde{G}_{i,t} - C_t) .$$
\end{itemize}
}
\caption{INF (Implicitly Normalized Forecaster) for the adversarial bandit.} \label{fig:inf1}
\end{figure}
So we can define the implicitly normalized forecaster (INF) as detailed in Figure \ref{fig:inf1}.
Indeed, Equality \eqref{eq:sum} makes the fourth step in Figure \ref{fig:inf1} legitimate.
From \eqref{eq:cmax}, $C(\tilde{G}_{t})$ is roughly equal to $\max_{i=1,\dots,K} \tilde{G}_{i,t}$.
This means that INF chooses the probability assigned to arm $i$ as a function of the (estimated) regret.
In spirit, it is similar to the traditional weighted average forecaster, see e.g. Section 2.1 of \cite{CesLug06},  where the probabilities are proportional to a function of the difference between the (estimated) cumulative reward of arm $i$ and the cumulative reward of the policy, which should be, for a well-performing policy, of order $C(\tilde{G}_t)$. Weigthed average forecasters and implicitly normalized forecasters
are in fact two different classes of forecasters which intersection
contains exponentially weighted average forecasters such as the one considered in \cite{ACFS03}.
The interesting feature of the implicit normalization is the following argument, which
allows to recover the result of \cite{ACFS03} and more interestingly
to propose a policy having a regret of order $\sqrt{nK}$. 
It starts with an Abel transformation
and consequently is ``orthogonal'' to the usual argument which, for sake of comparison,
has been reproduced in Appendix \ref{sec:lexp}. 
Letting  $\tG_{0} =0 \in\R^K$.
We have
\begin{align}
\sum_{t=1}^{n} g_{I_t,t} & = \sum_{t=1}^{n} \sum_{i=1}^K p_{i,t} \tg_{i,t} \notag \\
& = \sum_{t=1}^{n} \sum_{i=1}^K p_{i,t} (\tG_{i,t} - \tG_{i,t-1}) \notag \\
& = \sum_{i=1}^K p_{i,n+1} \tG_{i,n} + \sum_{i=1}^K \sum_{t=1}^{n} \tG_{i,t} (p_{i,t} - p_{i,t+1}) \notag \\
& = \sum_{i=1}^K p_{i,n+1} \left(\psi^{-1}(p_{i,n+1}) + C_{n} \right) + \sum_{i=1}^K \sum_{t=1}^{n} (\psi^{-1}(p_{i,t+1}) + C_{t}) (p_{i,t} - p_{i,t+1}) \notag \\
& = C_{n} + \sum_{i=1}^K p_{i,n+1} \psi^{-1}(p_{i,n+1}) + \sum_{i=1}^K \sum_{t=1}^{n} \psi^{-1}(p_{i,t+1}) (p_{i,t} - p_{i,t+1}),
  \label{eq:fund}
\end{align}
where the remarkable simplification in the last step is closely linked to our
specific class of randomized algorithms.
The equality is interesting since, from \eqref{eq:cmax}, $C_n$ approximates the maximum estimated cumulative reward $\max_{i=1,\dots,K} \tG_{i,n}$,
which should be close to the cumulative reward of the optimal arm
$\max_{i=1,\dots,K} G_{i,n}$, where $G_{i,n}=\sum_{t=1}^n g_{i,t}$. Besides, the last term in the right-hand side is roughly 
equal to 
  $$\sum_{i=1}^K\sum_{t=1}^{n} \int_{p_{i,t}}^{p_{i,t+1}} \psi^{-1}(u) du
    = \sum_{i=1}^K\int_{1/K}^{p_{i,n+1}} \psi^{-1}(u) du$$
To make this precise, we use a Taylor-Lagrange expansion and technical arguments to control the residual terms.
Putting this together, we roughly have
  $$
  \max_{i=1,\dots,K} G_{i,n}-\sum_{t=1}^{n} g_{I_t,t} 
    \lessapprox
    - \sum_{i=1}^K p_{i,n+1} \psi^{-1}(p_{i,n+1}) + \sum_{i=1}^K \int_{1/K}^{p_{i,n+1}} \psi^{-1}(u) du.
  $$  
The right-hand side is easy to study: it depends only on the final probability vector 
and has simple upper bounds for adequate choices of $\psi$. For instance, for 
$\psi(x) = \exp(\eta x) + \frac{\gamma}{K}$ with $\eta >0$ and $\gamma \in [0,1)$, the right-hand side 
is smaller than $\frac{1-\gamma}{\eta} \log\big(\frac{K}{1-\gamma}\big) + \gamma C_n$.
For $\psi(x) = \big(\frac{\eta}{-x}\big)^{q} + \frac{\gamma}{K}$ with $\eta >0,$ $q >1$ and $\gamma \in [0,1)$,
it is smaller than $ \frac{q}{q-1} \eta K^{1/q} + \gamma C_n$.
For sake of simplicity, we have been hiding the residual terms but
these terms when added together ($nK$ terms!) are not that small, and in fact constrain the choice 
of the parameters $\gamma$ and $\eta$ if one wishes to get the tightest bound.
Our main result is the following.
\begin{thm} \label{th:polybandits}
The INF algorithm with $\psi(x) = \big(\frac{3\sqrt{n}}{-x}\big)^{2} + \frac{1}{\sqrt{nK}}$ satisfies
$$\E R_n \leq 11 \sqrt{n K}.$$
\end{thm}


\subsection{Extensions to other sequential prediction games}

Let us now describe a more general setting, in which the feedback received by the forecaster
after drawing an arm differs from game to game.  
The four games are detailed in Figure \ref{fig:game}.
\begin{figure} 
\bookbox{

\medskip\noindent

Parameters: the number of arms (or actions) $K$ and the number of rounds $n$ with $n \geq K \geq 2$.

\medskip\noindent
For each round $t=1,2,\ldots, n$
\begin{itemize}
\item[(1)]
The forecaster chooses an arm $I_t \in \{1,\hdots, K\}$, possibly with the help of an external randomization.
\item[(2)]
Simultaneously the adversary chooses the reward vector 
  $$g_t=(g_{1,t},\hdots,g_{K,t}) \in [0,1]^K$$
\item[(3)]
The forecaster receives the gain $g_{I_t,t}$ (without systematically observing it). He observes
\begin{itemize}  
\item the reward vector $(g_{1,t},\dots,g_{K,t})$ in the {\bf full information} game,
\item the reward vector $(g_{1,t},\dots,g_{K,t})$ if he asks for it with the global constraint that he is not allowed to ask it more than $m$ times for some fixed integer number $1\le m \le n$. This prediction game is the {\bf label efficient} game,
\item only $g_{I_t,t}$ in the {\bf bandit} game,
\item only his obtained reward $g_{I_t,t}$ if he asks for it with the global constraint that he is not allowed to ask it more than $m$ times for some fixed integer number $1\le m \le n$. This prediction game is the {\bf bandit label efficient} game.
\end{itemize}
\end{itemize}

Goal : The forecaster tries to maximize his cumulative gain $\sum_{t=1}^n g_{I_t,t}$.
}
\caption{The four prediction games considered in this section.} \label{fig:game}
\end{figure}
As for the weighted average forecasters, the INF forecaster can be adapted to the 
different games by simply modifying the estimates $\tg_{i,t}$ of $g_{i,t}$.
The resulting slightly modified INF forecaster is given in Figure \ref{fig:inf2}.
Interestingly, we can provide a unified analysis of 
these games for the INF forecaster. It allows to essentially recover 
the known minimax bounds, while sometimes improving the best known upper bound 
by a logarithmic term. It also leads to high probability bounds on
the regret holding for any confidence level, which contrasts with
previously known results. Let us now detail the main results for the last three games
of Figure~\ref{fig:inf2} and for the tracking the best expert scenario.

\begin{figure}[ht]
\bookbox{
{\em INF (Implicitly Normalized Forecaster):}

\medskip\noindent

Parameters: \begin{itemize}
\item the continuously differentiable function $\psi: \R_-^* \rightarrow \R^*_+$ satisfying \eqref{eq:psiconds}
\item the estimates $\tg_{i,t}$ of $g_{i,t}$ based on the (drawn arms and) observed rewards at time $t$ (and before time $t$)
\end{itemize}
Let $p_1$ be the uniform distribution over $\{1,\hdots,K\}$.

\medskip\noindent
For each round $t=1,2,\ldots,$
\begin{itemize}
\item[(1)]
Draw an arm $I_t$ from the probability distribution $p_t$.
\item[(2)]
Use the (potentially) observed reward(s) to build the estimate $\tg_t=(\tg_{1,t},\hdots,\tg_{K,t})$ 
of $(g_{1,t},\hdots,g_{K,t})$ and let:
$\tG_t = \sum_{s=1}^t \tg_s = (\tG_{1,t},\hdots,\tG_{K,t})$.
\item[(3)]
Compute the normalization constant $C_t = C(\tG_{t})$.
\item[(4)]
Compute the new probability distribution $p_{t+1}=(p_{1,t+1},\hdots,p_{K,t+1})$ where
$$p_{i,t+1}=\psi(\tG_{i,t} - C_t) .$$
\end{itemize}
}
\caption{The proposed policy for the four prediction games.}\label{fig:inf2}
\end{figure}

\subsubsection{The label efficient game} \label{sec:le}

This game was introduced by \cite{Hel97}: as explained in Figure \ref{fig:game}, 
the forecaster observes the reward vector only if he asks for it, and 
he is not allowed to ask it more than $m$ times for some fixed integer number $1\le m \le n$. 
Following the work of Cesa-Bianchi, Lugosi and Stoltz \cite{CeLuSt05},
we consider the following policy for requesting the reward vector. 
At each round, we draw a Bernoulli random variable $Z_t$, with parameter 
$\delta=\frac{3m}{4n}$, to decide whether we ask for the rewards or not. To fulfil the game requirement, 
we naturally do not ask for the rewards if $\sum_{s=1}^{t-1} Z_s \ge m$.

\begin{thm} \label{th:exphLE}
Let $\psi(x) = \exp\big(\frac{\sqrt{m \log K}}{n} x\big)$ and $\tg_{i,t}= \frac{g_{i,t}}{\delta} Z_t$
with $\delta=\frac{3m}{4n}$. Then, for any $\eps>0$, with probability at least
$1-\eps$, INF satisfies:
  \[
  R_n \le 2 n \sqrt{\frac{\log K}{m}}+ n \sqrt{\frac{27\log(2 K \eps^{-1})}{m}},
  \]
hence
$$\E R_n \le 8 n \sqrt{\frac{\log(6 K)}{m}}.$$
\end{thm}

This theorem is similar to Theorem 6.2 of \cite{CesLug06}. The main difference and novelty is that the policy does not
depend on the confidence level, so the high probability bound is valid for any confidence level \emph{for the
same policy}, 
and the expected regret of this policy has also the minimax optimal rate, i.e. $n \sqrt{\frac{\log(K)}{m}}$.

\subsubsection{High probability bounds for the bandit game} \label{sec:hp}

Here the main difference with Section \ref{sec:adv} is to use the biased estimates
$\tg_{i,t} = \frac{g_{i,t}}{p_{i,t}} \dsone_{I_t=i}+\frac{\be}{p_{i,t}}$ for some appropriate small $\beta>0$.
It may appear surprising as it introduces a bias in the estimate of $g_{i,t}$. However this modification allows to have 
high probability upper bounds with the correct rate on the difference
$\sum_{t=1}^n g_{i,t} - \sum_{t=1}^n \tg_{i,t}$. 
A second reason for this modification (but useless for this particular section) 
is that it allows to track the best expert (see Section \ref{sec:track}). 
For sake of simplicity, the following theorem concerns deterministic adversaries (which is defined
by a fixed matrix of the $nK$ rewards).
\begin{thm}
For a deterministic adversary, The INF algorithm with 
$\psi(x) = \big(\frac{3\sqrt{n}}{-x}\big)^{2} + \frac{1}{\sqrt{nK}}$ and
$\tg_{i,t} = \frac{g_{i,t}}{p_{i,t}} \dsone_{I_t=i}+\frac{1}{p_{i,t}\sqrt{nK}}$
satisfies: for any $\eps>0$, with probability at least
$1-\eps$, 
  \[
  R_n \le 10 \sqrt{nK} + 2 \sqrt{n K}\log(\eps^{-1}).
  \]
(Consequently, it also satisfies 
  $
  \E R_n \le 12 \sqrt{n K}.)
  $
\end{thm}

The novelty of the result, which is similar to Theorem 6.10 of \cite{CesLug06}, 
is both the absence of the $\log K$ factor and that the high probability bound is valid 
for the same policy at any confidence level.

\subsubsection{Label efficient and bandit game (LE bandit)} \label{sec:leb}

In this game first considered by Gy{\"o}rgy and Ottucs{\'a}k \cite{Gyo06} and which is a combination of two
previously seen games, the forecaster observes the reward of the arm he selected 
only if he asks for it, and he is not allowed to request it more than $m$ times for some fixed integer number $1\le m \le n$. 
We consider a similar policy for requesting the reward vector as in the label efficient game. 
At each round, we draw a Bernoulli random variable $Z_t$, with parameter 
$\delta=\frac{3m}{4n}$, to decide whether we ask for the obtained reward or not. To fulfil the game requirement, 
we do not ask for the rewards if $\sum_{s=1}^{t-1} Z_s \ge m$.

\begin{thm}
For 
$\psi(x) = \big(\frac{3n}{-\sqrt{m}x}\big)^{2} + \frac{1}{\sqrt{nK}}$ and
$\tg_{i,t} = g_{i,t}\frac{\dsone_{I_t=i}}{p_{i,t}} \frac{Z_t}\delta$,
the INF algorithm satisfies
  $$
  \E R_n \le 40 n \sqrt{\frac{K}{m}}.
  $$
\end{thm}

As for the bandit game, the use of the INF forecaster allows to get rid of the 
$\log K$ factor which was appearing in previous works.

\subsubsection{Tracking the best expert in the bandit game} \label{sec:track} 

In the previous sections, the cumulative gain of the forecaster was compared to 
the cumulative gain of the best single expert. Here, it will be compared to more flexible strategies that are allowed to switch actions. 
A switching strategy is described by a vector $(i_1,\dots,i_n)\in\{1,\dots,K\}^n$. 
Its size is defined by 
  \[
  \size{\its} = \sum_{t=1}^{n-1} \dsone_{i_{t+1}\neq i_{t}},
  \]
and its cumulative gain is 
  \[
  G_{\its} = \sum_{t=1}^n g_{i_t,t}.
  \]
The regret of a forecaster with respect to the best switching strategy with $S$ switches is then given by:
\[
R_n^S = \max_{\its:\,\size\its\le S} G_{\its} - \sum_{t=1}^n g_{I_t,t}.
\]
As in Section \ref{sec:hp}, we use the estimates
  \[
  \tg_{i,t}= g_{i,t} \frac{\dsone_{I_t=i}}{p_{i,t}} + \frac{\be}{p_{i,t}},
  \] 
and
$0 < \be \le 1$. The $\be$ term, which, as already stated, introduces a bias in the estimate of $g_{i,t}$, constrains the differences 
  $\max_{i=1,\dots,K} \tG_{i,t} - \min_{j=1,\dots,K} \tG_{j,t}$ 
to be relatively small. This is the key property in order to track the best switching strategy, provided that the number of switches is not too large.
We have the following result for the INF forecaster using the above estimates and an exponential function $\psi$ (recall that for 
exponential $\psi$, the INF forecaster reduces to the traditional exponentially weighted forecasters).
  
\begin{thm}
Let $\tS=S \log\big(\frac{enK}S\big) + \log(2K)$ with $e=\exp(1)$ and the natural convention $S\log(enK/S)=0$ for $S=0$. 
Consider $\psi(x) = \exp(\eta x)+ \frac{\gamma}{K}$ with 
$\gamma=\min\Big(\frac12,\sqrt{\frac{K \tS}{n}}\Big)$ and $\eta = \sqrt{\frac{\tS}{20 n K}}$,
and the estimates $\tg_{i,t}= g_{i,t} \frac{\dsone_{I_t=i}}{p_{i,t}} + \frac{\be}{p_{i,t}}$ with $\be=2\sqrt{\frac{\tS}{nK}}$. 
For these choices, for any $0 \le S \le n-1$, for any $\eps>0$, with probability at least $1-\eps$, INF satisfies:
  $$
  R_n^S \leq 9 \sqrt{n K s}
    + \sqrt{\frac{n K}{s}} \logeps,
  $$
and
$$\E R_n^S \le 10 \sqrt{n K s}.$$
\end{thm}

\section*{Acknowledgements}

I am indebted to the Lab directors Renaud Keriven at Ecole des Ponts ParisTech and 
Jean Ponce at Ecole Normale Supérieure-INRIA for giving me the opportunity to work in an excellent research environment.
I am also very grateful to the reviewers of this work: Peter Bartlett, Pascal Massart and Arkadi Nemirovski.
  
I would like to thank all the colleagues with whom I had the pleasure to work or discuss, in particular,
Francis Bach, S\'ebastien Bubeck, Olivier Catoni, Matthias Hein, Hui Kong, Rémi Munos, Csaba Szepesv\'ari and Sacha Tsybakov.
I would also like to thank Yuri Golubev for having accepted to participate and chair the habilitation committee.
Special thanks to my loved ones who have constantly and unconditionally supported me while I was developing this research.

\appendix

\numberwithin{equation}{chapter}

\chapter{Some basic properties of the Kullback-Leibler divergence} \label{app:kl}

The KL divergence between two distributions on some measurable space $\G$
  \beglab{eq:test}
  K(\rho,\pi) = \left\{ \begin{array}{lll}
    \E_{g\sim\rho} \log( \frac{\rho}{\pi} (g) ) & \text{if } \rho \ll \pi\\
  +\infty   & \text{otherwise}
  \end{array} \right.
  \endlab
satisfies for $\rho \ll \pi$, 
  $
  K(\rho,\pi)=\E_{g\sim\pi} \chi\big(\frac{\rho}{\pi} (g)\big),
  $ 
with $\chi$ the function defined on $(0,+\infty)$ by 
  $\chi(u) \mapsto u\log(u)+1-u$.
Since the function $\chi$ is nonnegative and equals zero only at $1$, we have
  \beglab{eq:klpos}
  K(\rho,\pi)\ge 0,
  \endlab
and 
  \beglab{eq:klzero}
  K(\rho,\pi)= 0 \Lra \rho=\pi.
  \endlab
Let $h:\G\ra\R$ s.t. $\E_{g\sim\pi} e^{h(g)} < +\infty$. Define
  $$
  \pi_h(dg) = \frac{e^{h(g)}}{\E_{g'\sim\pi} e^{h(g')} } \cdot \pi(dg)
  $$
By expanding the definition of the KL divergence $K(\rho,\pi_h)$, we get
  $$
  K(\rho,\pi_h)=K(\rho,\pi)-\E_{g\sim\rho} h(g) + \log \E_{g\sim\pi} e^{h(g)},
  $$
which implies from \eqref{eq:klpos} and \eqref{eq:klzero}
  \beglab{eq:klb}
  \sup_{\rho} \big\{ \E_{g\sim\rho} h(g) - K(\rho,\pi) \big\} = \log \E_{g\sim\pi} e^{h(g)},
  \endlab
and
  \beglab{eq:klc}
  \argmax_{\rho} \big\{ \E_{g\sim\rho} h(g) - K(\rho,\pi) \big\} = \pi_h.
  \endlab
By differentiating, one may note that the function
$\lam\mapsto K(\pi_{\lam h},\pi)$ is nondecreasing on $[0,+\infty)$.
Finally, if $\G$ is finite and $\pi$ is the uniform distribution on $\G$, we have
    \beglab{eq:kla}
    K(\rho,\pi) = \log(|\G|) - H(\rho) \le \log( |\G| ),
    \endlab
where $H(\rho)=-\sum_{g\in\G} \rho(g) \log \rho(g)$ is the Shannon entropy of $\rho$.

\chapter{Proof of McAllester's PAC Bayesian bound} \label{app:McA}

McAllester's bound \eqref{eq:McA} (p.\pageref{eq:McA}) states that
with probability at least $1-\eps$, for any $\rho \in \M$, we have
  \beglab{mca1}
  \E_{g\sim\rho} R(g)-\E_{g\sim\rho} r(g) 
    \le \sqrt{\frac{K(\rho,\pi)+\log(2n)+\log(\eps^{-1})}{2n-1}} .
  \endlab
Here is a short proof of this statement that essentially follows the one proposed by Seeger.

Let us first recall that a real-valued random variable $V$ such that $\E e^V \le 1$
satisfies: for any $\eps>0$, with probability at least $1-\eps$, we have
  $V\le \log(\eps^{-1}).$
So to prove \eqref{mca1}, we only need to check that the random variable 
  $$V=\sup_{\rho} \Big\{ (2n-1)\big[ \max\big(\E_{\rho(df)} R(f) - \E_{\rho(df)} r(f),0\big)\big]^2  - K(\rho,\pi) - \log(4n) \Big\}$$
satisfies $\E e^V\le 1$.

From Jensen's inequality applied to the convex function $x\mapsto \big[\max(x,0)\big]^2$ and the Legendre transform of the KL divergence \eqref{eq:klb}, we have
  \begin{align*}
  V & \le \sup_{\rho} \Big\{ (2n-1) \E_{\rho(df)} [ \max\big(R(f) - r(f),0\big) ]^2 - K(\rho,\pi) - \log(2n) \Big\} \\
  & = - \log(2n) + \log \E_{\pi(df)} e^{(2n-1)[ \max(R(f) - r(f),0) ]^2},
  \end{align*}
hence
\begin{align*}
\E e^V & \le \frac{1}{2n} \E \E_{\pi(df)} e^{(2n-1)[ \max(R(f) - r(f),0) ]^2}\\
& = \frac{1}{2n} \E_{\pi(df)} \Big( 1 + \E \big\{ e^{(2n-1)[ \max(R(f) - r(f),0) ]^2}-1\big\} \Big) \qquad\text{from Fubini's theorem}\\
& = \frac{1}{2n} \E_{\pi(df)} \bigg( 1 + \int_0^{+\infty} \P(e^{(2n-1)[ \max(R(f) - r(f),0) ]^2} - 1 > t) dt\bigg)\\ 
& = \frac{1}{2n} \E_{\pi(df)} \Bigg( 1 + \int_0^{+\infty} \P\bigg( R(f) - r(f) > \sqrt\frac{\log(t+1)}{2n-1}\bigg) dt\Bigg)\\
& \le \frac{1}{2n} \E_{\pi(df)} \bigg( 1 + \int_0^{+\infty} e^{-2n \frac{\log(t+1)}{2n-1}} dt\bigg)\qquad\qquad\text{from Hoeffding's inequality}\\
& = \frac{1}{2n} \E_{\pi(df)} \bigg( 1 + \int_0^{+\infty} (t+1)^{-\frac{2n}{2n-1}} dt\bigg)\\
& = 1,
\end{align*}
which ends the proof.

\chapter{Proof of Seeger's PAC Bayesian bound} \label{app:S}

Here we sketch the proof of \eqref{eq:S} (p.\pageref{eq:S}), which states that 
with probability at least $1-\eps$, for any $\rho \in \M$, we have
  \beglab{seeger}
  K(\E_{g\sim\rho} r(g)||\E_{g\sim\rho} R(g)) \le \frac{K(\rho,\pi)+\log(2\sqrt{n}\eps^{-1})}{n},
  \endlab
where $K(q||p)=K(\text{Be}(q),\text{Be}(p))$ with $\text{Be}(q)$ and $\text{Be}(p)$ denoting the Bernoulli distributions of parameter $q$ and $p$.
The proof follows the same line as the one of \eqref{eq:McA}.
We introduce 
$$V=\sup_{\rho} \Big\{ n K( \E_{\rho(df)} r(f)) || \E_{\rho(df)} R(f)) - K(\rho,\pi) - \log(2\sqrt{n}) \Big\},$$
and as in the previous proof, we only need to check that $\E e^V\le 1$.
This is done by using Jensen's inequality for the convex function $(q,p)\mapsto K(q||p)$ and using the Legendre transform of the KL divergence \eqref{eq:klb}.
We have 
  \begin{align*}
  \E e^V & \le \E e^{\sup_{\rho} \big\{ n \E_{\rho(df)}K( r(f) || R(f)) - K(\rho,\pi) - \log(2\sqrt{n}) \big\}}\\
  & = \frac1{2\sqrt{n}} \E \E_{\pi(df)} e^{n K( r(f)) || R(f)) }\\
  & = \frac1{2\sqrt{n}} \E_{\pi(df)} \sum_{k=0}^{n} \P(n r(f)=k) \Big(\frac{k}{nR(f)}\Big)^k \Big(\frac{n-k}{n[1-R(f)]}\Big)^{n-k}\\
  & = \frac1{2\sqrt{n}} \E_{\pi(df)} \sum_{k=0}^{n} \binom{n}{k} \Big(\frac{k}{n}\Big)^k \Big(\frac{n-k}{n}\Big)^{n-k}\\
  & \le 1,
  \end{align*}  
where the last inequality is obtained from computations using Stirling's approximation.

The same procedure can be used to prove the other PAC-Bayesian bounds of Chapter \ref{chap:PB}, Section \ref{sec:PBB}.
A similar way of approaching PAC-Bayesian theorems is given in \cite{Ger09}.

\chapter{Proof of the learning rate\\ of the progressive mixture rule} \label{app:proofpm}

Here is the proof in a concise form under the boundedness assumptions of 
Theorem \ref{th:pim} that the expected excess risk of the progressive mixture rule
is upper bounded by $\frac{\log d}{\lam(n+1)}$ for $\lam\ge \frac18$. The condition on $\lam$
guarantees that for any $y\in[-1,1]$, the function $y'\mapsto e^{-\lam (y-y')^2}$
is concave on $[-1,1]$. Thus we can write
  \begin{align}
  & \E R\bigg(\frac1{n+1}\sum_{i=0}^n \expec{g}{\pi_{-\lam \Sigma_i}} g\bigg) \notag\\
  \le & \frac1{n+1}\sum_{i=0}^n \E R\big(\expec{g}{\pi_{-\lam \Sigma_i}} g\big) \label{eq:p1}\\
  = & \frac1{n+1}\sum_{i=0}^n \E_{Z_1^{i+1}} [Y_{i+1}-\expec{g}{\pi_{-\lam \Sigma_i}}g(X_{i+1})]^2 \label{eq:p2}\\
  = & \frac1{n+1} \E_{Z_1^{n+1}} \sum_{i=0}^n [Y_{i+1}-\expec{g}{\pi_{-\lam \Sigma_i}}g(X_{i+1})]^2 \label{eq:pb}\\
  \le & \frac1{n+1} \E_{Z_1^{n+1}} \sum_{i=0}^n \bigg\{ -\frac1\lam \log \expec{g}{\pi_{-\lam \Sigma_i}} 
    e^{-\lam[Y_{i+1}-g(X_{i+1})]^2} \bigg\} \label{eq:p3}\\
  = & \frac1{\lam(n+1)} \E_{Z_1^{n+1}} \sum_{i=0}^n \log\bigg(\frac{\expec{g}{\pi} e^{-\lam \Sigma_i(g)}}
    {\expec{g}{\pi} e^{-\lam \Sigma_{i+1}(g)}}\bigg) \notag\\ 
  = & - \frac1{\lam(n+1)} \E_{Z_1^{n+1}} \log \expec{g}{\pi} e^{-\lam \Sigma_{n+1}(g)} \label{eq:p4}\\ 
  \le & - \frac1{\lam(n+1)} \E_{Z_1^{n+1}} \log \bigg(\frac{e^{-\lam \Sigma_{n+1}(\gms)}}{d}\bigg) \notag\\ 
  = & R(\gms) + \frac{\log d}{\lam(n+1)},\notag
  \end{align}
where \eqref{eq:p1} comes from Jensen's inequality on the convex function $y'\mapsto(y-y')^2$,
\eqref{eq:p2} uses that the distribution $\pi_{-\lam \Sigma_i}$ depends only on $Z_1^i$,
\eqref{eq:p3} comes from Jensen's inequality on the concave function $y'\mapsto e^{-\lam (y-y')^2}$,
and \eqref{eq:p4} is the core of the proof and explains why PM is based on a Cesaro mean.
The steps \eqref{eq:p2} and \eqref{eq:pb} are exactly the two steps of the proof of Lemma \ref{lem:rand}.
Note that this analysis gives a result similar to the one in Theorem \ref{th:pim}, except that the
factor $2$ is replaced by $\frac1\lam\ge 8$. 
For the progressive indirect mixture rule, 
$\expec{g}{\pi_{-\lam \Sigma_i}}$ are replaced by $\hh_i$,
and the step \eqref{eq:p3} is still valid from the very definition \eqref{eq:pimdef} of $\hh_i$.

\chapter{The empirical Bernstein's inequality} \label{sec:eb}

The goal of the empirical Bernstein's inequality is to provide confidence bounds 
on the expectation of a distribution with bounded support, say in $[0,1]$, given a sample from it.
Let $U,U_1,U_2,\dots$ be independent and identically distributed random variables taking
their values in $[0,1]$. Let 
  $$
  \bar{U}_{t} = \frac1t\sum_{i=1}^{t}U_i \text{\qquad and\qquad} \bV_t = \frac1t\sum_{i=1}^{t}(U_i-\bar{U}_t)^2.
  $$
Here we prove the empirical Bernstein's inequality (Lemma \ref{le:empber}, p.\pageref{le:empber}), which states that for 
any $\eps>0$, with probability at least $1-2\eps$, for any $t\in \{1,\dots,n\}$ and $\bl=\frac{n\log(\eps^{-1})}{t^2}$,
we have
   \beglab{eq:ebb}
   \bar{U}_{t}-\E U < \min\Bigg(\sqrt{2 \bl (\bV_t +\bl)}+\bl\Big(\frac13 +\sqrt{1- 3 \bV_t}\Big)\ ,\ \sqrt{\frac{\bl}2}\Bigg).
   \endlab
\begin{proof}
Let $\Lambda(\lam)=\log \E e^{\lam(U-\E U)}$ be the log-Laplace transform of the random variable $U-\E U$.
Let $S_t=\sum_{i=1}^t (U_i-\E U_i)$ with the convention $S_0=0$.
From Inequality (2.17) of \cite{Hoe63}, we have\footnote{This comes from a martingale argument due to Doob. For any $\lam>0$, 
the sequence $(e^{\lam S_t-t \Lambda(\lam)})_{t\ge0}$ is a martingale with respect to the filtration $\big(\sigma(U_1,\dots,U_t)\big)_{t\ge 0}$
since $\E (e^{\lam S_t-t \Lambda(\lam)}|U_1,\dots,U_{t-1})=e^{\lam S_{t-1}-(t-1) \Lambda(\lam)}$.
Introduce the stopping time $T=\min\big(n+1,\min\{t\in\N: S_t  \ge s\}\big)$. From the optional stopping theorem, for any $\lam>0$, we have
  $$1 = \E e^{\lam S_T - T \Lambda(\lam)} \ge \P(T\le n) e^{\lam s-n\Lambda(\lam)},$$
hence
  $$\P\big(\max_{1\le t \le n} S_t \ge s \big)=\P(T\le n) \le \und{\inf}{\lam>0} e^{-\lam s+n\Lambda(\lam)}.$$
}
  $$
  \P\big( \max_{1\le t \le n} S_t \ge s \big) \le \und{\inf}{\lam>0} e^{-\lam s+n\Lambda(\lam)}.
  $$ 
Let $V=\Var U$. Hoeffding's inequality and Bennett's inequality implies 
  $$\Lambda(\lam)\le\min\Big(\frac{\lam^2}8,(e^\lam-1-\lam) V \Big),$$
which by standard computations (see, e.g., Inequality (45) of \cite{AuMuSz09}) gives that for any $\eps>0$, with probability at least $1-\eps$,
    \beglab{eq:bez}
    \max_{1\le t \le n} S_t < \min\bigg( \sqrt{\frac{ n\logeps}{2}}, \sqrt{2nV\logeps}+{\frac{\logeps}{3}}\bigg).
    \endlab
Let $W=(U-\E U)^2$ and $W_i=(U_i-\E U_i)^2$ for $i\ge 1$.
Let $S'_t=\sum_{i=1}^t (-W_i+\E W_i)$ and $\Lambda'(\lam)=\log \E e^{\lam(-W+\E W)}$. As above, from Inequality (2.17) of \cite{Hoe63}, we have
  $$
  \P\big( \max_{1\le t \le n} S'_t \ge s \big) \le \und{\inf}{\lam>0} e^{-\lam s+n\Lambda'(\lam)}.
  $$
Now using that $e^{-u} \le 1-u+\frac{u^2}2$ for $u\ge 0$ and $\log(1+u)\le u$ from $u\ge -1$,
we have
  $\log \E e^{-\lam W} \le \log \E(1-\lam W+\frac{\lam^2W^2}{2})
  \le -\lam \E W+\frac{\lam^2}2\E(W^2)$,
hence $\Lambda'(\lam) \le \frac{\lam^2}2 \E(W^2)$.
Optimizing with respect to $\lam$ gives that for any $\eps>0$, with probability at least $1-\eps$,
    \beglab{eq:bea}
    \max_{1\le t \le n} S'_t < \sqrt{{2n\E(W^2)\logeps}}.
    \endlab
Now we use the following lemma to bound $\E(W^2)$.
\begin{lemma}
A random variable $U$ taking its values in $[0,1]$ satisfies
  \beglab{eq:funny}
  \E[(U-\E U)^4] \le V(1-3V),
  \endlab
where $V=\E[(U-\E U)^2]$ is the variance of $U$.
If $U$ admits a Bernoulli distribution, one can put an equality in \eqref{eq:funny}.
\end{lemma}
\begin{proof}
We have
  \begin{align*}
  \E[(U-\E U)^4] - V(1-3V) = \E\big([U^3-& U+\E(U)][U-\E(U)]\big)\\ & + 3 \big( [\E(U^2)]^2-\E(U)\E(U^3)\big).
  \end{align*}
From Chebyshev's association inequality (also referred to as the Fortuin-Kasteleyn-Ginibre inequality), both terms in the right-hand side are nonpositive.
An alternative proof consists in expanding the terms in Lemma 8 of \cite{Mau09} and noticing that this exactly gives \eqref{eq:funny}.
The result for Bernoulli distributions comes from direct computations.
\end{proof}

Combining the above lemma with \eqref{eq:bea}, we get that with probability at least $1-\eps$,
  \beglab{eq:beb}
  \max_{1\le t \le n} S'_t < \sqrt{2{n}V(1-3V)\logeps}.
  \endlab
We now work on the event $\calE$ of probability at least $1-2\eps$ on which both \eqref{eq:beb} and \eqref{eq:bez}
hold. The variance decomposition gives $\bV_t=\frac1t\sum_{i=1}^{t}(U_i-\bar{U}_t)^2=-(\E U-\bar{U}_t)^2+\frac1t\sum_{i=1}^{t} W_i$, hence
$S'_t=t(V-\bV_t)-t(\E U-\bar{U}_t)^2.$
For any $1\le t\le n$, we have 
  \beglab{eq:bed}
  \bar{U}_{t}-\E U < \min\bigg( \sqrt{\frac{\bl}{2}}, \sqrt{2V\bl}+{\frac{\bl}{3}}\bigg),
  \endlab
and
  \beglab{eq:bee}
  V-\bV_t < \sqrt{2V(1-3V)\bl}+(\bar{U}_t-\E U)^2
  \endlab
If $\bar{U}_t < \E U$, then \eqref{eq:ebb} is trivial.
If $\bV_t \ge V$, \eqref{eq:ebb} is a direct consequence of \eqref{eq:bed} (since $\frac43 - 3 \bV_t\ge\frac43 - \frac34>\frac13$).
Therefore, from now and on, we consider $\bar{U}_t \ge \E U$ and $\bV_t < V$.
Then \eqref{eq:bed} implies $(\bar{U}_t-\E U)^2\le \bl/2,$ and \eqref{eq:bee} leads to
$$\bV_t > V - \sqrt{2V(1-3\bV_t)\bl}-\bl/2 = \Bigg( \sqrt{V}-\sqrt{\frac{\bl(1-3\bV_t)}2} \Bigg)^2 - \frac{\bl(2-3\bV_t)}2,$$
hence
  $$\sqrt{V} < \sqrt{\bV_t+\frac{\bl(2-3\bV_t)}2}+\sqrt{\frac{\bl(1-3\bV_t)}2} \le \sqrt{\bV_t+\bl}+\sqrt{\frac{\bl(1-3\bV_t)}2}.$$
By plugging this inequality into \eqref{eq:bed}, we get \eqref{eq:bez}.
For the two-sided inequality \eqref{eq:eba}, one just needs to add the same inequality
as \eqref{eq:bed} for $-\bar{U}_t$. At the end, three maximal inequalities are used (corresponding to
$U_i$, $-U_i$ and $-W_i$), so that the result holding with probability at least $1-\eps$
contains $\log(3\eps^{-1})$ terms.
\end{proof}

\chapter{On Exploration-Exploitation with Exponential weights (EXP3)} \label{app:exptrois}

\section{The variants of EXP3}

\begin{figure}[th] 
\bookbox{
Parameters: $\eta\in(0,1/K]$ and $\gamma\in[0,1]$.\\
Let $p_1$ be the uniform distribution over $\{1,\hdots,K\}$.

\medskip\noindent
For each round $t=1,2,\ldots,$
\begin{itemize}
\item[(1)]
Draw an arm $I_t$ according to the probability distribution $p_t$.
\item[(2)]
Compute the estimated gain for each arm: 
  $$\tilde{g}_{i,t} = \left\{\begin{array}{ll}
  \frac{g_{i,t}}{p_{i,t}} \dsone_{I_t = i} & \text{for the reward-magnifying version of EXP3}\\
  1-\frac{1-g_{i,t}}{p_{i,t}} \dsone_{I_t = i} & \text{for the loss-magnifying version of EXP3}\\
  \frac{g_{i,t}}{p_{i,t}} \dsone_{I_t = i}+\frac{\be}{p_{i,t}} & \text{for the tracking version of EXP3}\\
  \frac{g_{i,t}(1+\be\frac{g_{i,t}}{p_{i,t}})}{p_{i,t}} \dsone_{I_t = i} & \text{for the tightly biased version of EXP3}
  \end{array}\right.
  $$
and update the estimated cumulative gain:
$\tilde{G}_{i,t} = \sum_{s=1}^t \tilde{g}_{i,s} $.
\item[(3)]
Compute the new probability distribution over the arms:
  $$p_{t+1} = \gamma p_1 + (1-\gamma) q_{t+1},$$
with
  $$q_{i,t+1} = \frac{\exp{\left(\eta \tilde{G}_{i,t}\right)}}{\sum_{k=1}^K \exp{\left(\eta \tilde{G}_{k,t}\right)}}.$$
\end{itemize}
}
\caption{EXP3 (Exploration-Exploitation with Exponential weights) for the adversarial bandit problem.}\label{fig:exptrois}
\end{figure}

There are several variants of EXP3. 
They differ by the way $g_{i,t}$ is estimated as shown in Figure~\ref{fig:exptrois}.
For deterministic adversaries, the loss-magnifying version of EXP3 has the advantage 
to provide the best known constant in front of the $\sqrt{nK\log K}$ term, that is $\sqrt{2}$ (note that our work
succeeds in removing the $\log K$ term but at the price of a larger numerical constant factor).
For deterministic adversaries, the reward-magnifying version of EXP3 (which is the one in 
the seminal paper of Auer, Cesa-Bianchi, Freund and Schapire \cite{ACFS03} for $\gamma=K\eta$) has the advantage that 
the factor $n$ in $\sqrt{nK\log K}$ can be replaced by $\max_{i=1,\dots,n} G_{i,n}$, where $G_{i,n}=\sum_{t=1}^n g_{i,t}$. 
The tracking version of EXP3 is the one proposed in Section 6.8 of \cite{CesLug06} (and the one presented in Section \ref{sec:track}).
It (slightly) overestimates the rewards since we have $\E_{I_t\sim p_t} \tg_{i,t} = g_{i,t}+\frac{\be}{p_{i,t}}$.
This idea was introduced in \cite{Aue02} for tracking the best expert. 
In \cite{AudBub10}, we have introduced the tightly biased version of EXP3 to achieve regret bounds depending on the performance of the optimal arm.
Contrarily to the reward-magnifying version of EXP3, these bounds hold for any adversary and 
high probability regret bounds are also obtained.

\section{Proof of the learning rate of the reward-magnifying EXP3} \label{sec:lexp}

Here we give an analysis of the reward-magnifying EXP3 (defined in Figure~\ref{fig:exptrois}),
which is an improvement (in terms of constant only) of the one in \cite[Section 3]{ACFS03}.

\begin{thm} \label{th:rew}
Let $\Gmax=\max_{i=1,\dots,K} G_{i,n}.$
For deterministic adversaries, if $4\eta K \le 5\gamma$, the expected 
regret of the reward-magnifying EXP3 satisfies
  $$\E R_n \le \frac{\log K}{\eta} + \gamma \Gmax.$$
In particular, if $\eta=\sqrt{ \frac{5\log K}{4n K}}$ and $\gamma=\min\Big(\sqrt{ \frac{4K\log K}{5n}},1\Big)$, we have
  $$\E R_n \le \sqrt{\frac{16}5n K\log(K)}.$$
\end{thm}

\begin{proof}
The condition $4\eta K \le 5\gamma$ is put to guarantee that 
$\Psi\big(\frac{\eta K}\gamma\big) \le \frac\gamma{\eta K}$,
where $\Psi: u\mapsto \frac{e^u-1-u}{u^2}$ is an increasing function.
For any adversary, we have
  \begin{align*}
  \sum_{t=1}^n g_{I_t,t} & = \sum_{t=1}^n \E_{k\sim p_t} \tg_{k,t}\\
    & = \frac{1-\gamma}\eta \sum_{t=1}^n \bigg(\log \E_{i\sim q_t} e^{\eta \tg_{i,t}} 
      - \log\big[e^{-\frac{\eta}{1-\gamma} \E_{k\sim p_t} \tg_{k,t}} \E_{i\sim q_t} e^{\eta \tg_{i,t}}\big] \bigg) \\
    & = \frac{1-\gamma}\eta \bigg(S-\sum_{t=1}^n \log(D_t) \bigg),
  \end{align*}
where
  $$
  S=\sum_{t=1}^n \log \E_{i\sim q_t} e^{\eta \tg_{i,t}}
  =\sum_{t=1}^n \log\bigg( \frac{\E_{i\sim p_1} e^{\eta \tG_{i,t}}}{\E_{i\sim p_1} e^{\eta \tG_{i,t-1}}} \bigg)
  =\log \E_{i\sim p_1} e^{\eta \tG_{i,n}}
  $$
and
  \begin{align}
  D_t& =e^{-\frac{\eta}{1-\gamma} \E_{k\sim p_t} \tg_{k,t}} \E_{i\sim q_t} e^{\eta \tg_{i,t}} \notag\\
  & \le e^{-\frac{\eta}{1-\gamma} \E_{k\sim p_t} \tg_{k,t}} \E_{i\sim q_t} 
    \bigg( 1+\eta \tg_{i,t}+\Psi\Big(\frac{\eta K}\gamma\Big)\eta^2 \tg_{i,t}^2\bigg) \label{c2}\\
  & = e^{-\frac{\eta}{1-\gamma} \E_{k\sim p_t} \tg_{k,t}}  
    \bigg( 1+\eta \frac{\E_{i\sim p_t} \tg_{i,t}-\gamma \E_{i\sim p_1} \tg_{i,t}}{1-\gamma}+\Psi\Big(\frac{\eta K}\gamma\Big)\eta^2 \E_{i\sim q_t}\tg_{i,t}^2\bigg) \notag\\
  & \le e^{-\frac{\eta}{1-\gamma} \E_{k\sim p_t} \tg_{k,t}}  
    \bigg( 1+\frac{\eta}{1-\gamma} \E_{i\sim p_t} \tg_{i,t}-\frac{\eta\gamma}{1-\gamma} \E_{i\sim p_1} \tg_{i,t}
      +\frac{\Psi(\frac{\eta K}\gamma)\eta^2 K}{1-\gamma} \E_{i\sim p_1}\tg_{i,t}\bigg) \label{c3}\\
  & \le e^{-\frac{\eta}{1-\gamma} \E_{k\sim p_t} \tg_{k,t}}  
    \bigg( 1+\frac{\eta}{1-\gamma} \E_{i\sim p_t} \tg_{i,t}\bigg) \label{c4}\\
  & \le 1.\notag
  \end{align}
To get \eqref{c2}, we used that $\Psi$ is an increasing function and that 
  $\eta\tg_{i,t} \le \frac\eta{p_{i,t}} \le \frac{\eta K}\gamma$.
For \eqref{c3}, we used $(1-\gamma)\E_{i\sim q_t} \tg_{i,t}^2\le\E_{i\sim p_t} \tg_{i,t}^2=\frac{g_{I_t,t}^2}{p_{I_t,t}}\le \sum_{i=1}^K \tg_{i,t}=K\E_{i\sim p_1} \tg_{i,t}$.
For \eqref{c4}, we used $\eta K\Psi\big(\frac{\eta K}\gamma\big) \le \gamma$.
We have thus proved 
  \beglab{eq:eee}
  \sum_{t=1}^n g_{I_t,t} \ge \frac{1-\gamma}\eta \log \E_{i\sim p_1} e^{\eta \tG_{i,n}}
  \endlab
For a deterministic adversary, we have
  $\E \tG_{i,n} = \E G_{i,n} = G_{i,n}$, so that
  \begin{align}
  \E \sum_{t=1}^n g_{I_t,t} & \ge \frac{1-\gamma}\eta \E \log \E_{i\sim p_1} e^{\eta \tG_{i,n}}\notag\\
  & \ge \frac{1-\gamma}\eta \log \E_{i\sim p_1} e^{\eta \E \tG_{i,n}} \label{eq:cc1}\\
  & = \frac{1-\gamma}\eta \log \E_{i\sim p_1} e^{\eta G_{i,n}} \ge - \frac{(1-\gamma)\log K}{\eta} + (1-\gamma) \max_{i=1,\dots,K} G_{i,n},\notag
  \end{align}
where Inequality \eqref{eq:cc1} which moves the expectation sign inside the exponential 
can be viewed as an infinite dimensional Jensen's inequality (see Lemma 3.2 of \cite{Aud09a}). For a deterministic adversary,
we have proved 
  \begin{align*}
  \E R_n=\max_{i=1,\dots,K} G_{i,n}-\E \sum_{t=1}^n g_{I_t,t} \le \frac{(1-\gamma)\log K}{\eta} + \gamma \max_{i=1,\dots,K} G_{i,n},
  \end{align*}
hence the first claimed result.

 The second result is trivial when $\sqrt{\fracr{4K\log K}{5n}}\ge 1$ since the upper bound is then larger than $n$.
Otherwise, we have $\gamma=\sqrt{\fracr{4K\log K}{5n}}<1$ and $4\eta K = 5\gamma$ so that the result follows from the first one.
\end{proof}


\chapter{Experimental results for the min-max truncated estimator defined in Section \ref{sec:minmax}} \label{app:exper}

In Section \ref{sec:noise}, we detail the different kinds of noises we work with.
Then, Sections \ref{sec:expind}, \ref{sec:hcc} and \ref{sec:ts} describe 
the three types of functional relationships between the input, the output and the noise involved
in our experiments. A motivation for choosing these input-output distributions was the ability to compute
exactly the excess risk, and thus to compare easily estimators.
Section \ref{sec:exper} presents the experimental results. 

\section{Noise distributions} \label{sec:noise}

In our experiments, we consider different types of noise that are centered and with unit variance:
\begin{itemize}
\item the standard Gaussian noise: $W \sim \cN(0,1)$,
\item a heavy-tailed noise defined by: 
  $W=\sign(V)/|V|^{1/q}$, 
with $V \sim \cN(0,1)$ a standard Gaussian random variable and $q=2.01$ (the 
real number $q$ is taken strictly larger than $2$ as for $q=2$, the random variable $W$ would not admit
a finite second moment).  
\item an asymmetric heavy-tailed noise defined by: 
  $$
  W=\left\{\begin{array}{lll} 
  |V|^{-1/q} & \text{if } V>0, \\
  -\frac{q}{q-1}  & \text{otherwise},
  \end{array}\right.
  $$
with $q=2.01$
with $V \sim \cN(0,1)$ a standard Gaussian random variable.
\item a mixture of a Dirac random variable with a low-variance Gaussian random variable defined by: 
with probability $p$, $W=\sqrt{\fracl{1-\rho}p}$, and with probability $1-p$, $W$ is drawn from 
  $$\N\bigg(-\frac{\sqrt{p(1-\rho)}}{1-p},\frac{\rho}{1-p}-\frac{p(1-\rho)}{(1-p)^2}\bigg).$$ 
The parameter $\rho\in[p,1]$ characterizes the 
part of the variance of $W$ explained by the Gaussian part of the mixture.
Note that this noise admits exponential moments, but for $n$ of order $1/p$, the
Dirac part of the mixture generates low signal to noise points.
\end{itemize}

\section{Independent normalized covariates (INC$(n,d)$)} \label{sec:expind}
In INC$(n,d)$, the input-output pair is such that
  $$
  Y=\langle \th^* , X \rangle + \sigma W,
  $$
where the components of $X$ are independent standard normal distributions,
$\th^*=(10,\dots,10)^T \in\R^d$, and $\sigma=10$.

\section{Highly correlated covariates (HCC$(n,d)$)} \label{sec:hcc}
In HCC$(n,d)$, the input-output pair is such that
  $$
  Y=\langle \th^* , X \rangle + \sigma W,
  $$
where $X$ is a multivariate centered normal Gaussian with
covariance matrix $Q$ obtained by drawing a $(d,d)$-matrix $A$ of uniform random variables 
in $[0,1]$ and by computing $Q=A A^T$,
$\th^*=(10,\dots,10)^T \in\R^d$, and $\sigma=10$.
So the only difference with the setting of Section \ref{sec:expind} is 
the correlation between the covariates.

\section{Trigonometric series (TS$(n,d)$)} \label{sec:ts}
Let $X$ be a uniform random variable on $[0,1]$.
Let $d$ be an even number.
Let 
  $$\vp(X)= \big(\cos(2 \pi X),\dots,\cos(d \pi X),\sin(2 \pi X),\dots,\sin(d \pi X)\big)^T.$$
In TS$(n,d)$, the input-output pair is such that
  $$
  Y=20 X^2-10 X-\frac53+\sigma W,
  $$
with $\sigma=10$.
One can check that this implies
  $$\theta^*=\bigg(\frac{20}{\pi^2},\dots,\frac{20}{\pi^2 (\frac{d}2)^2},
   -\frac{10}{\pi},\dots,-\frac{10}{\pi (\frac{d}2)}\bigg)^T \in\R^d.$$ 
  
\section{Results} \label{sec:exper}

Tables \ref{tab:b01} and \ref{tab:b04} give the results for the mixture noise.
Tables \ref{tab:a201}, \ref{tab:a-201} and \ref{tab:a0} provide the results for the heavy-tailed noise
and the standard Gaussian noise. Each line of the tables has been obtained after $1000$ generations of the training set.
These results show that the min-max truncated estimator $\hg$ is often equal to the ordinary least squares estimator $\hfols$,
while it ensures impressive consistent improvements when it differs from $\hfols$.
In this latter case, the number of points that are not considered in $\hg$, i.e. the number of points
with low signal to noise ratio, varies a lot from $1$ to $150$ and is often of order $30$. Note that not only the points
that we expect to be considered as outliers (i.e. very large output points) are erased, and
that these points seem to be taken out by local groups: see Figures 
\ref{fig:1} and \ref{fig:2} 
in which the erased points are marked by surrounding circles.

Besides, the heavier the noise tail is (and also the larger the variance of the noise is), the more often the truncation modifies the initial ordinary least squares estimator, and the more improvements we get from the min-max truncated estimator, which also becomes much more robust than the ordinary least squares estimator (see the confidence intervals in the tables).

\begin{table}[!ht]
\caption{Comparison of the min-max truncated estimator $\hg$
with the ordinary least squares estimator $\hfols$ for the mixture noise (see Section \ref{sec:noise})
with $\rho=0.1$ and $p=0.005$. In parenthesis, the $95\%$-confidence intervals for the estimated quantities.}
\label{tab:b01} 
\begin{center}
\scalebox{0.7}{\begin{tabular}{l|c|c|c|c|c|c|c}
\hline 
& \rotatebox{90}{nb of iterations } & \rotatebox{90}{nb of iter. with $R(\hg)\neq R(\hfols)$} & 
\rotatebox{90}{nb of iter. with $R(\hg)< R(\hfols)$} & \rotatebox{90}{$\E R(\hfols)-R(\gl)$}
& \rotatebox{90}{$\E R(\hg)-R(\gl)$} & \rotatebox{90}{$\E R[(\hfols)|\hg\neq\hfols]-R(\gl)$} 
& \rotatebox{90}{$\E [R(\hg)|\hg\neq\hfols]-R(\gl)$}\\ 
\hline 
INC(n=200,d=1)& $1000$& $419$& $405$& $ 0.567 (\pm  0.083)$& $ 0.178 (\pm  0.025)$& $ 1.191 (\pm  0.178)$& $ 0.262 (\pm  0.052)$\\
INC(n=200,d=2)& $1000$& $506$& $498$& $ 1.055 (\pm  0.112)$& $ 0.271 (\pm  0.030)$& $ 1.884 (\pm  0.193)$& $ 0.334 (\pm  0.050)$\\
HCC(n=200,d=2)& $1000$& $502$& $494$& $ 1.045 (\pm  0.103)$& $ 0.267 (\pm  0.024)$& $ 1.866 (\pm  0.174)$& $ 0.316 (\pm  0.032)$\\
TS(n=200,d=2)& $1000$& $561$& $554$& $ 1.069 (\pm  0.089)$& $ 0.310 (\pm  0.027)$& $ 1.720 (\pm  0.132)$& $ 0.367 (\pm  0.036)$\\
INC(n=1000,d=2)& $1000$& $402$& $392$& $ 0.204 (\pm  0.015)$& $ 0.109 (\pm  0.008)$& $ 0.316 (\pm  0.029)$& $ 0.081 (\pm  0.011)$\\
INC(n=1000,d=10)& $1000$& $950$& $946$& $ 1.030 (\pm  0.041)$& $ 0.228 (\pm  0.016)$& $ 1.051 (\pm  0.042)$& $ 0.207 (\pm  0.014)$\\
HCC(n=1000,d=10)& $1000$& $942$& $942$& $ 0.980 (\pm  0.038)$& $ 0.222 (\pm  0.015)$& $ 1.008 (\pm  0.039)$& $ 0.203 (\pm  0.015)$\\
TS(n=1000,d=10)& $1000$& $976$& $973$& $ 1.009 (\pm  0.037)$& $ 0.228 (\pm  0.017)$& $ 1.018 (\pm  0.038)$& $ 0.217 (\pm  0.016)$\\
INC(n=2000,d=2)& $1000$& $209$& $207$& $ 0.104 (\pm  0.007)$& $ 0.078 (\pm  0.005)$& $ 0.206 (\pm  0.021)$& $ 0.082 (\pm  0.012)$\\
HCC(n=2000,d=2)& $1000$& $184$& $183$& $ 0.099 (\pm  0.007)$& $ 0.076 (\pm  0.005)$& $ 0.196 (\pm  0.023)$& $ 0.070 (\pm  0.010)$\\
TS(n=2000,d=2)& $1000$& $172$& $171$& $ 0.101 (\pm  0.007)$& $ 0.080 (\pm  0.005)$& $ 0.206 (\pm  0.020)$& $ 0.083 (\pm  0.012)$\\
INC(n=2000,d=10)& $1000$& $669$& $669$& $ 0.510 (\pm  0.018)$& $ 0.206 (\pm  0.012)$& $ 0.572 (\pm  0.023)$& $ 0.117 (\pm  0.009)$\\
HCC(n=2000,d=10)& $1000$& $669$& $669$& $ 0.499 (\pm  0.018)$& $ 0.207 (\pm  0.013)$& $ 0.561 (\pm  0.023)$& $ 0.125 (\pm  0.011)$\\
TS(n=2000,d=10)& $1000$& $754$& $753$& $ 0.516 (\pm  0.018)$& $ 0.195 (\pm  0.013)$& $ 0.558 (\pm  0.022)$& $ 0.131 (\pm  0.011)$\\

\hline 
\end{tabular}}
\end{center}
\end{table}

\begin{table}[!ht]
\caption{Comparison of the min-max truncated estimator $\hg$
with the ordinary least squares estimator $\hfols$ for the mixture noise (see Section \ref{sec:noise})
with $\rho=0.4$ and $p=0.005$. In parenthesis, the $95\%$-confidence intervals for the estimated quantities.} \label{tab:b04}
{
\begin{center}\scalebox{0.7}{\begin{tabular}{l|c|c|c|c|c|c|c}
\hline 
& \rotatebox{90}{nb of iterations } & \rotatebox{90}{nb of iter. with $R(\hg)\neq R(\hfols)$} & 
\rotatebox{90}{nb of iter. with $R(\hg)< R(\hfols)$} & \rotatebox{90}{$\E R(\hfols)-R(\gl)$}
& \rotatebox{90}{$\E R(\hg)-R(\gl)$} & \rotatebox{90}{$\E R[(\hfols)|\hg\neq\hfols]-R(\gl)$} 
& \rotatebox{90}{$\E [R(\hg)|\hg\neq\hfols]-R(\gl)$}\\
\hline 
INC(n=200,d=1)& $1000$& $234$& $211$& $ 0.551 (\pm  0.063)$& $ 0.409 (\pm  0.042)$& $ 1.211 (\pm  0.210)$& $ 0.606 (\pm  0.110)$\\
INC(n=200,d=2)& $1000$& $195$& $186$& $ 1.046 (\pm  0.088)$& $ 0.788 (\pm  0.061)$& $ 2.174 (\pm  0.293)$& $ 0.848 (\pm  0.118)$\\
HCC(n=200,d=2)& $1000$& $222$& $215$& $ 1.028 (\pm  0.079)$& $ 0.748 (\pm  0.051)$& $ 2.157 (\pm  0.243)$& $ 0.897 (\pm  0.112)$\\
TS(n=200,d=2)& $1000$& $291$& $268$& $ 1.053 (\pm  0.079)$& $ 0.805 (\pm  0.058)$& $ 1.701 (\pm  0.186)$& $ 0.851 (\pm  0.093)$\\
INC(n=1000,d=2)& $1000$& $127$& $117$& $ 0.201 (\pm  0.013)$& $ 0.181 (\pm  0.012)$& $ 0.366 (\pm  0.053)$& $ 0.207 (\pm  0.035)$\\
INC(n=1000,d=10)& $1000$& $262$& $249$& $ 1.023 (\pm  0.035)$& $ 0.902 (\pm  0.030)$& $ 1.238 (\pm  0.081)$& $ 0.777 (\pm  0.054)$\\
HCC(n=1000,d=10)& $1000$& $201$& $192$& $ 0.991 (\pm  0.033)$& $ 0.902 (\pm  0.031)$& $ 1.235 (\pm  0.088)$& $ 0.790 (\pm  0.067)$\\
TS(n=1000,d=10)& $1000$& $171$& $162$& $ 1.009 (\pm  0.033)$& $ 0.951 (\pm  0.031)$& $ 1.166 (\pm  0.098)$& $ 0.825 (\pm  0.071)$\\
INC(n=2000,d=2)& $1000$& $80$& $77$& $ 0.105 (\pm  0.007)$& $ 0.099 (\pm  0.006)$& $ 0.214 (\pm  0.042)$& $ 0.135 (\pm  0.029)$\\
HCC(n=2000,d=2)& $1000$& $44$& $42$& $ 0.102 (\pm  0.007)$& $ 0.099 (\pm  0.007)$& $ 0.187 (\pm  0.050)$& $ 0.120 (\pm  0.034)$\\
TS(n=2000,d=2)& $1000$& $47$& $47$& $ 0.101 (\pm  0.007)$& $ 0.099 (\pm  0.007)$& $ 0.147 (\pm  0.032)$& $ 0.103 (\pm  0.026)$\\
INC(n=2000,d=10)& $1000$& $116$& $113$& $ 0.511 (\pm  0.016)$& $ 0.491 (\pm  0.016)$& $ 0.611 (\pm  0.052)$& $ 0.437 (\pm  0.042)$\\
HCC(n=2000,d=10)& $1000$& $110$& $105$& $ 0.500 (\pm  0.016)$& $ 0.481 (\pm  0.015)$& $ 0.602 (\pm  0.056)$& $ 0.430 (\pm  0.044)$\\
TS(n=2000,d=10)& $1000$& $101$& $98$& $ 0.511 (\pm  0.016)$& $ 0.499 (\pm  0.016)$& $ 0.601 (\pm  0.054)$& $ 0.486 (\pm  0.051)$\\

\hline 
\end{tabular}}
\end{center}}
\end{table}

\begin{table}[!ht]
\caption{Comparison of the min-max truncated estimator $\hg$
with the ordinary least squares estimator $\hfols$ with the heavy-tailed noise (see Section \ref{sec:noise}).}
\label{tab:a201} {
\begin{center}\scalebox{0.75}{\begin{tabular}{l|c|c|c|c|c|c|c}
\hline 
& \rotatebox{90}{nb of iterations } & \rotatebox{90}{nb of iter. with $R(\hg)\neq R(\hfols)$} & 
\rotatebox{90}{nb of iter. with $R(\hg)< R(\hfols)$} & \rotatebox{90}{$\E R(\hfols)-R(\gl)$}
& \rotatebox{90}{$\E R(\hg)-R(\gl)$} & \rotatebox{90}{$\E R[(\hfols)|\hg\neq\hfols]-R(\gl)$} 
& \rotatebox{90}{$\E [R(\hg)|\hg\neq\hfols]-R(\gl)$}\\
\hline 
INC(n=200,d=1)& $1000$& $163$& $145$&  $ 7.72   (\pm  3.46)$& $ 3.92 (\pm  0.409)$& $30.52 (\pm 20.8)$& $ 7.20 (\pm  1.61)$\\
INC(n=200,d=2)& $1000$& $104$& $98$&   $22.69    (\pm 23.14)$& $19.18 (\pm 23.09)$&  $45.36 (\pm 14.1)$& $11.63 (\pm  2.19)$\\
HCC(n=200,d=2)& $1000$& $120$& $117$&  $18.16   (\pm 12.68)$& $ 8.07 (\pm  0.718)$& $99.39 (\pm 105)$& $ 15.34 (\pm  4.41)$\\
TS(n=200,d=2)& $1000$& $110$& $105$&   $43.89    (\pm 63.79)$& $39.71 (\pm 63.76)$&  $48.55 (\pm 18.4)$& $10.59 (\pm  2.01)$\\
INC(n=1000,d=2)& $1000$& $104$& $100$& $ 3.98  (\pm  2.25)$& $ 1.78 (\pm  0.128)$& $23.18 (\pm 21.3)$& $ 2.03 (\pm  0.56)$\\
INC(n=1000,d=10)& $1000$& $253$& $242$&$16.36 (\pm  5.10)$& $ 7.90 (\pm  0.278)$& $41.25 (\pm 19.8)$& $ 7.81 (\pm  0.69)$\\
HCC(n=1000,d=10)& $1000$& $220$& $211$&$13.57 (\pm  1.93)$& $ 7.88 (\pm  0.255)$& $33.13 (\pm  8.2)$& $ 7.28 (\pm  0.59)$\\
TS(n=1000,d=10)& $1000$& $214$& $211$& $18.67  (\pm 11.62)$& $13.79 (\pm 11.52)$&  $30.34 (\pm  7.2)$& $ 7.53 (\pm  0.58)$\\
INC(n=2000,d=2)& $1000$& $113$& $103$& $ 1.56  (\pm  0.41)$& $ 0.89 (\pm  0.059)$& $ 6.74 (\pm  3.4)$& $ 0.86 (\pm  0.18)$\\
HCC(n=2000,d=2)& $1000$& $105$& $97$&  $ 1.66   (\pm  0.43)$& $ 0.95 (\pm  0.062)$& $ 7.87 (\pm  3.8)$& $ 1.13 (\pm  0.23)$\\
TS(n=2000,d=2)& $1000$& $101$& $95$&   $ 1.59    (\pm  0.64)$& $ 0.88 (\pm  0.058)$& $ 8.03 (\pm  6.2)$& $ 1.04 (\pm  0.22)$\\
INC(n=2000,d=10)& $1000$& $259$& $255$&$ 8.77 (\pm  4.02)$& $ 4.23 (\pm  0.154)$& $21.54 (\pm 15.4)$& $ 4.03 (\pm  0.39)$\\
HCC(n=2000,d=10)& $1000$& $250$& $242$&$ 6.98 (\pm  1.17)$& $ 4.13 (\pm  0.127)$& $15.35 (\pm  4.5)$& $ 3.94 (\pm  0.25)$\\
TS(n=2000,d=10)& $1000$& $238$& $233$& $ 8.49  (\pm  3.61)$& $ 5.95 (\pm  3.486)$& $14.82 (\pm  3.8)$& $ 4.17 (\pm  0.30)$\\

\hline 
\end{tabular}}
\end{center}}
\end{table}

\begin{table}[!ht]
\caption{Comparison of the min-max truncated estimator $\hg$
with the ordinary least squares estimator $\hfols$ with the asymmetric heavy-tailed noise (see Section \ref{sec:noise}).}
\label{tab:a-201} {
\begin{center}
\scalebox{0.72}{\begin{tabular}{l|c|c|c|c|c|c|c}
\hline 
& \rotatebox{90}{nb of iterations } & \rotatebox{90}{nb of iter. with $R(\hg)\neq R(\hfols)$} & 
\rotatebox{90}{nb of iter. with $R(\hg)< R(\hfols)$} & \rotatebox{90}{$\E R(\hfols)-R(\gl)$}
& \rotatebox{90}{$\E R(\hg)-R(\gl)$} & \rotatebox{90}{$\E R[(\hfols)|\hg\neq\hfols]-R(\gl)$} 
& \rotatebox{90}{$\E [R(\hg)|\hg\neq\hfols]-R(\gl)$}\\
\hline 
INC(n=200,d=1)& $1000$& $87$& $77$&    $ 5.49 (\pm  3.07)$& $ 3.00 (\pm  0.330)$& $35.44 (\pm 34.7)$&  $ 6.85 (\pm  2.48)$\\
INC(n=200,d=2)& $1000$& $70$& $66$&    $19.25 (\pm 23.23)$& $17.4  (\pm 23.2)$&   $37.95 (\pm 13.1)$&  $11.05 (\pm  2.87)$\\
HCC(n=200,d=2)& $1000$& $67$& $66$&    $ 7.19 (\pm  0.88)$& $ 5.81 (\pm  0.397)$& $31.52 (\pm 10.5)$&  $10.87 (\pm  2.64)$\\
TS(n=200,d=2)& $1000$& $76$& $68$&     $39.80 (\pm 64.09)$& $37.9  (\pm 64.1)$&   $34.28 (\pm 14.8)$&  $ 9.21 (\pm  2.05)$\\
INC(n=1000,d=2)& $1000$& $101$& $92$&  $ 2.81 (\pm  2.21)$& $ 1.31 (\pm  0.106)$& $16.76 (\pm 21.8)$&  $ 1.88 (\pm  0.69)$\\
INC(n=1000,d=10)& $1000$& $211$& $195$&$10.71 (\pm  4.53)$& $ 5.86 (\pm  0.222)$& $29.00 (\pm 21.3)$&  $ 6.03 (\pm  0.71)$\\
HCC(n=1000,d=10)& $1000$& $197$& $185$&$ 8.67 (\pm  1.16)$& $ 5.81 (\pm  0.177)$& $20.31 (\pm  5.59)$& $ 5.79 (\pm  0.43)$\\
TS(n=1000,d=10)& $1000$& $258$& $233$& $13.62 (\pm 11.27)$& $11.3  (\pm 11.2)$&   $14.68 (\pm  2.45)$& $ 5.60 (\pm  0.36)$\\
INC(n=2000,d=2)& $1000$& $106$& $92$&  $ 1.04 (\pm  0.37)$& $ 0.64 (\pm  0.042)$& $ 4.54 (\pm  3.45)$& $ 0.79 (\pm  0.16)$\\
HCC(n=2000,d=2)& $1000$& $99$& $90$&   $ 0.90 (\pm  0.11)$& $ 0.66 (\pm  0.042)$& $ 3.23 (\pm  0.93)$& $ 0.82 (\pm  0.16)$\\
TS(n=2000,d=2)& $1000$& $84$& $81$&    $ 1.11 (\pm  0.66)$& $ 0.60 (\pm  0.042)$& $ 6.80 (\pm  7.79)$& $ 0.69 (\pm  0.17)$\\
INC(n=2000,d=10)& $1000$& $238$& $222$&$ 6.32 (\pm  4.18)$& $ 3.07 (\pm  0.147)$& $16.84 (\pm 17.5)$&  $ 3.18 (\pm  0.51)$\\
HCC(n=2000,d=10)& $1000$& $221$& $203$&$ 4.49 (\pm  0.98)$& $ 2.98 (\pm  0.091)$& $ 9.76 (\pm  4.39)$& $ 2.93 (\pm  0.22)$\\
TS(n=2000,d=10)& $1000$& $412$& $350$& $ 5.93 (\pm  3.51)$& $ 4.59 (\pm  3.44)$&  $ 6.07 (\pm  1.76)$& $ 2.84 (\pm  0.16)$\\

\hline 
\end{tabular}}
\end{center}}
\end{table}

\begin{table}[!ht]
\caption{Comparison of the min-max truncated estimator $\hg$
with the ordinary least squares estimator $\hfols$ for standard Gaussian noise.}
\label{tab:a0} {
\begin{center}\scalebox{0.74}{\begin{tabular}{l|c|c|c|c|c|c|c}
\hline 
& \rotatebox{90}{nb of iter. } & \rotatebox{90}{nb of iter. with $R(\hg)\neq R(\hfols)$} & 
\rotatebox{90}{nb of iter. with $R(\hg)< R(\hfols)$} & \rotatebox{90}{$\E R(\hfols)-R(\gl)$}
& \rotatebox{90}{$\E R(\hg)-R(\gl)$} & \rotatebox{90}{$\E R[(\hfols)|\hg\neq\hfols]-R(\gl)$} 
& \rotatebox{90}{$\E [R(\hg)|\hg\neq\hfols]-R(\gl)$}\\
\hline 
INC(n=200,d=1)& $1000$& $20$& $8$& $ 0.541 (\pm  0.048)$& $ 0.541 (\pm  0.048)$& $ 0.401 (\pm  0.168)$& $ 0.397 (\pm  0.167)$\\
INC(n=200,d=2)& $1000$& $1$& $0$& $ 1.051 (\pm  0.067)$& $ 1.051 (\pm  0.067)$& $ 2.566 $& $ 2.757 $\\
HCC(n=200,d=2)& $1000$& $1$& $0$& $ 1.051 (\pm  0.067)$& $ 1.051 (\pm  0.067)$& $ 2.566 $& $ 2.757 $\\
TS(n=200,d=2)& $1000$& $0$& $0$& $ 1.068 (\pm  0.067)$& $ 1.068 (\pm  0.067)$& --& --\\
INC(n=1000,d=2)& $1000$& $0$& $0$& $ 0.203 (\pm  0.013)$& $ 0.203 (\pm  0.013)$& --& --\\
INC(n=1000,d=10)& $1000$& $0$& $0$& $ 1.023 (\pm  0.029)$& $ 1.023 (\pm  0.029)$& --& --\\
HCC(n=1000,d=10)& $1000$& $0$& $0$& $ 1.023 (\pm  0.029)$& $ 1.023 (\pm  0.029)$& --& --\\
TS(n=1000,d=10)& $1000$& $0$& $0$& $ 0.997 (\pm  0.028)$& $ 0.997 (\pm  0.028)$& --& --\\
INC(n=2000,d=2)& $1000$& $0$& $0$& $ 0.112 (\pm  0.007)$& $ 0.112 (\pm  0.007)$& --& --\\
HCC(n=2000,d=2)& $1000$& $0$& $0$& $ 0.112 (\pm  0.007)$& $ 0.112 (\pm  0.007)$& --& --\\
TS(n=2000,d=2)& $1000$& $0$& $0$& $ 0.098 (\pm  0.006)$& $ 0.098 (\pm  0.006)$& --& --\\
INC(n=2000,d=10)& $1000$& $0$& $0$& $ 0.517 (\pm  0.015)$& $ 0.517 (\pm  0.015)$& --& --\\
HCC(n=2000,d=10)& $1000$& $0$& $0$& $ 0.517 (\pm  0.015)$& $ 0.517 (\pm  0.015)$& --& --\\
TS(n=2000,d=10)& $1000$& $0$& $0$& $ 0.501 (\pm  0.015)$& $ 0.501 (\pm  0.015)$& --& --\\

\hline 
\end{tabular}}
\end{center}}
\end{table}\pagebreak

\begin{figure}[!ht] 
\caption{Surrounding points are the points of the training set generated several times from $TS(1000,10)$ (with the mixture noise with $p=0.005$ and $\rho=0.4$) that are not taken into account in the min-max truncated estimator (to the extent that the estimator would not change by removing simultaneously all these points). 
The min-max truncated estimator $x\mapsto\hf(x)$ appears in dash-dot line, while $x\mapsto\E(Y|X=x)$ is in solid line. In these six simulations, 
it outperforms the ordinary least squares estimator.} \label{fig:1}
\begin{center}
   \vspace*{-.2cm}\begin{minipage}{.49\linewidth}
      \hspace*{-0.8cm}\includegraphics[width=1.15\linewidth]{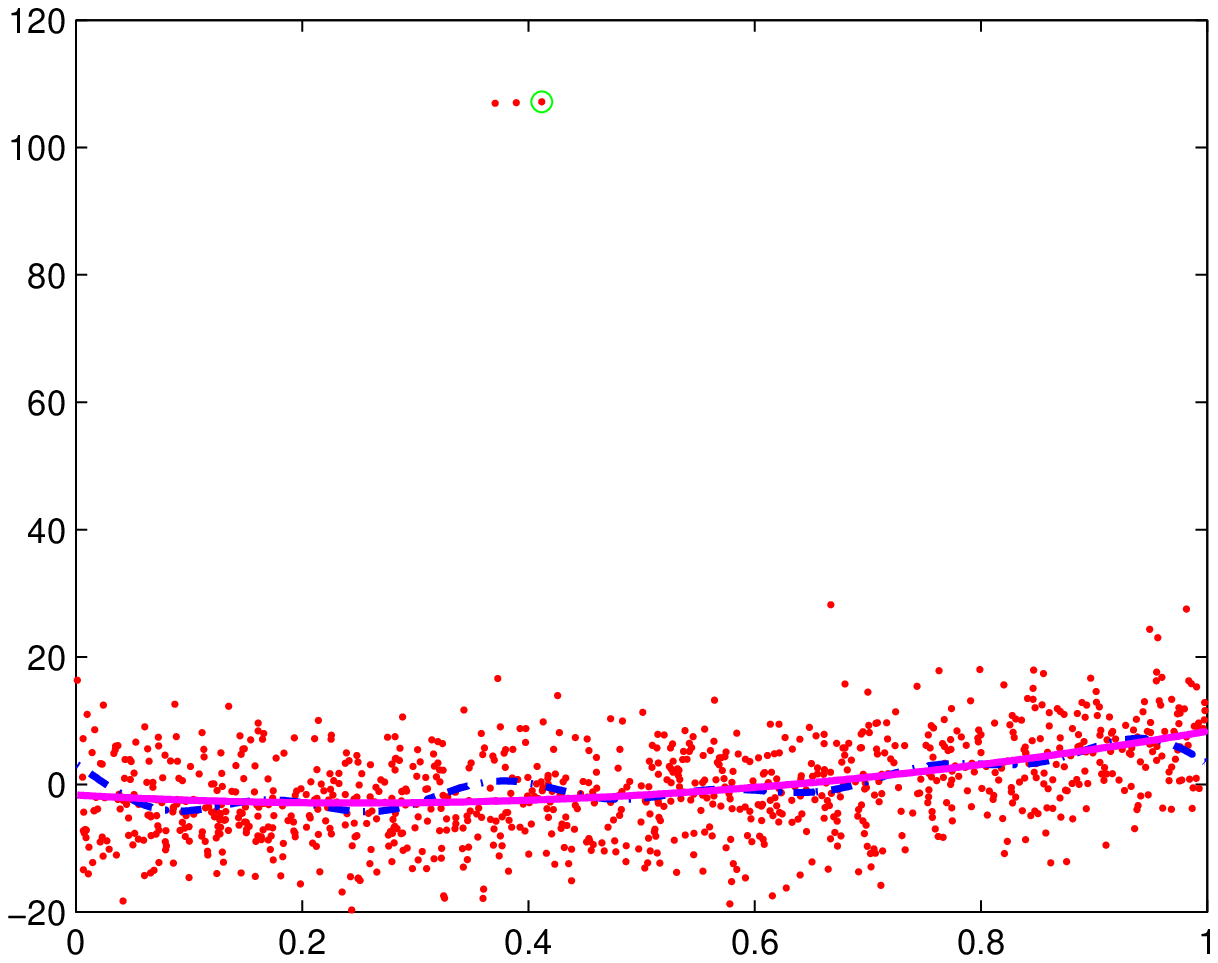}
   \end{minipage} 
   \hfill
   \begin{minipage}{0.49\linewidth}
      \hspace*{-0.4cm}\includegraphics[width=1.15\linewidth]{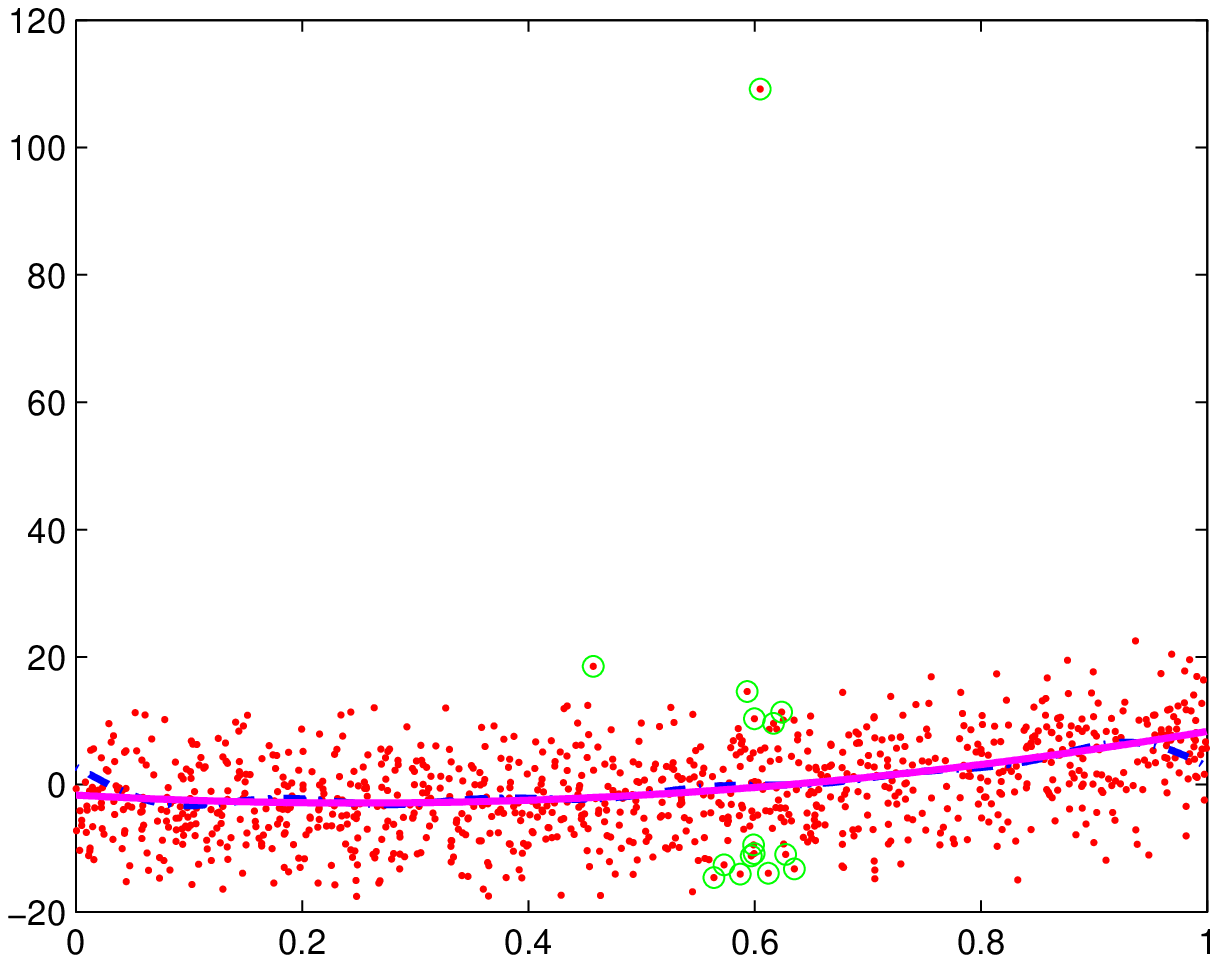}
   \end{minipage}
   \vspace*{-.2cm}\begin{minipage}{.49\linewidth}
      \hspace*{-0.8cm}\includegraphics[width=1.15\linewidth]{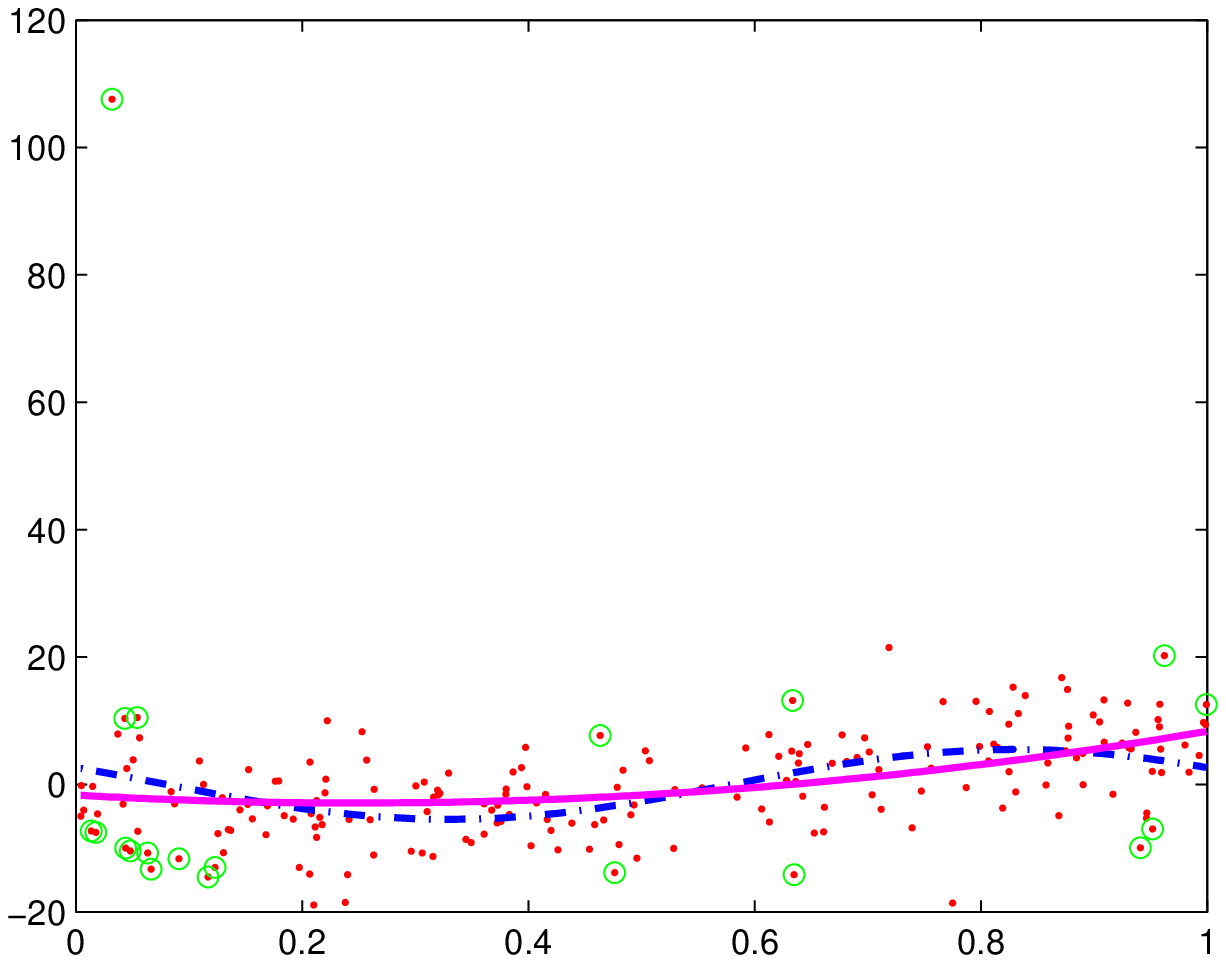}
   \end{minipage} 
   \hfill
   \begin{minipage}{0.49\linewidth}
      \hspace*{-0.4cm}\includegraphics[width=1.15\linewidth]{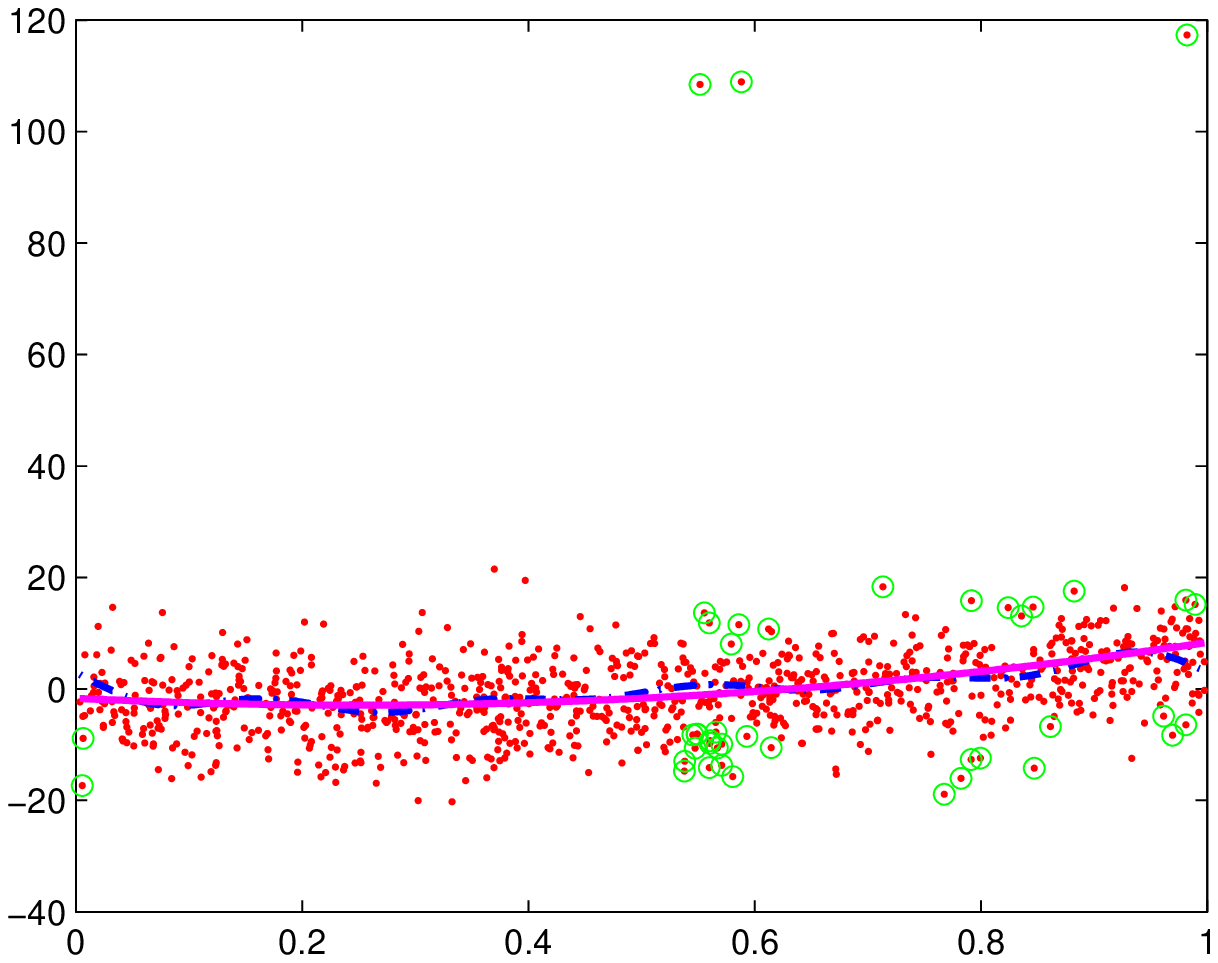}
   \end{minipage}
   \vspace*{-.2cm}\begin{minipage}{.49\linewidth}
      \hspace*{-0.8cm}\includegraphics[width=1.15\linewidth]{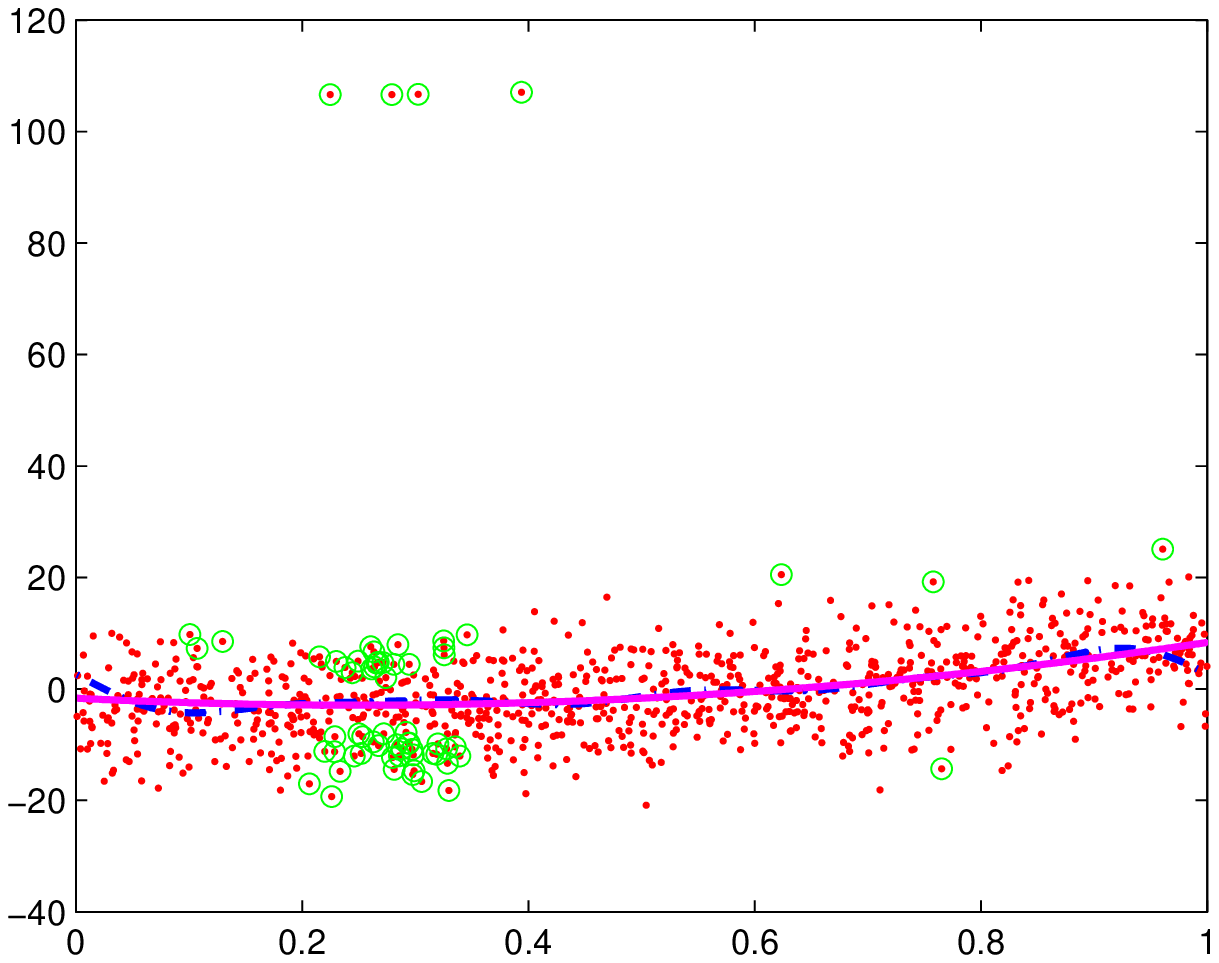}
   \end{minipage} 
   \hfill
   \begin{minipage}{0.49\linewidth}
      \hspace*{-0.4cm}\includegraphics[width=1.15\linewidth]{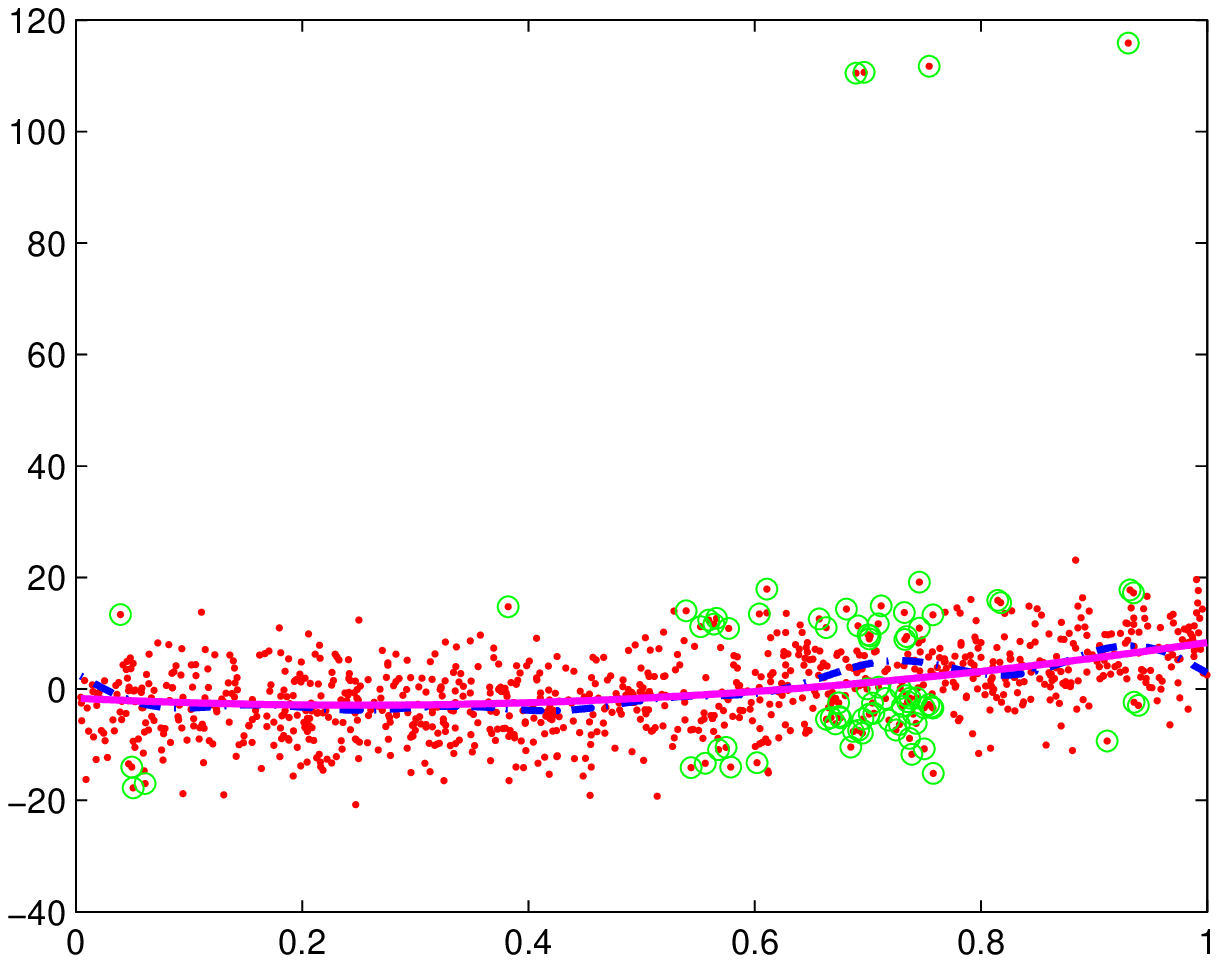}
   \end{minipage}
\end{center}
\end{figure}

\begin{figure}[!ht] 
\caption{Surrounding points are the points of the training set generated several times from $TS(200,2)$ (with the heavy-tailed noise) that are not taken into account in the min-max truncated estimator (to the extent that the estimator would not change by removing these points). 
The min-max truncated estimator $x\mapsto\hf(x)$ appears in dash-dot line, while $x\mapsto\E(Y|X=x)$ is in solid line. In these six simulations, it outperforms the ordinary least squares estimator. Note that in the last figure, it does not consider $64$ points among the $200$ training points.} \label{fig:2}
\begin{center}
   \vspace*{-.2cm}\begin{minipage}{.49\linewidth}
      \hspace*{-0.8cm}\includegraphics[width=1.15\linewidth]{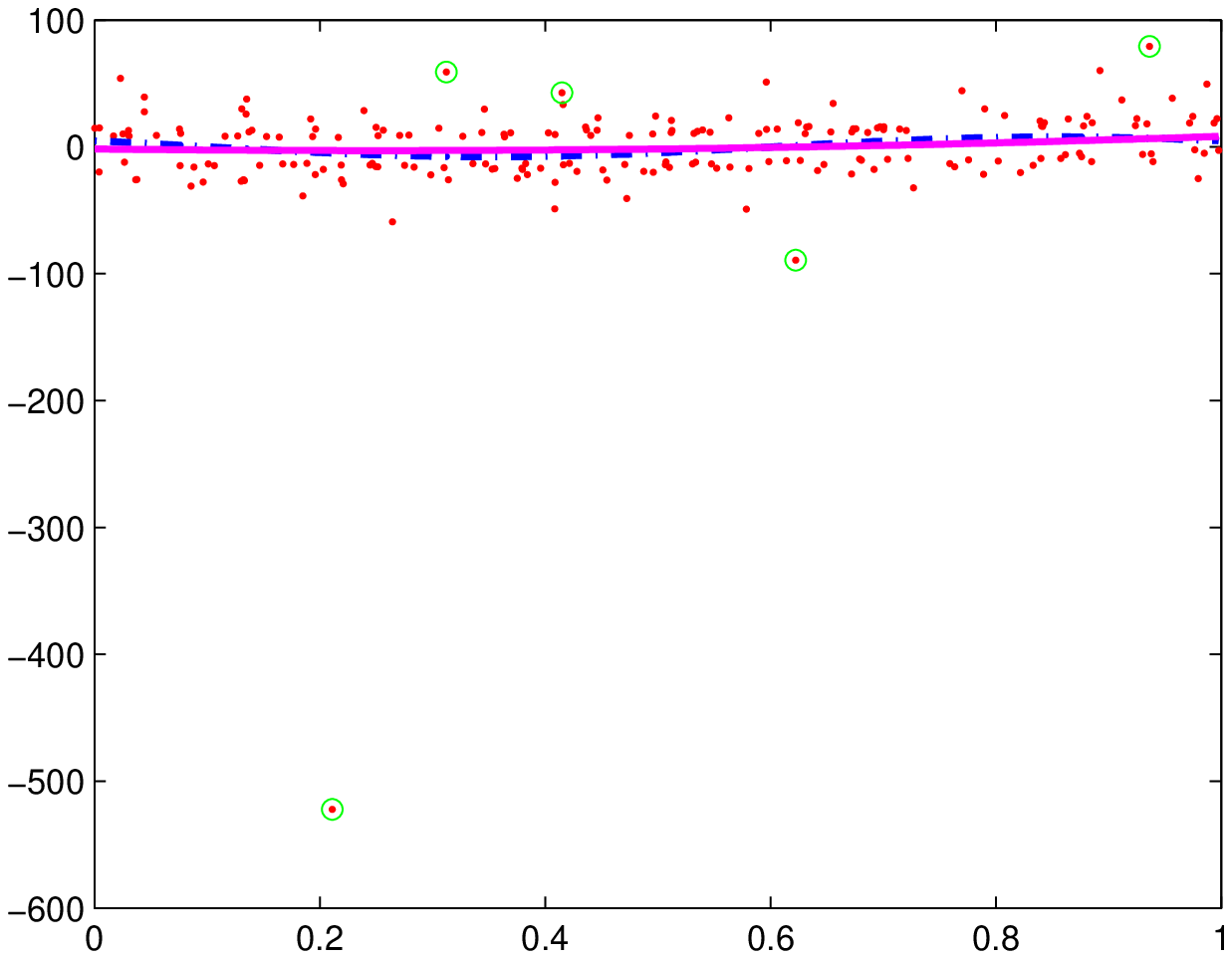}
   \end{minipage} 
   \hfill
   \begin{minipage}{0.49\linewidth}
      \hspace*{-0.4cm}\includegraphics[width=1.15\linewidth]{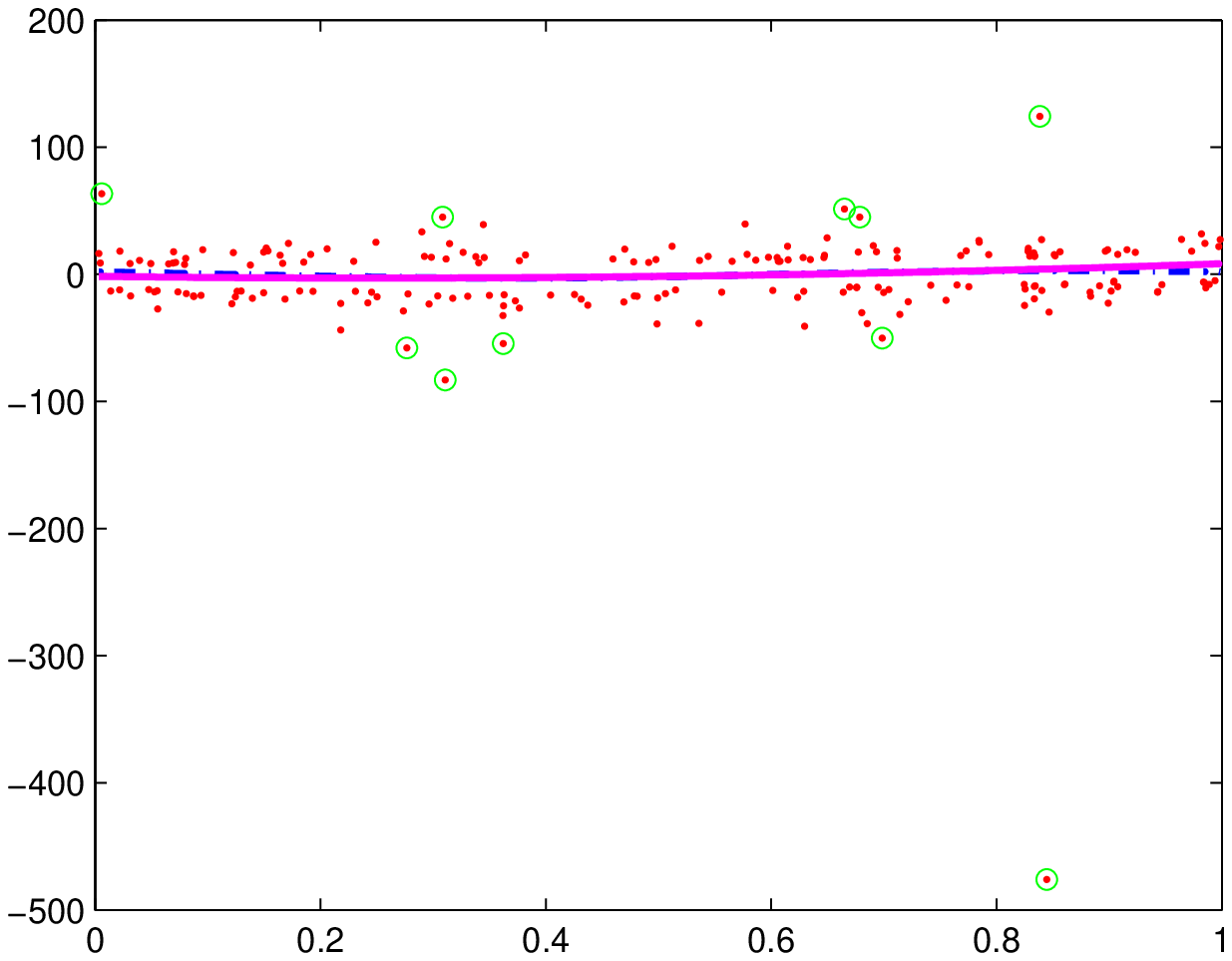}
   \end{minipage}
   \vspace*{-.2cm}\begin{minipage}{.49\linewidth}
      \hspace*{-0.8cm}\includegraphics[width=1.15\linewidth]{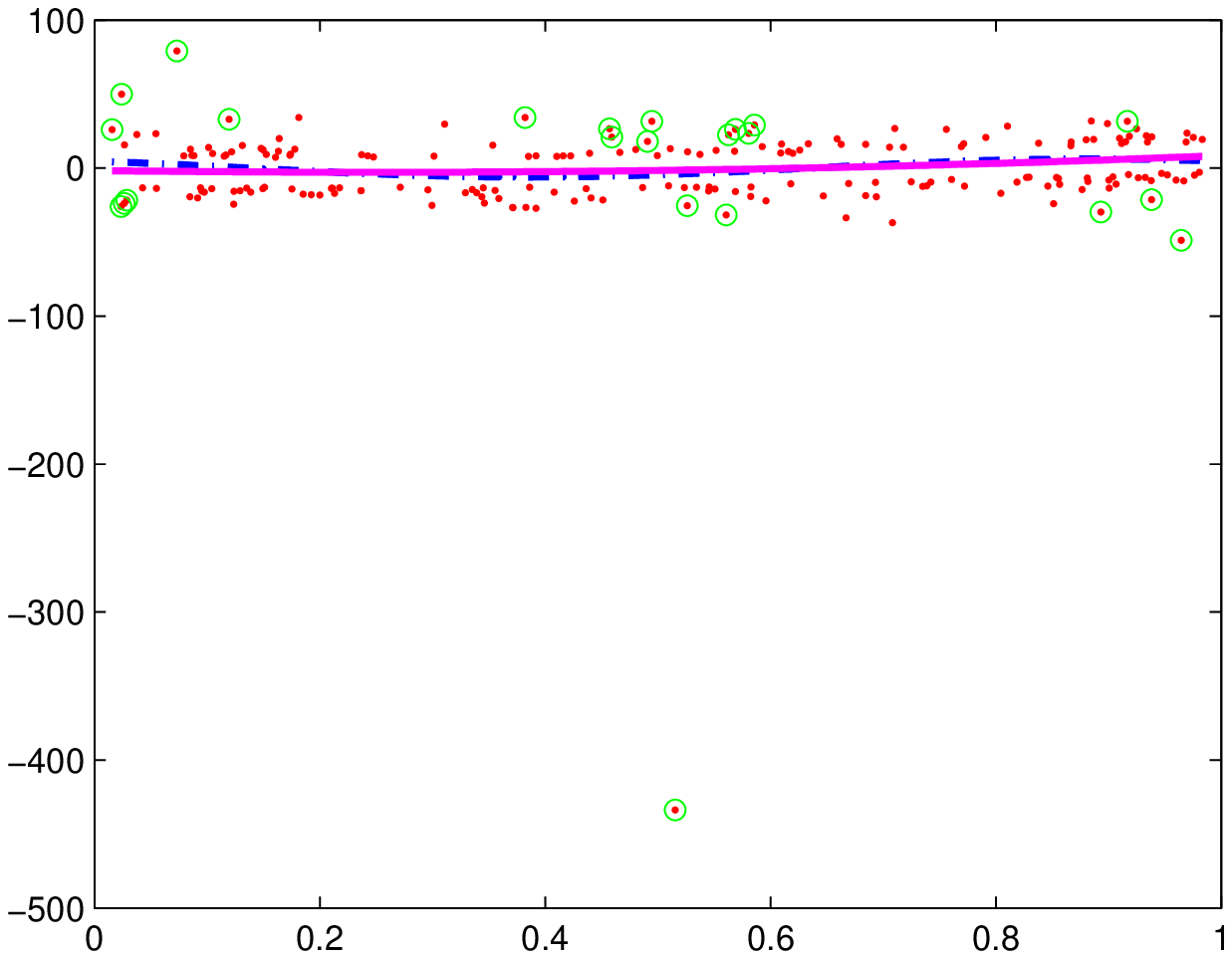}
   \end{minipage} 
   \hfill
   \begin{minipage}{0.49\linewidth}
      \hspace*{-0.4cm}\includegraphics[width=1.15\linewidth]{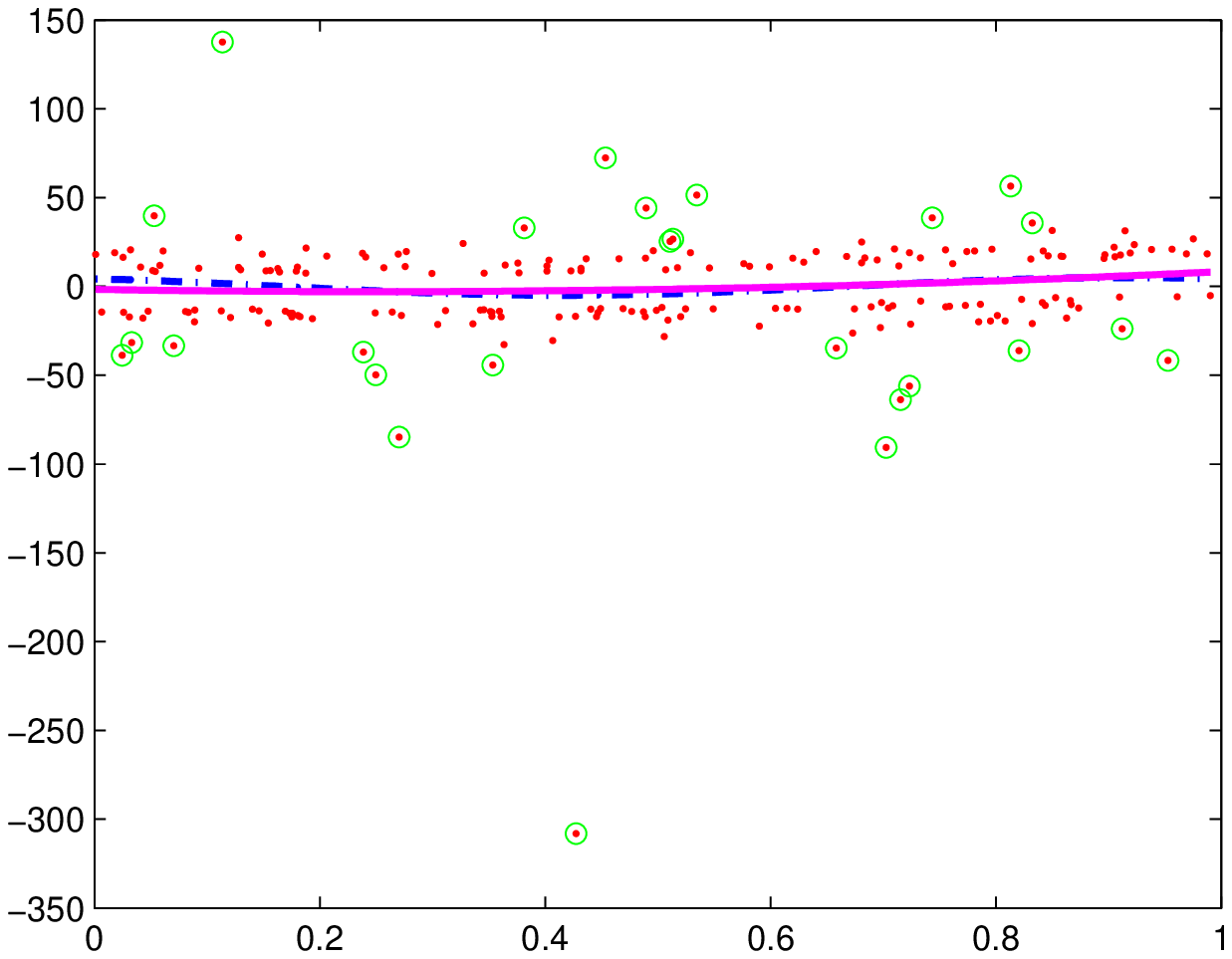}
   \end{minipage}
   \vspace*{-.2cm}\begin{minipage}{.49\linewidth}
      \hspace*{-0.8cm}\includegraphics[width=1.15\linewidth]{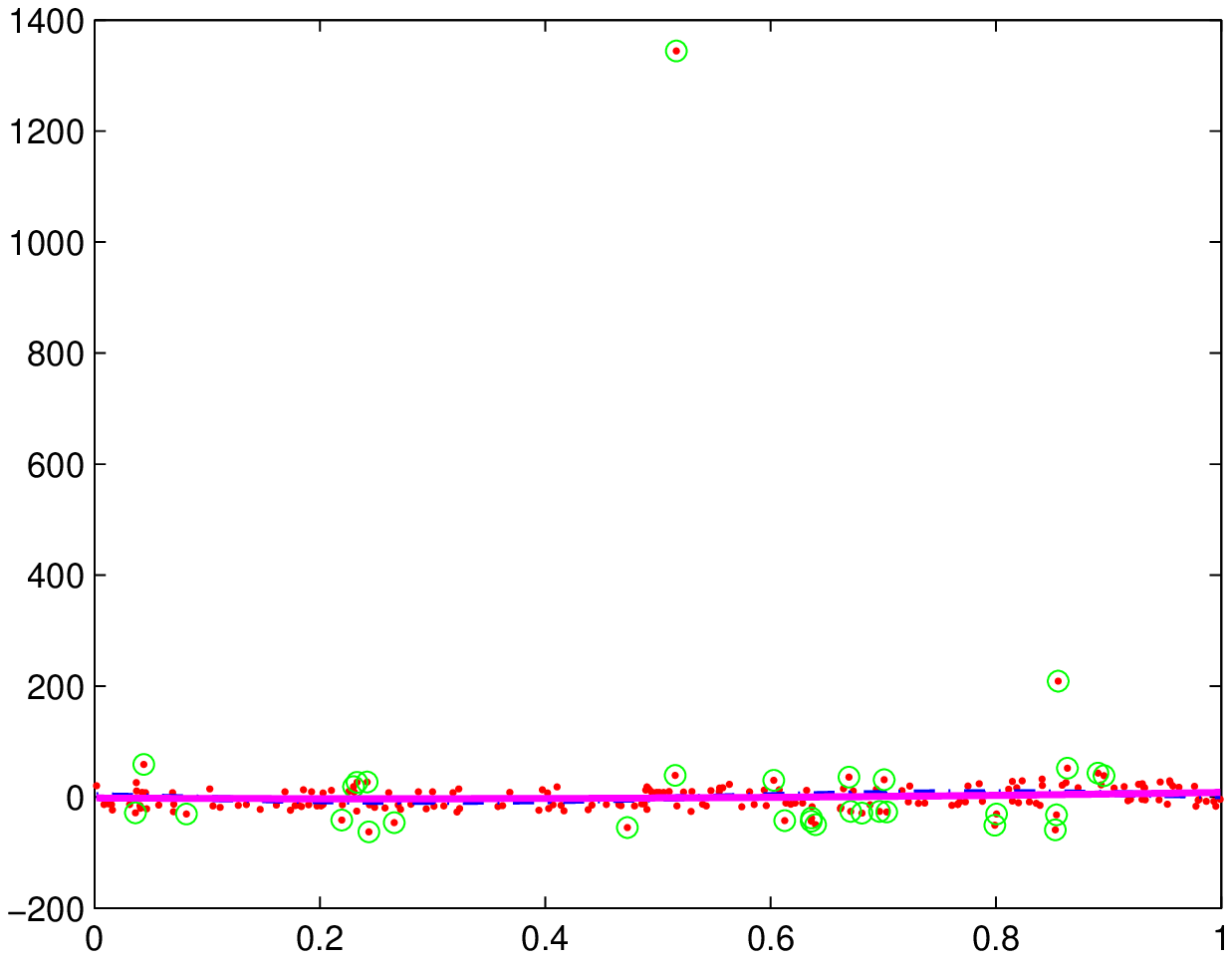}
   \end{minipage} 
   \hfill
   \begin{minipage}{0.49\linewidth}
      \hspace*{-0.4cm}\includegraphics[width=1.15\linewidth]{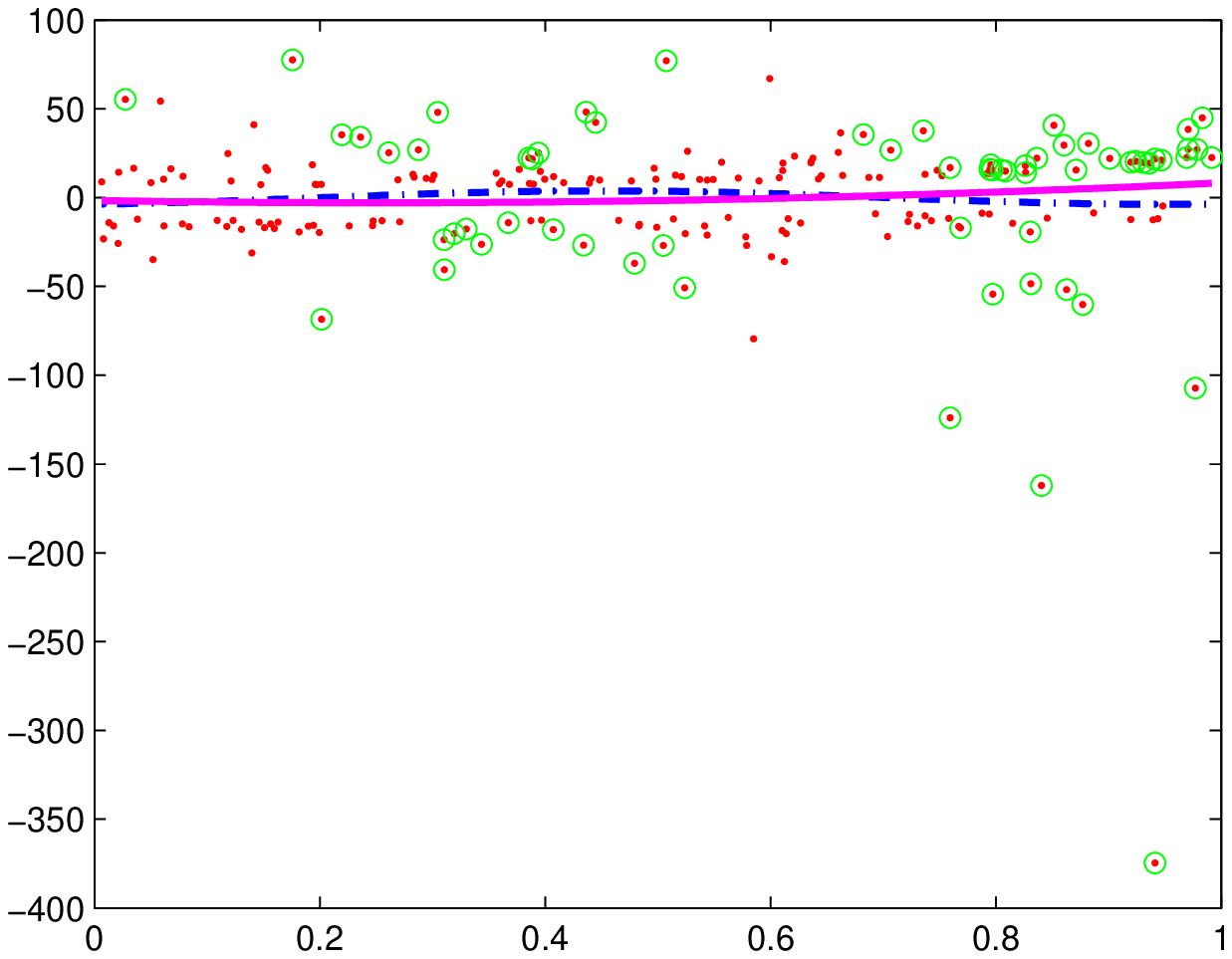}
   \end{minipage}
\end{center}
\end{figure}

\bibliographystyle{plain}
\bibliography{ref.bib}
\end{document}